\documentclass[a4paper]{amsart}

\usepackage{subfiles}
\usepackage{amssymb,pstricks,amscd,epsfig}
\usepackage[utf8]{inputenc}
\usepackage{graphicx}
\usepackage{a4wide}

\usepackage{pgf,tikz}
\usetikzlibrary{arrows}
\usepackage{pgfplots}

\usepackage{hyperref}

\newcounter{notes}%
\newcommand{\marginnote}[1]{
\refstepcounter{notes}  
\nolinebreak
$\hspace{-5pt}{}^{\text{\tiny \rm \arabic{notes}}}$
\marginpar{\tiny \arabic{notes}) #1}
}

\newtheorem{cor}{Corollary}[section]
\newtheorem{theorem}[cor]{Theorem}
\newtheorem{prop}[cor]{Proposition}
\newtheorem{lemma}[cor]{Lemma}

\theoremstyle{definition}
\newtheorem{defi}[cor]{Definition}
\theoremstyle{remark}
\newtheorem{remark}[cor]{Remark}
\newtheorem{example}[cor]{Example}

\newcommand{\cH}{{\mathcal H}}

\newcommand{\C}{{\mathbb C}}
\newcommand{\HH}{{\mathbb H}}

\newcommand{\R}{{\mathbb R}}
\newcommand{\Z}{{\mathbb Z}}

\newcommand{\bH}{\mathbb H}
\newcommand{\bS}{\mathrm S}

\newcommand{\Ad}{\mathrm{Ad}}

\newcommand{\Id}{\mathrm{Id}}

\newcommand{\dev}{\mbox{dev}}

\newcommand{\sym}{\textit{sym}}

\newcommand{\cotan}{\mbox{cotan}}

\newcommand{\Isom}{\mathrm{Isom}}

\newcommand{\Conv}{\mathrm{Conv}}

\newcommand{\SO}{\mathrm{SO}}

\renewcommand{\O}{\mathrm{O}}
\newcommand{\so}{\mathfrak{so}}
\newcommand{\PGL}{\mathrm{PGL}}
\newcommand{\GL}{\mathrm{GL}}
\newcommand{\Sp}{\mathrm{Sp}}

\newcommand{\U}{\mathrm U}

\newcommand{\AdS}{\mathrm{AdS}}
\newcommand{\dS}{\mathrm{dS}}
\newcommand{\Ein}{\mathrm{Ein}}

\newcommand{\End}{\mathrm{End}}

\newcommand{\Proj}{\mathbf P}

\newcommand{\Span}{\mathrm{Span}}
\newcommand{\Hom}{\mathrm{Hom}}
\newcommand{\Grad}{\mathrm{Grad}}
\newcommand{\secondFF}{\mathrm{II}}

\renewcommand{\tilde}{\widetilde}

\newcommand{\equaldef}{\overset{\mathrm{def}}{=}}

\newcommand{\set}[2]{\left\lbrace #1 \,\middle|\, #2 \right\rbrace}

\title[Gromov-Thurston manifolds]{Gromov-Thurston manifolds and anti-de Sitter geometry}

\author{Daniel Monclair}
\address{Daniel Monclair: Université Paris-Saclay, Laboratoire de Mathématiques d'Orsay, 91405 Orsay, France
}
\email{daniel.monclair@universite-paris-saclay.fr}

\author{Jean-Marc Schlenker}
\address{Jean-Marc Schlenker:
University of Luxembourg, FSTM, Department of Mathematics,
Maison du nombre, 6 avenue de la Fonte,
L-4364 Esch-sur-Alzette, Luxembourg}
\email{jean-marc.schlenker@uni.lu}
\thanks{J.-M. S. was partially supported by FNR project O20/14766753.}

\author{Nicolas Tholozan}
\address{Nicolas Tholozan: CNRS, \'ENS-PSL, 45 rue d'Ulm, 75005 Paris, France}
\email{nicolas.tholozan@ens.fr}

\date{v1, \today}

\begin{document}
\begin{abstract}
  We consider hyperbolic and anti-de Sitter (AdS) structures on $M\times (0,1)$, where $M$ is a $d$-dimensional Gromov--Thurston manifold. If $M$ has cone angles greater than $2\pi$, we show that there exists a ``quasifuchsian'' (globally hyperbolic maximal) AdS manifold such that the future boundary of the convex core is isometric to $M$. When $M$ has cone angles less than $2\pi$, there exists a hyperbolic end with boundary a concave pleated surface isometric to $M$.

  Moreover, in both cases, if $M$ is a Gromov--Thurston manifold with $2k$ pieces (as defined below), the moduli space of quasifuchsian AdS structures (resp. hyperbolic ends) satisfying this condition contains a submanifold of dimension $2k-3$.

  When $d=3$, the moduli space of quasifuchsian AdS (resp. hyperbolic) manifolds diffeomorphic to $M\times (0,1)$ contains a submanifold of dimension $2k-2$, and extends up to a ``Fuchsian'' manifold, that is, an AdS (resp. hyperbolic) warped product of a closed hyperbolic manifold by~$\R$.

  We use this construction of quasifuchsian AdS manifolds to obtain new compact quotients of $\O(2d,2)/\U(d,1)$. The construction uses an explicit correspondence between quasifuchsian $2d+1$-dimensional AdS manifolds and compact quotients of $\O(2d,2)/\U(d,1)$ which we interpret as the space of timelike geodesic Killing fields of $\AdS^{2d+1}$.
  
\end{abstract}

\maketitle

\tableofcontents

\section{Introduction and main results}

Gromov and Thurston constructed in \cite{gromov-thurston} families of closed manifolds of dimension at least~$4$ which carry negatively curved Riemannian metrics but do not admit any locally homogeneous metric.

Roughly speaking, these manifolds are obtained by taking ramified covers and quotients of certain closed hyperbolic manifolds which admit a dihedral group of symmetries generated by two reflections along 
totally geodesic hypersurfaces, see Section \ref{sc:2}. In particular, they carry a natural hyperbolic metric with a cone singularity of rational angle along a totally geodesic submanifold of codimension 2. Using arguments based on Mostow's rigidity in codimension 1, Gromov and Thurston prove that they cannot carry a smooth hyperbolic metric. However, their very geometric origin suggests that one might endow them with geometric structures of a ``weaker'' type.

Indeed, Kapovich proved in \cite{kapovich:gromov-thurston} that Gromov--Thurston manifolds with cone singularity of angle less than $2\pi$ carry a convex projective structure, namely, that they are quotients of a convex open subset of a projective space by a discrete group of transformations. In particular, their fundamental group admits quasi-isometric embeddings in a real linear group.

In a similar spirit, we will show here that, if $M$ is a Gromov--Thurston manifold with cone singularity larger than $2\pi$, then $M\times \R$ carries a \emph{globally hyperbolic maximal Cauchy compact anti-de Sitter structure}, later abreviated in GHMC AdS structure. Following a recent trend, we will also use the term ``quasifuchsian AdS manifold'', which brings to mind the analogy between those AdS manifolds and quasifuchsian hyperbolic manifolds (see \cite{mess,mess-notes}). This provides exotic examples of such manifolds and answers negatively to Questions 5.1 and 5.2 of the survey \cite{adsquestions}. Note that counter-examples to these questions were also constructed by Lee--Marquis \cite{lee-marquis} in dimension up to $8+1$ using reflection groups.

By the work of Guéritaud--Guichard--Kassel--Wienhard \cite{GGKW17}, our construction of exotic AdS quasifuchsian groups in dimension $2d+1$ also provides examples of exotic compact quotients of the homogeneous spaces $\O(2d,2)/\U(d,1)$. These are, to our knowledge, the first examples of discrete groups acting properly discontinuously and cocompactly on a homogeneous space of reductive type which are not isomorphic to a uniform lattice in some other Lie group. We will describe an explicit geometric relation between these two objects in Section \ref{sc:clifford}.

In dimension $3$, though Gromov--Thurston's construction still makes sense, their manifolds also carry a smooth hyperbolic metric, according (for instance) to Perelman's geometrization 
theorem. In that case, using Hodgson--Kerckhoff's results on conical hyperbolic metrics in dimension $3$, we will construct a family of GHMC AdS structures on $M\times \R$ that interpolates between the ``Fuchsian structure'' and our general construction. We also show that this family ``integrates'' linear combinations of infinitesimal ``bending'' deformations of the representation $i:\pi_1(M) \to \SO(3,1)$ within $\SO(3,2)$.

\subsection{AdS structures associated to Gromov--Thurston manifolds} \label{ss:IntroAdSStructures}



Recall that a Lorentzian manifold is called \emph{globally hyperbolic Cauchy compact} if it admits a compact \emph{Cauchy hypersurface}, i.e. a topological 
hypersurface intersecting any inextendible timelike curve at a single point. It is further called \emph{maximal} (abreviated in GHMC) if it is maximal for the inclusion among such spaces.

If $N$ is a GHMC \emph{anti-de Sitter manifold} (i.e. a GHMC Lorentzian manifold of constant sectional curvature $-1$) of dimension $d+1$, then $N$ is the quotient of a convex domain of the \emph{anti-de Sitter space} $\AdS^{d+1}$ by a discrete subgroup $\Gamma$ of $\Isom(\AdS^{d+1}) \simeq \O(d,2)$ (see \cite{mess,barbot_causal}). If furthermore $N$ admits a \emph{convex} Cauchy hypersurface, then the group $\Gamma$ is Gromov-hyperbolic, its embedding into $\O(d,2) \subset \GL(d+2,\R)$ has a refined discreteness property called \emph{$P_1$-Anosov} (see Theorem \ref{thm-characterization AdS convex}) and, by a theorem of Barbot \cite{barbot_deformations}, any continuous deformation of the inclusion in $\Hom(\Gamma, \O(d,2))$ is again the holonomy of a GHMC AdS manifold homeomorphic to $N$. In that case, we will call $N$ a \emph{quasifuchsian AdS manifold}. We give more details on AdS geometry in Section \ref{sc:globally}.

We recall in Section \ref{sc:2} the construction of Gromov--Thurston cone-manifolds. The main point that we will use here is that those are cone-manifolds of dimension $d$ (for $d\geq 3$), obtained by gluing $2k$ isometric ``pieces'' along a manifold of dimension $d-2$. Each piece is a hyperbolic manifold with ``corner'', whose boundary is composed of two totally geodesic hypersurfaces meeting along a manifold of codimension $2$ with an interior dihedral angle of $\pi/n$. Gromov--Thurston manifolds thus carry a hyperbolic metric with a cone singularity along a totally geodesic submanifold of codimension $2$, and the angle around this cone singularity can be smaller or greater than $2\pi$, depending on whether $k<n$ or $k>n$. 

The main result of this paper is the following:

\begin{theorem} \label{tm:existence_ads}
Let $(M,g)$ be a $d$-dimensional Gromov--Thurston cone-manifold with cone angle larger than  $2\pi$ at the singularity, $d\geq 3$. Then there is a quasifuchsian AdS spacetime $N$ of dimension $d+1$ for which the future boundary of the convex core is isometric to $(M,g)$.
\end{theorem}

There is some flexibility in our construction which allows to construct a non-trivial moduli space of such quasifuchsian AdS spacetimes.

\begin{theorem} \label{tm:dim_ads}
Let $M$ be a $d$-dimensional Gromov--Thurston cone-manifold with $2k$ pieces, with cone angle larger than $2\pi$, $d\geq 3$. Then there is a $2k-3$ parameter family of quasifuchsian AdS manifolds for which the future boundary of the convex core is isometric to $M$.
\end{theorem}

Theorem \ref{tm:existence_ads} is partly motivated by Questions 5.1 and 5.2 of the survey \cite{adsquestions}. When this survey was written, the only known examples of GHMC AdS manifolds were either deformations of \emph{Fuchsian} AdS manifolds -- those admitting a totally geodesic Cauchy hypersurface -- or quotients of an open convex domain of $\AdS^{d+1}$ by a uniform lattice in $\O(p,1)\times \O(q,1)\subset \O(d,2)$, $p+q=d$. In particular, such manifolds are always homeomorphic to $M\times \R$ with $M$ a compact quotient of $\HH^p\times \HH^q$ by a uniform lattice. Question 5.1 of the survey \cite{adsquestions} asked whether these are all the possible topologies, while Question 5.2 asked whether every GHMC manifold could be deformed to one of these standard ones. In the same direction, Barbot--Mérigot \cite[Question 8.7]{barbot-merigot} asked whether any AdS quasifuchsian manifold is homeomorphic to the product of a hyperbolic manifold with $\R$. Note that the answer to these questions is known to be positive in dimension $2+1$ by the work of Mess (see \cite{mess}).

Question 5.1 (and thus Question 5.2) was answered negatively by Lee and Marquis \cite{lee-marquis} in dimension $4+1$ to $8+1$: they constructed Coxeter reflection groups which are not hyperbolic lattices but admit AdS quasifuchsian representations. However, as often with reflection groups, these can only exist up to a certain dimension. In contrast, Theorem \ref{tm:existence_ads} provides a negative answer to Question 5.1 in every dimension $d+1 \geq 4+1$. Indeed, Gromov--Thurston manifolds of dimension $d\geq 4$ are not diffeomorphic to quotients of $\HH^p \times \HH^q$. In fact, we have the stronger result:

\begin{theorem}[Gromov--Thurston]
The fundamental group of a Gromov--Thurston manifold $M$ of dimension $d\geq 4$ is not commensurable to a lattice in any Lie group.
\end{theorem}

\begin{remark} \label{rem-GTnotLattice}
One easily sees that $\pi_1(M)$ cannot be a lattice in a Lie group with non-trivial solvable radical. Since $\pi_1(M)$ surjects onto a uniform hyperbolic lattice, Margulis superrigidity implies that $\pi_1(M)$ is not a lattice in a higher rank semisimple Lie group either.

With arguments involving Mostow's rigidity, Gromov and Thurston prove that it is not a lattice in $\Isom(\HH^d)$. They also construct a Riemannian metric on $M$ with sectional curvature pinched between $-1-\epsilon$ and $-1$, which cannot exist on quotients of other rank $1$ symmetric spaces by a result of Yau--Zheng \cite{pinched_kahler}. 
To construct such a metric, one needs the additional assumption that the ``injectivity radius'' of the singular locus is sufficiently large (so that one has enough room to smoothen the singular hyperbolic metric).

However, Giralt proved in her thesis \cite{giralt} that $\pi_1(M)$ is always cubulable and virtually special. By a theorem of Delzant--Py \cite{DelzantPy}, it is thus not isomorphic to a complex hyperbolic lattice, without any additional geometric assumption. Another consequence is that $\pi_1(M)$ has the Haagerup property (see \cite{measured_walls}). This rules out the possibility that $\pi_1(M)$ be a lattice in the remaining rank $1$ simple Lie groups $\Sp(n,1)$ and $\mathrm F_4^{-20}$, which have Kazhdan's property (T).
\end{remark}

As another consequence of Theorem \ref{tm:existence_ads}, one obtains the existence of nice linear representations of fundamental groups of Gromov--Thurston cone-manifolds with cone angle larger than $2\pi$. 

\begin{cor}
Let $M$ be a Gromov--Thurston cone-manifold of dimension $d$ with cone angle larger than $2\pi$. Then $\pi_1(M)$ admits a quasi-isometric embedding into $\O(d,2)$. In particular, $\pi_1(M)$ is linear.
\end{cor}

\begin{remark}
Kapovich, on the other hand, constructed \emph{convex projective structures} on Gromov--Thurston cone-manifolds with cone angle less than $2\pi$ \cite{kapovich:gromov-thurston}. Combining his result with ours, we get that fundamental groups of $d$-dimensional Gromov--Thurston manifolds embed quasi-isometrically in $\PGL(d+2,\R)$ without any angle condition.
\end{remark}

\begin{remark}
In fact, both our representations and those of Kapovich satisfy a stronger form of quasi-isometric property called \emph{$P_1$-Anosov property}. We refer to \cite{KLP, DGK:convexProjective, GGKW17} for more details on $P_1$-Anosov representations.
\end{remark}

\begin{remark}
Giralt's cubulation theorem combined with the work of 
Haglund--Wise \cite{haglund_wise}  implies that fundamental groups of Gromov--Thurston manifolds virtually embed into right-angled Artin groups and are thus linear. These arguments, however, give little control on the dimension of a faithful linear representation.
\end{remark}

Yet another consequence of Theorem \ref{tm:existence_ads} is that a quasifuchsian $\AdS$ manifold of dimension $d+1$ does not always contain a Cauchy hypersurface whose geometry is intrinsically locally isometric to the hyperbolic space $\HH^d$ when $d\geq 3$. The fact that this is true for $d=2$ was crucial in the proof of the rigidity theorem in \cite{glorieux_monclair}, stating that the \emph{limit set} of a quasifuchsian $\AdS$ manifold of dimension $2+1$ has \emph{Lorentzian Hausdorff dimension} smaller than $1$, with equality only in the Fuchsian case. Such a statement is believed to be true in higher dimension, but Theorem \ref{tm:existence_ads} confirms that the techniques used in \cite{glorieux_monclair} cannot be used in this case. 

\subsection{Exotic compact Clifford--Klein forms}

A \emph{compact Clifford--Klein form} of a homogeneous space $G/H$ is a quotient of $G/H$ by a discrete subgroup $\Gamma \subset G$ acting properly discontinuously and cocompactly on $G$.

Guéritaud--Guichard--Kassel--Wienhard remarked in \cite{GGKW17} that AdS quasifuchsian subgroups $\Gamma$ of $\O(2d,2)$ act properly discontinuously and cocompactly on the pseudo-Riemannian symmetric space $\mathrm O(2d,2)/\U(d,1)$. Hence, we obtain as a direct consequence of Theorem \ref{tm:existence_ads}:

\begin{cor} \label{cor - CliffordKleinForm}
Let $M$ be a Gromov--Thurston cone-manifold of dimension $2d$ with cone angle larger than $2\pi$. Then there exists a faithful representation $\rho: \pi_1(M)\to \mathrm O(2d,2)$ such that $\rho(\pi_1(M))$ acts properly discontinuously and cocompactly on $\mathrm O(2d,2)/\U(d,1)$.
\end{cor}

Very few pseudo-Riemannian symmetric spaces $G/H$ are known to admit compact quotients that are \emph{non-standard}, i.e. where $\Gamma$ is not commensurable to a lattice in a connected subgroup of $G$. The space $\O(2d,2)/\U(d,1)$ is one of them. So far, however, these non-standard quotients were obtained as deformations of the standard ones (corresponding to AdS Fuchsian manifolds). Together with the work of Lee--Marquis \cite{lee-marquis}, Corollary \ref{cor - CliffordKleinForm} thus provides the first \emph{exotic} examples of compact Clifford--Klein forms of $\O(2d,2)/\U(d,1)$ (i.e. which are not deformations of standard ones). In fact, to our knowledge, these are the first examples of a compact Clifford--Klein form $\Gamma \backslash G/H$ for which $\Gamma$ is not virtually isomorphic to a lattice in some Lie group.

Guéritaud--Guichard--Kassel--Wienhard's argument to associate compact Clifford--Klein forms to AdS quasifuchsian manifolds is rather indirect. In Section \ref{sc:initial}, we provide a direct, geometric explanation of this correspondence, which also allows for a more precise analysis of the Clifford-Klein forms which are obtained. We start by interpreting $\mathrm O(2d,2)/\U(d,1)$ as the space of timelike unit geodesic Killing vector fields in $\AdS^{2d+1}$. The correspondence follows from the fact that given a smooth, complete strictly convex Cauchy hypersurface $\mathcal H$ in a quasifuchsian AdS manifold of dimension $2d+1$, any unit timelike geodesic Killing vector field on $\AdS^{2d+1}$ is orthogonal to the lift to $\AdS^{2d+1}$ of $\mathcal H$ at a unique point. This gives a natural projection of the Clifford--Klein form to the Cauchy hypersurface.

\begin{theorem}
Let $N$ be an AdS quasifuchsian manifold of dimension $2d+1$ with fundamental group $\Gamma\subset \mathrm O(2d,2)$ and $\mathcal H$ a Cauchy hypersurface. Then there exists a smooth fibration $\pi:\Gamma\backslash \mathrm O(2d,2)/\U(d,1)\to \mathcal H$ whose fibers are translates of the compact homogeneous subspace $\mathrm O(2d)/\U(d)$.
\end{theorem}

This Theorem confirms, for the case of $\mathrm O(2d,2)/\U(d,1)$ a general conjecture formulated by the third author in \cite[Section 8]{volume_non_existence}.

\subsection{Hyperbolic ends associated to Gromov--Thurston manifolds}


Our anti-de Sitter geometrization of Gromov--Thurston manifolds with cone angle larger than $2\pi$ has a hyperbolic counterpart when the cone angle is smaller than $2\pi$. In that case, one can realize a Gromov--Thurston manifold $M$ of dimension $d$ as the boundary of a \emph{hyperbolic end} of dimension $d+1$.



\begin{defi}
A \emph{hyperbolic end} of dimension $d+1$ is a manifold with boundary of the form $M\times [0,+\infty)$ with $M$ closed of dimension $d$ equipped with a hyperbolic metric, such that a neighbourhood of $M\times \{0\}$ is developped to the exterior of a convex hypersurface in $\HH^{d+1}$, and which is maximal (in the sense of inclusion) under this condition.
\end{defi}

A simple example of a hyperbolic end is provided by the closure a connected component of the complement of the convex core in a quasifuchsian hyperbolic 3-dimensional manifold. 

The following theorems are the hyperbolic counterparts of Theorems \ref{tm:existence_ads} and \ref{tm:dim_ads}

\begin{theorem} \label{tm:existence_hyp}
  Let $M$ 
  be a $d$-dimensional Gromov--Thurston cone-manifold with cone angle smaller than $2\pi$ at the singularity. Then there exists a hyperbolic end $N$ of dimension $d+1$ for which the boundary is isometric to $M$.
\end{theorem}



Moreover, if $k\geq 2$, then there is a non-trivial moduli space of deformations of the hyperbolic ends realizing $M$ as their concave pleated boundary. 

\begin{theorem} \label{tm:dim_hyp}
Let $M$ be a Gromov--Thurston cone-manifold with $2k$ pieces, with cone angles smaller than $2\pi$. Then there is a $2k-3$ parameter family of hyperbolic ends for which the concave pleated boundary is isometric to $M$.
\end{theorem}

\begin{remark}
The construction of hyperbolic ends associated to Gromov--Thurston manifolds is already suggested in the initial paper of Gromov--Thurston and was a starting point for Kapovich's investigation of the geometry of these manifolds. Though it might be considered folklore knowledge, its details do not seem to appear in the litterature, hence our decision to include them here.
\end{remark}

A hyperbolic end $N= M\times [0,+\infty)$ admits a conformal compactification obtained by adding a ``boundary at infinity'' $\partial_\infty N = M\times \{+\infty\}$. This boundary admits an atlas with charts in $\partial_\infty\HH^{d+1}\simeq \mathbb S^d$ and transitions maps in $\O(d+1,1)\simeq \mathrm{M\ddot ob}(\mathbb S^d)$, providing $M$ with a conformally flat structure. We thus have the following:

\begin{cor}
  Let $M$ be a $d$-dimensional Gromov--Thurston cone-manifold with cone angle smaller than $2\pi$. Then $M$ admits a conformally flat metric.
\end{cor}

\begin{remark}
One can show that different hyperbolic ends yield different conformal structures (see Theorem \ref{tm:kp} and \cite{kulkarni-pinkall}). Hence we also have a $2k-3$-dimensional moduli space of flat conformal structures on $M$.
\end{remark}

Finally, each flat conformal structure on $M$ is also the conformal boundary at infinity of a unique maximal globally hyperbolic \emph{de Sitter structure} on $M\times \R$, which is in some sense ``dual'' to the hyperbolic end. Hence we get:

\begin{cor}
Let $M$ be a $d$-dimensional Gromov--Thurston cone-manifold with cone angle smaller than $2\pi$. Then there exists a GHMC de Sitter spacetime diffeomorphic to $M\times \R$. 
\end{cor}

We refer to Scannell's thesis \cite{scannell} for more details on the correspondence between hyperbolic ends, conformally flat manifolds and GHMC de Sitter spacetimes.


An important difference between the AdS and hyperbolic settings is the following: for GHMC AdS manifolds $M\times \R$, the fact that the universal cover $\tilde{M}$ is developed to a spacelike hypersurface forces this development to be an embedding and the holonomy $\rho:\pi_1(M) \to \O(d,2)$ to be discrete and faithful. In contrast, if $M\times [0,+\infty)$ is a hyperbolic end, the development of $\tilde M$ need not be an embedding and the holonomy representation is not in general discrete and faithful. Gromov and Thurston remark in their paper that, for cone singularities sufficiently close to $2\pi$ and assuming that the singular locus of $M$ has a sufficiently large injectivity radius, one could construct hyperbolic ends for which $\tilde M$ is quasi-isometrically embedded in $\HH^{d+1}$ and the holonomy $\rho:\pi_1(M) \to \O(d+1,1)$ is convex-cocompact. We will not prove this result which is beyond the scope of this paper.

\subsection{Deformations in dimension $3+1$}

In this section we consider the case $d=3$. Contrary to higher dimension, Gromov--Thurston manifolds of dimension $3$ do carry smooth hyperbolic structures. Thanks to the rich deformation theory for hyperbolic cone-manifolds in this dimension~\cite{HK}, we can provide more precise results and show that our AdS spacetimes with Gromov--Thurston Cauchy hypersurfaces (as defined in Section 2) can be deformed continuously to Fuchsian AdS spacetimes. 



\begin{theorem} \label{tm:31_ads}
  Let $M$ be a 3-dimensional Gromov--Thurston cone-manifold with $2k$ pieces, with cone angle larger than $2\pi$. Then there is a connected $2k-2$ parameter family of quasifuchsian AdS spacetimes diffeomorphic to $M\times (0,1)$ containing a $2k-3$-dimensional family of spacetimes with future boundary of the convex core isometric to $M$ and a point corresponding to a Fuchsian AdS spacetime.
\end{theorem}



\begin{theorem} \label{tm:31_hyp}
Let $M$ be a 3-dimensional Gromov--Thurston cone-manifold with $2k$ pieces, with cone angle smaller than $2\pi$. Then there is a connected $2k-2$ parameter family of hyperbolic ends diffeomorphic to $M\times (0,\infty)$ containing a $2k-3$-dimensional family of hyperbolic ends with pleated boundary isometric to $M$ and a point corresponding to a Fuchsian end.
\end{theorem}



\subsection{Integrating bending deformations}

In dimension $d=3$, Theorems \ref{tm:31_ads} and \ref{tm:31_hyp}  can be interpreted in terms of integration of infinitesimal deformations of the holonomy representation of a hyperbolic $3$-dimensional manifold in $\O(3,2)$ and $\O(4,1)$ respectively. 

When a $d$-dimensional hyperbolic manifold $M$ contains a 2-sided totally geodesic hypersurface, then its holonomy representation can be deformed into larger Lie groups such as $\O(d+1,1)$, $\O(d,2)$ or  $\GL(d+1,\R)$, via some generalized ``bending'' (see for instance \cite{Johnson1987}). It was already noted in \cite{Johnson1987} that if $M$ contains $r$ such disjoint hypersurfaces, then the deformation space has dimension at least~$r$. On the other hand, when the hypersurfaces intersect, for $d\geq 3$, the deformation space might be singular, and some infinitesimal deformations given by sums of  infinitesimal bending deformations along two intersecting hypersurfaces may not be integrated into actual deformations.

Theorems \ref{tm:31_ads} and \ref{tm:31_hyp} show that, in dimension $d=3$, a significant degree of flexibility exists to deform Fuchsian representations of Gromov--Thurston manifolds.

\begin{theorem} \label{tm:bending3d-ads}
  Let $M$ be a 3-dimensional Gromov--Thurston manifold with $2k$ pieces and no cone singularity (that is, total angle $2\pi$ at the gluing curve). Then there exists a neighbourhood $U$ of $0$ in $\R^k$ such that, for every $(\theta_1, \cdots, \theta_k)\in U$, there exists a quasifuchsian AdS $4$-manifold $N_\theta$ with a Cauchy hypersurface which is a Gromov--Thurston manifold homeomorphic to $M$, consisting of $2k$ totally geodesic pieces pleated at angles $t\theta_1, t\theta_2, \cdots, t\theta_k, t\theta_1, \cdots, t\theta_k$ (in cyclic order around the singular curve). 
\end{theorem}

\begin{theorem} \label{tm:bending3d-hyp}
  Let $M$ be a 3-dimensional Gromov--Thurston manifold, with $2k$ pieces and no cone singularity (that is, total angle $2\pi$ at the gluing curve). Then there exists a neighbourhood $U$ of $0$ in $\R^k$ such that, for every $(\theta_1, \cdots, \theta_k)\in U$, there exists a quasifuchsian hyperbolic $4$-manifold $N_\theta$ containing a hypersurface which is a Gromov--Thurston manifold homeomorphic to $M$, consisting of $2k$ totally geodesic pieces pleated at angles $t\theta_1, t\theta_2, \cdots, t\theta_k, t\theta_1, \cdots, t\theta_k$ (in cyclic order around the singular curve). 
\end{theorem}

The holonomy of $N_\theta$ is a representation $\rho_\theta: \pi_1(M) \to \SO_\circ(3,2)$ or $\SO_\circ(4,1)$. When all but one of the $\theta_i$ vanish, the representation $\rho_\theta$ corresponds to Johnson--Millson ``bending deformation'' of $\rho_0: \pi_1(M) \to \SO_\circ(3,1)$. Theorem \ref{tm:bending3d-ads} gives a geometric construction of representations which combine several bendings along intersecting hypersurfaces. In particular, it shows that linear combinations of infinitesimal bendings can be integrated (see Section \ref{ss:IntegrationBending}). Those deformations should be compared to the stamping deformations defined by Apanasov \cite{apanasov:deformations} (see also \cite{bart-scannell}) which seem to be closely related (in the hyperbolic setting).

\subsection{Initial singularity of GHMC spacetimes}

The results presented above, concerning the induced metrics on the future boundary of the convex cores of quasifuchsian AdS spacetimes, or on the concave boundary of hyperbolic ends, have consequences for the possible geometry of the initial singularity of GHMC AdS or dS spacetimes. 

In dimension $2+1$, the geometry of the initial singularity of an AdS or dS spacetime is rather well understood, thanks to the work of Mess \cite{mess,mess-notes}. The initial singularity is the quotient of a real tree by an action of the fundamental group of the spacetime. In some cases, the initial singularity is a finite graph (the quotient of a simplicial tree by the fundamental group of the manifold) but this is rather exceptional.

In higher dimension, however, the geometric structure of the initial singularity is much more mysterious. Here we provide examples of spacetimes for which the initial singularity is remarkably simple. We do not know to what extent this phenomenon is ``generic'', or whether ``generic'' spacetimes in dimension $d+1$, for $d\geq 3$, have a much more intricate initial singularity.

\begin{theorem} \label{tm:initial-ads}
Let $M$ be a $d$-dimensional Gromov--Thurston manifold with $2k$ pieces, with cone angles larger than  $2\pi$ at the singularities. There is a $2k-3$-dimensional family of quasifuchsian AdS spacetimes of dimension $d+1$ with Cauchy hypersurfaces diffeomorphic to $M$ for which the initial singularity is a 2-dimensional cell complex, with exactly one 2-dimensional cell. 
\end{theorem}

\begin{theorem} \label{tm:initial-ds}
Let $M$ be a $d$-dimensional Gromov--Thurston manifold with $2k$ pieces, with cone angles smaller than  $2\pi$ at the singularities. There is a $2k-3$-dimensional family of GHMC $\dS$ spacetimes of dimension $d+1$ with Cauchy hypersurfaces diffeomorphic to $M$ for which the initial singularity is a 2-dimensional cell complex, with exactly one 2-dimensional cell. 
\end{theorem}

Those two statements follow from the description of the geometry of the pleated boundary (resp. the future boundary of the convex core) for the hyperbolic ends (resp. AdS manifolds) appearing in Theorem \ref{tm:dim_hyp} and Theorem \ref{tm:dim_ads}, through the duality between hyperbolic and de Sitter space, resp. between the AdS space and itself. This correspondence is briefly recalled in Section \ref{sc:initial}, where Theorem \ref{tm:initial-ads} and Theorem \ref{tm:initial-ds} are proven.

\subsection{Outline of the paper}

In Section \ref{sc:2}, we recall the construction of Gromov--Thurston manifolds. We then recall in Section \ref{sc:globally} a number of background definitions and statements that are needed, such as the key definitions of AdS geometry, hyperbolic ends, and properties of hypersurfaces in hyperbolic and AdS manifolds. We explain in particular that the data of a quasifuchsian AdS manifold with Cauchy hypersurface homeomorphic to $M$ is equivalent to the data of a \emph{spacelike embedding structure} on $M$ i.e. an atlas of local embeddings of $M$ as spacelike hypersurfaces in $\AdS$, with coordinate changes in $\Isom(\AdS)$. When $M$ is a Gromov--Thurston manifold, one is then reduced to prescribing a way to ``bend'' the hyperbolic pieces in the anti-de Sitter space.

Such bendings are parametrized by their link along the codimension $2$ singularity, which is a spacelike polygon in the de Sitter space of dimension $2$. Section \ref{sc:polygons} focuses on the geometry of polygons in the sphere and the de Sitter plane. It starts with a characterization of the infinitesimal variations of lengths and angles of spherical and de Sitter polygons, and further describes various families of polygons (equilateral polygons, polygons with a central symmetry), which give us the material to prove our main theorems. 

In Section \ref{sec:geometrization} we prove the main results of the paper concerning the geometrization of Gromov--Thurston manifolds in dimension $d$, for $d\geq 4$, while Section \ref{sc:6} contains the proofs of the main results for Gromov--Thurston manifolds in dimension $3$. Finally, Section \ref{sc:initial} is focused on the initial singularities of de Sitter and anti-de Sitter spacetimes, and Section \ref{sc:clifford} on the applications to compact Clifford--Klein forms of $\O(2d,2)/\U(d,1)$.


%


\section{Gromov--Thurston (cone-)manifolds} 
\label{sc:2}

Here we describe a fairly general version of the Gromov--Thurston construction, providing us with a family of cone-manifolds that we will ``geometrize'', in the sense that we will show that they occur as either the future boundary of the convex core of a quasifuchsian AdS manifold, or the concave boundary of a hyperbolic end.

\subsection{Hyperbolic cone-manifolds}

We first recall the definition given by Thurston \cite[Section 3]{SOP}, \cite[Def. 3.1]{boileau-leeb-porti} of a hyperbolic cone-manifold. The definition is recursive in the dimension. We briefly recall this definition here for hyperbolic, Euclidean and spherical cone-manifolds.
\begin{itemize}
\item A one-dimensional cone-manifold is simply a one-dimensional Riemannian manifold.
\item For $d\geq 2$, a $d$-dimensional spherical (resp. hyperbolic, Euclidean) cone-manifold $M$ is a compact metric space, together with a singular metric in which every point has a neighborhood isometric to $N\times [0,\epsilon]$ equipped with the singular metric $dr^2 + \sin^2(r)h$ (resp. $dr^2 + \sinh^2(r)h$, $dr^2 + r^2h$) where $N$ is a spherical cone-manifold of dimension $d-1$ equipped with the singular metric $h$.
\end{itemize}
For instance, a $2$-dimensional hyperbolic cone-manifold -- also called hyperbolic surface with cone singularities -- contains a finite set of singular points. It is hyperbolic outside of those singular points, and each singular point has a neighborhood isometric to a ``model'' which only depends on one parameter, an ``angle'' which is the length of the $1$-dimensional manifold appearing in the definition.

\subsection{Dihedral hyperbolic manifolds}

Gromov--Thurston's construction starts with the data of a closed oriented hyperbolic manifold $M$ of dimension $d$ and two isometric involutions $\sigma_1$ and $\sigma_2$ of $M$ with the following properties:
\begin{itemize}
\item The fixed loci of $\sigma_1$ and $\sigma_2$ are connected embedded totally geodesic hypersurfaces,
\item The intersection $S= {\rm Fix }\,\sigma_1 \cap {\rm Fix }\,\sigma_2$ is connected,
\item ${\rm Fix }\,\sigma_1$ and ${\rm Fix }\,\sigma_2$ intersect along $S$ with an angle $\frac{\pi}{n}$.
\item ${\rm Fix }\,\sigma_1$ and ${\rm Fix }\,\sigma_2$ are homologically trivial.
\end{itemize}

The existence of manifolds $M$ of any dimension $d\geq 2$ with those properties is proved in \cite{gromov-thurston}. Under these conditions, $\sigma_1$ and $\sigma_2$ generate a dihedral group of isometries of $M$ of order $2n$, denoted $D_n$. We denote by $R_n$ its cyclic subgroup of order $n$, spanned by $\rho =\sigma_1 \sigma_2$. We call the data of $(M,\sigma_1,\sigma_2)$ a \emph{$n$-dihedral hyperbolic manifold}.

Let $H_1\subset {\rm Fix}\,\sigma_1$ be the closure of a connected component of ${\rm Fix}\,\sigma_1\setminus S$, and $H_2\subset {\rm Fix}\,\sigma_2$ the closure of a connected component of ${\rm Fix}\,\sigma_2\setminus S$ chosen so that the oriented angle at $S$ from $H_1$ to $H_2$ is $\frac{\pi}{n}$. We then consider the copies of $H_1$ and $H_2$ under the isometry $\rho =\sigma_1 \sigma_2$ which we denote  $H_{2i+1}=\rho^i(H_1)$ and $H_{2i}=\rho^{i-1}(H_2)$ for $i=1,\dots,n-1$. Together, they divide $M$ into $2n$ pieces $V_1,\ldots, V_{2n}$ which are fundamental domains for the action of $D_n$. Note that ${\rm Fix}\,\sigma_1 = H_1 \cup H_{n+1}$ and ${\rm Fix}\,\sigma_2 = H_2 \cup H_{n+2}$.

When considering the action of the cyclic subgroup $R_n$, a fundamental domain is given by the union of two of the former small pieces, e.g. the domain bounded by $H_1$ and $H_3$ containing $H_2$ (see Figure \ref{fig:dihedralhyperbolicmanifold}).

The quotient $\overline M = R_n\backslash M$ is a topological manifold and the quotient map $M\to \overline M$ is a ramified covering of degree $n$: it is $n$ to $1$ on the complement of $S$ and injective in restriction to $S$. We still denote by $S$ its image under the quotient map. One can show (see \cite{gromov-thurston}) that $S$ bounds two codimension $1$ submanifolds with boundary $\overline H_1, \overline H_2\subset\overline M$, which are the respective projections of $H_1$ and $H_2$.

\definecolor{qqwuqq}{rgb}{0.,0.39215686274509803,0.}
\definecolor{qqqqff}{rgb}{0.,0.,1.}
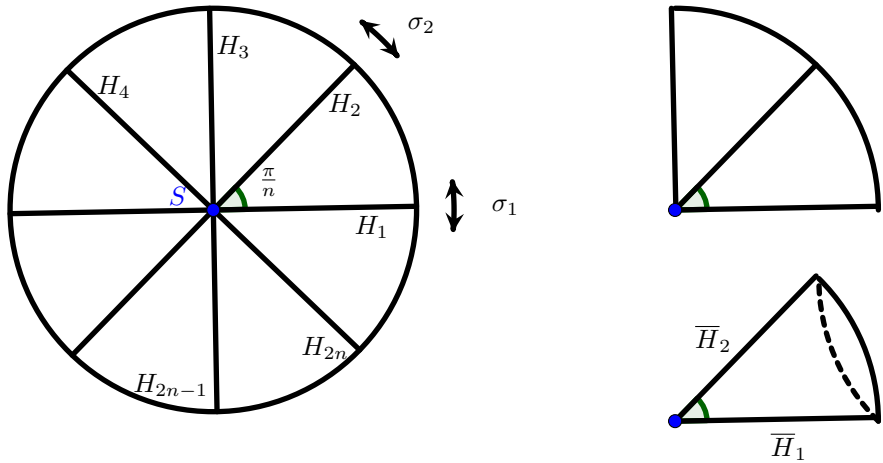
\begin{figure}[h] 
\begin{tikzpicture}[line cap=round,line join=round,>=stealth,x=1cm,y=1cm,scale=0.7]
\clip(1,-5) rectangle (26,5);
\draw [shift={(6.32,0.03)},line width=2.pt,color=qqwuqq,fill=qqwuqq,fill opacity=0.10000000149011612] (0,0) -- (1.049805166630512:0.6) arc (1.049805166630512:46.049805166630506:0.6) -- cycle; 
\draw [line width=2.pt] (2.5,-0.04)-- (10.14,0.1);
\draw [line width=2.pt] (6.32,0.03) circle (3.8206413074247103cm); 
\draw [line width=2.pt] (8.971650429449554,2.7806453788156698)-- (3.668349570550447,-2.720645378815669);
\draw[color=qqqqff] (6,-0.05) node[anchor=south east] {$S$};
\draw [fill=qqqqff] (6.32,0.03) circle (3.5pt);
\draw (7,1) node[anchor=north west] {$\frac{\pi}{n}$};
\draw (11.38,0.4) node[anchor=north west] {$\sigma_1$};
\draw (9.8,3.86) node[anchor=north west] {$\sigma_2$};
\draw (8.8,0.1) node[anchor=north west] {$H_1$};
\draw (8.3,2.4) node[anchor=north west] {$H_2$};
\draw (6.2,3.5) node[anchor=north west] {$H_3$};
\draw (3.95,2.75) node[anchor=north west] {$H_4$};
\draw (6.4,-2.9) node[anchor=north east] {$H_{2n-1}$};
\draw (7.8,-2.2) node[anchor=north west] {$H_{2n}$};
\draw[<->] [shift={(6.32,0.03)},line width=2.pt]  plot[domain=-0.09086024784723268:0.12750536117153277,variable=\t]({1.*4.525229275959396*cos(\t r)+0.*4.525229275959396*sin(\t r)},{0.*4.525229275959396*cos(\t r)+1.*4.525229275959396*sin(\t r)});
\draw[<->] [shift={(6.32,0.03)},line width=2.pt]  plot[domain=0.6924597270211887:0.9149817130980076,variable=\t]({1.*4.525229275959395*cos(\t r)+0.*4.525229275959395*sin(\t r)},{0.*4.525229275959395*cos(\t r)+1.*4.525229275959395*sin(\t r)});
\draw [line width=2.pt] (6.25,3.85)-- (6.39,-3.79);
\draw [line width=2.pt] (3.569354621184332,2.6816504294495527)-- (9.070645378815668,-2.6216504294495513);
\draw [fill=qqqqff] (6.32,0.03) circle (3.5pt);
\begin{scope}[shift={(-2,0)}]
\draw [line width=2.pt] (20.8206413074247103,0.03) arc (0:90:3.8206413074247103); 
\draw [shift={(17,0.03)},line width=2.pt,color=qqwuqq,fill=qqwuqq,fill opacity=0.10000000149011612] (0,0) -- (1.049805166630512:0.6) arc (1.049805166630512:46.049805166630506:0.6) -- cycle; 
\draw [line width=2.pt] (16.93,3.85)-- (17.015,0.03);
\draw [line width=2.pt] (17,0.03)-- (20.82,0.1);
\draw [line width=2.pt] (19.65,2.7806453788156698)-- (17,0.03);
\draw [fill=qqqqff] (17,0.03) circle (3.5pt);
\end{scope}
\begin{scope}[shift={(-2,-4)}]
\draw [line width=2.pt] (20.8206413074247103,0.03) arc (0:45:3.8206413074247103); 
\draw [line width=2.pt, dashed] (20.8206413074247103,0.03) arc (225:180:3.8206413074247103);
\draw [shift={(17,0.03)},line width=2.pt,color=qqwuqq,fill=qqwuqq,fill opacity=0.10000000149011612] (0,0) -- (1.049805166630512:0.6) arc (1.049805166630512:46.049805166630506:0.6) -- cycle; 
\draw [shift={(10,0.03)}] (8.62,-0.04) node[anchor=north west] {$\overline H_1$};
\draw [shift={(10,0.03)}] (7.2,2) node[anchor=north west] {$\overline H_2$};
\draw [line width=2.pt] (17,0.03)-- (20.82,0.1);
\draw [line width=2.pt] (19.65,2.7806453788156698)-- (17,0.03);
\draw [fill=qqqqff] (17,0.03) circle (3.5pt);
\end{scope}
\end{tikzpicture}
\caption{A $n$-dihedral manifold $M$, its fundamental piece and the quotient  $\overline M$.}
\label{fig:dihedralhyperbolicmanifold}
\end{figure}

\subsection{Gromov--Thurston manifolds}


\begin{defi}
Let $M$ be an $n$-dihedral hyperbolic manifold. For every $a\in \frac{1}{n}\mathbb N_{>0}$, we define the \emph{Gromov--Thurston} manifold $M^a$ of ramification $a$ associated to $M$ as the cyclically ramified cover of $\overline M$ along $S$ of degree $na$.
\end{defi}

More visually, $H_1, H_3, \ldots, H_{2n-1}$  cut $M$ into $n$ copies of the aforementioned fundamental piece, and $M^a$ is obtained by gluing $n a$ copies of this fundamental piece (see Figure \ref{fig:dihedralhyperbolicmanifold}).

\begin{example}
We have $M^{1/n} = \overline M$ and $M^1= M$. If $a$ is an integer, then $M^a$ is the cyclically ramified cover of $M$ along $S$ of degree $a$.
\end{example}

The hyperbolic metric $g_\bH$ on $M$ induces a singular hyperbolic metric on $M^a$ with a cone singularity of angle $2\pi a$ along $S\subset M^a$, the preimage of $S\subset \overline M$ by the covering map. In particular, for $a\geq 1$, this metric is locally $CAT(-1)$, implying that the fundamental group $\pi_1(M^a)$ is Gromov hyperbolic.

We will denote by $H_1,\ldots, H_{2k}$ the lifts of $\overline H_1$ and $\overline H_2$ to $M^a$ (in cyclic order around $S$), and denote by $V_i$ the component of $M^a \backslash \bigcap_{i=1}^{2k} H_i$ bounded by $H_i$ and $H_{i+1}$. (These notations are compatible with the ones introduced in the previous paragraph in the particular case $k=n$.)
\\

Applying Mostow's rigidity in dimension $d-1$ (and more precisely to the hypersurfaces $\overline H_1$ and $\overline H_2$), Gromov and Thurston show that in dimension $d\geq 4$ the fundamental group of $M^a$ is never isomorphic to a hyperbolic lattice when $a\neq 1$. Gromov--Thurston manifolds in dimension $d\geq 4$ are thus never homeomorphic to quotients of the hyperbolic space. In fact, their fundamental group is not commensurable to a lattice in any Lie group (see Remark \ref{rem-GTnotLattice}).


\section{Globally hyperbolic AdS manifolds and hyperbolic ends} 
\label{sc:globally}

We recall in this section some key notions concerning AdS geometry and more specifically the geometry of globally hyperbolic AdS spacetimes. Additional results can be found e.g. in \cite{mess,mess-notes,barbot-merigot,maximal}. We also present hyperbolic ends. 


\subsection{The anti-de Sitter space} 
Here we use the hyperboloid model of the anti-de Sitter space. Let $\R^{d,2}$ denote the real vector space $\R^{d+2}$ endowed with the standard quadratic form $\mathbf q$ of signature~$(d,2)$:
\[\mathbf q( x) = x_1^2 + \ldots + x_d^2 -x_{d+1}^2 - x_{d+2}^2~.\]
We denote by $\langle \cdot, \cdot \rangle$ the associated bilinear form.

\begin{defi}
The anti-de Sitter space of dimension $d+1$ is the quadric:
\[\AdS^{d+1} = \{x \in \R^{d,2} \mid \mathbf q ( x) = -1\}~.\]
\end{defi}

The restriction of $\mathbf q$ to the tangent bundle of $\AdS^{d+1}$ endows the anti-de Sitter space with a Lorentzian metric of constant sectional curvature $-1$, which we denote by $g_{\AdS}$. This metric is homogeneous under the action of the group $\mathrm O(d,2)$ of linear transformations of $\R^{d,2}$ preserving~$\mathbf q$. We denote by $\SO_\circ(d,2)$ the connected component of the identity in $\mathrm O(d,2)$. This is an index 4 subgroup consisting of those isometries of
$\AdS^{d+1}$ preserving an orientation of space and time. We call it for short the group of orientation-preserving isometries.

The subgroup of $\SO_\circ(d,2)$ fixing the point $(0,\ldots, 0, 1)\in \AdS^{d+1}$ is the group $\SO_\circ(d,1)$ embedded via
\[A\mapsto \left( \begin{matrix} A & \\ & 1\end{matrix} \right)~.\]
The (space and time-oriented) anti-de Sitter space of dimension $d+1$ can thus be identified with the coset space
\[\SO_\circ(d,2)/\SO_\circ(d,1)~.\\ \]

\emph{Boundary.} 
The space $\AdS^{d+1}$  can also be identified with an open set of the $d+1$-dimensional sphere, seen as the double cover of $\Proj{\R^{d,2}}$, via the map
\[\begin{array}{ccc}
\AdS^{d+1} & \to & \bS^{d+1} = (\R^{d,2}\setminus \{0\}) / \R_{>0} \Id\\
x & \mapsto & \R_{>0} x~.
\end{array}\]
Its boundary in $ \bS^{d+1}$ is called the \emph{Einstein space}.

\begin{defi}
The Einstein space $\Ein^d$ is defined as
\[ \Ein^d=\partial_\infty \AdS^{d+1}=\{x\in \R^{d,2}\backslash \{0\} \mid \mathbf q (x) = 0\}/ \R_{>0} \Id~.\]
\end{defi}
The Einstein space carries a conformally flat Lorentz metric which is conformally invariant under the action of $\mathrm \SO_\circ(d,2)$. \\

\emph{Geodesics and causality in $\AdS^{d+1}$.} The geodesics of $\AdS^{d+1}$ are its intersections with $2$-planes~$P$ in $\R^{d,2}$. These are of three kinds:
\begin{itemize}
\item If $\mathbf q_{\vert P}$ is negative definite, then $P\cap \AdS^{d+1}$ is an ellipse. It is a \emph{timelike geodesic}, i.e. the Lorentz metric is negative along that geodesic.
\item If $\mathbf q_{\vert P}$ is non-positive with $1$-dimensional kernel, then $P\cap \AdS^{d+1}$ consists of two parallel affine lines, each of which is a \emph{lightlike geodesic}, i.e. the Lorentz metric vanishes along that geodesic.
\item If $\mathbf q_{\vert P}$ has signature $(1,1)$, then $P\cap \AdS^{d+1}$ consists of two branches of hyperbolas, each of which is a \emph{spacelike geodesic}, i.e. the Lorentz metric is positive along that geodesic.
\end{itemize}

We call two points $x$ and $y$ in $\AdS^{d+1}$ \emph{space (resp. light, time) related} if they belong to the same spacelike (resp. lightlike, timelike) geodesic. We have the following characterization:

\begin{prop} \label{p:CharacterizationSpacelikeAdS}
Two points $x$ and $y\in \AdS^{d+1}$ are 
\begin{itemize}
\item space related if and only if $\langle x, y\rangle < -1$,
\item light related if and only if $\langle x, y\rangle = -1$,
\item time related if and only if $-1 < \langle x, y\rangle < 1$.
\end{itemize}
\end{prop}

\begin{remark}
If $\langle x, y\rangle \geq 1$ then $x$ and $y$ do not belong to a common geodesic, but $x$ and $-y$ are light or space related. In the projective model $\AdS^{d+1}/\pm \Id$, any two points are either space, light, or time related.
\end{remark}

\emph{Photons and causality in $\Ein^d$.} A \emph{photon} in $\Ein^d$ is the projectivisation of a totally isotropic $2$-plane in $\R^{d,2}$. We call two points $[x]$ and $[y] \in \Ein^d$ \emph{light related} if they belong to the same photon and \emph{space related} if they are the endpoints of a spacelike geodesic in $\AdS^{d+1}$. 
 We have again a characterization in terms of scalar products:

\begin{prop}
Two points $[x]$ and $[y]\in \Ein^d$ are 
\begin{itemize}
\item space related if and only if $\langle x, y\rangle < 0$,
\item light related if and only if $\langle x, y\rangle = 0$.
\end{itemize}
\end{prop}

Note that $[x]$ is always space or light related to either $[y]$ or $[-y]$. Causality thus does not make sense in the projective model $\mathbb E\mathrm{in}^d = \Ein^d/\pm \Id$. A more robust notion is space relation for triples of points.

\begin{defi}
We call a subset $S$ of $\Ein^d$ \emph{acausal} if any two points in $S$ are space related, and \emph{achronal} if any two points in $S$ are space or light related.
\end{defi}

\begin{prop}
If $\{[x], [y], [z]\}\subset \Ein^d$ is acausal, then the restriction of $\mathbf q$ to $\Span(x,y,z)$ has signature $(2,1)$.

Conversely, if the restriction of $\mathbf q$ to $\Span(x,y,z)$ has signature $(2,1)$, then there exist unique $\epsilon_y$ and $\epsilon_z\in \{-1,1\}$ such that $\{[x], [\epsilon_y y], [\epsilon_z z]\}$ is acausal.
\end{prop}

\subsection{Spacelike hypersurfaces in $\AdS^{d+1}$}  
\label{ssc:spacelike hypersurfaces}

A smooth hypersurface $\mathcal H$ in $\AdS^{d+1}$ is called spacelike when the restriction of the Lorentz metric to $\mathcal H$ is positive definite. Here, we will construct hypersurfaces that are piecewise geodesic, and it is thus useful to generalize this definition to a lower regularity.\\

Let $\mathcal H$ be a Lipschitz manifold of dimension $d$. Let $d_{\mathcal H}$ be a Lipschitz distance on $\mathcal H$ (i.e. a distance which is locally bi-Lipschitz to the Euclidean distance in local coordinates).

\begin{defi}
A map $i : \mathcal H \to \AdS^{d+1}$ is a \emph{spacelike immersion} if it is locally Lipschitz and if  every point in $\mathcal H$ has a neighbourhood $U$ such that there exists a constant $c>0$ satisfying
\[\langle i(p),i(p')\rangle \leq -1 - c\, d_{\mathcal H}^2(p,p')\]
for all $p,p'\in U$.

A Lipschitz hypersurface $\mathcal H$ of $\AdS^{d+1}$ is called \emph{spacelike} if the inclusion $i:\mathcal H \to \AdS^{d+1}$ is a spacelike immersion.
\end{defi}

By Proposition \ref{p:CharacterizationSpacelikeAdS} the above condition implies that $i(p)$ is space related to $i(p')$ when $p,p'\in \mathcal H$ are sufficiently close. The constant $c$ prevents the hypersurface to be  ``tangent'' to the light cone through $p$. In particular, we have:

\begin{prop} \label{prop-smoothspacelikeimmersion}
If $\mathcal H$ and $i$ are of class $\mathcal C^1$, then $i$ is a spacelike immersion if and only if $i^*g_{\AdS}$ is positive definite at every point. 
\end{prop}

Before we prove Proposition \ref{prop-smoothspacelikeimmersion}, let us interpret spacelike immersions in terms of graphs in an appropriate model for $\AdS^{d+1}$. Consider the open hemisphere $\bS^d_+=\{(x_0,\dots,x_d)\in \bS^d \,|\, x_0>0 \}$. The map
\[ \Phi : \left\lbrace\begin{matrix} \bS^1\times \bS^d_+ & \to & \AdS^{d+1} \\ (\theta,x) & \mapsto & \left( \frac{x_1}{x_0},\dots,\frac{x_d}{x_0},\frac{\cos \theta}{x_0}, \frac{\sin\theta}{x_0}\right) \end{matrix} \right. \]
is a diffeomorphism, and $\Phi^*g_{\AdS}=x_0^{-2}(-d\theta^2+g_{\bS^d})$. An important feature of this model is that $\bS^d_+$ equipped with $x_0^{-2}g_{\bS^d}$ is isometric to the hyperbolic space $\HH^d$, in particular the hypersurfaces $\theta={\rm cst}$ are totally geodesic copies of $\HH^d$. Another one is that it extends to the boundary as a conformal map $\partial\Phi:\bS^1\times \bS^{d-1}\to \partial\AdS^{d+1}=\Ein^d$.

\begin{lemma} \label{lem-spacelike implies graph locally}
Let $\mathcal H$ be a Lipschitz manifold, and $i:\mathcal H\to \AdS^{d+1}$ a spacelike immersion. Write $i(p)=\Phi(\theta(p),x(p))$ for $p\in \mathcal H$. Then the map $x:\mathcal H\to \bS^d_+$ is locally bi-Lipschitz.
\end{lemma}

\begin{proof}
For $(\theta,x),(\theta',x')\in \bS^1\times \bS^d_+$ we find
\begin{align*} \langle\Phi(\theta,x),\Phi(\theta',x')\rangle &= \frac{x_1x'_1+\cdots+x_dx'_d-\cos(\theta-\theta')}{x_0x'_0} \\
&= \frac{1-\frac{1}{2}\Vert x-x'\Vert^2-\cos(\theta-\theta')}{x_0x'_0}-1\\
&\geq -\frac{\Vert x-x'\Vert^2}{2x_0x'_0}-1
\end{align*}
This shows that  for $p,p'\in \mathcal H$, we have
\[ \Vert x(p)-x(p')\Vert^2 \geq -x_0(p)x_0(p') (1+\langle i(p),i(p')\rangle \]
Since $p\mapsto x_0(p)$ is continuous, it is locally bounded from below by some $c'>0$, and locally we find
\[ \Vert x(p)-x(p')\Vert \geq \frac{\sqrt{c}}{c'}d_{\mathcal H}(p,p') \]
\end{proof}

Lemma \ref{lem-spacelike implies graph locally} means that a Lipschitz spacelike hypersurface is locally a graph in the conformal model $\AdS^{d+1}\approx \bS^1\times \bS^d_+$. Now given a function from $\bS^d_+$ to $\bS^1$, we wish to know under which condition its graph is a Lipschitz spacelike hypersurface.

\begin{lemma} \label{lm:contracting graph is spacelike}
Let $U\subset \bS^d_+$ be an open subset, and consider a map $\theta:U\to \bS^1$. The map $i:U\to \AdS^{d+1}$ defined by $i(x)=\Phi(\theta(x),x)$ is a spacelike immersion if and only if $\theta$ is locally contracting (i.e. every point in $U$ has a neighbourhood on which $\theta$ is $k$-Lipschitz for some $k\in (0,1)$).
\end{lemma}

\begin{remark} 
As a distance on $\bS^d_+$ we can pick either the spherical distance or the Euclidean distance. The condition on $\theta$ being locally contracting does not depend on this choice, since for every $\varepsilon>0$ we can find a neighbourhood of any point on which they are $1+\varepsilon$-bi-Lipschitz. We will use the Euclidean distance in the proof. 
\end{remark}

\begin{proof}[Proof of Lemma \ref{lm:contracting graph is spacelike}]
For $x,x'\in U$ we have
\[ \langle i(x),i(x')\rangle =  \frac{1-\frac{1}{2}\Vert x-x'\Vert^2-\cos(\theta(x)-\theta(x'))}{x_0x'_0}-1 .\]
If $\theta$ is $k$-Lipschitz for some $k\in (0,1)$, then up to shrinking $U$ we obtain
\[ 1-\cos(\theta(x)-\theta(x'))\leq \frac{k}{2}\Vert x-x'\Vert^2 \]
Since $x_0,x'_0\leq 1$, we find
\[ \langle i(x),i(x')\rangle \leq -1- \frac{1-k}{2}\Vert x-x'\Vert^2~,\]
hence $i$ is a spacelike immersion.

Now assume that $i$ is a spacelike immersion, and let $c\in (0,\frac{1}{2})$ be a constant given by the definition (note that $c$ can always be replaced by a smaller constant). The map $\theta$ is continuous, so  we can shrink $U$ in order to have 
\[ 1-\cos(\theta(x)-\theta(x'))\geq \frac{1-c}{2}\vert \theta(x)-\theta(x')\vert^2 \]
 for all $x,x'\in U$. This in turn leads to
\[ \vert \theta(x)-\theta(x')\vert^2 \leq \underbrace{\frac{\frac{1}{2}-c}{\frac{1}{2}-\frac{c}{2}}}_{<1} \Vert x-x'\Vert^2~.
 \]
\end{proof}

\begin{proof}[Proof of Proposition \ref{prop-smoothspacelikeimmersion}]
If $i$ is $\mathcal C^1$, then the map $x:\mathcal H\to \bS^d_+$ from Lemma \ref{lem-spacelike implies graph locally} is bi-Lipschitz and $\mathcal C^1$, hence a local diffeomorphism. So we may assume that $\mathcal H$ is an open subset of $\bS^d_+$ and that $i(x)=\Phi(\theta(x),x)$ where $\theta:\mathcal H\to \bS^1$ is $\mathcal C^1$. Now Lemma \ref{lm:contracting graph is spacelike} shows that $i$ is a spacelike immersion if and only if $\theta$ is locally contracting, which is equivalent to $\Vert d\theta\Vert<1$.\\
But $i^*g_{\AdS}$ is in the same conformal class as $-d\theta^2+g_{S^d}$, so it is positive definite  if and only if $\Vert d\theta\Vert <1$.
\end{proof}

More generally, if $i:\mathcal H \to \AdS^{d+1}$ is a spacelike Lipschitz immersion and $\gamma:[0,1] \to \mathcal H$ is a Lipschitz path, then $i\circ \gamma$ is Lipschitz hence differentiable at Lebesgue almost every point. The derivative of $i\circ \gamma$ is never timelike, so one can then define the \emph{length} of $i(\gamma)$ as
\[L_\AdS(i(\gamma)) = \int_0^1 \sqrt{g_{\AdS}((i\circ \gamma)'(t), (i\circ \gamma)'(t))} \mathrm d t~.\]

Finally, for $x,y\in \mathcal H$, set
\[i^*d_\AdS(x,y) = \inf_{\gamma(0)=x, \gamma(1)=y} L_\AdS(i\circ \gamma)~.\]
The spacelike immersion property of $i$ implies that $i^*d_\AdS$ is a Lipschitz distance on $\mathcal H$ (the usual proof for smooth Riemannian metrics can be easily adapted thanks to the graph description of Lemma \ref{lem-spacelike implies graph locally}). When $\mathcal H$ and $i$ are $\mathcal C^1$, it is simply the Riemannian distance associated to $i^*g_\AdS$.

\begin{defi}
The spacelike immersion $i$ is called \emph{complete} when the distance $i^*d_\AdS$ is complete.\\
A spacelike hypersurface $\mathcal H\subset \AdS^{d+1}$ is called \emph{complete} if the inclusion $i:\mathcal H\to \AdS^{d+1}$ is a complete spacelike immersion.
\end{defi}

The following proposition brings together a number of key properties of spacelike hypersurfaces in $\AdS^{d+1}$.

\begin{prop} \label{prop: properties of complete spacelike hypersurfaces}
Let $\mathcal H$ be a connected Lipschitz manifold of dimension $d$ and $i: \mathcal H \to \AdS^{d+1}$ a complete spacelike immersion. Then 
\begin{itemize}
\item[(1)] $\mathcal H$ is homeomorphic to the open ball $\mathrm B^d$ of dimension $d$,
\item[(2)] $i$ is an embedding,
\item[(3)] $i(x)$ and $i(y)$ are space related for any $x\neq y \in \mathcal H$,
\item[(4)] $i$ extends continuously to an achronal embedding
\[\partial i:  \bS^{d-1}= \partial \mathrm B^d \to \Ein^d~.\\ \]
\end{itemize}
\end{prop}

\begin{proof}
Write $i(p)=\Phi(\theta(p),x(p))$ for $p\in \mathcal H$. We have seen in Lemma \ref{lem-spacelike implies graph locally} that $x:\mathcal H\to \bS^d_+$ is locally bi-Lipschitz. We can use it to define another distance on $\mathcal H$ by defining the length of a Lipschitz curve $\gamma:[0,1]\to \mathcal H$ as the hyperbolic length $L_{\HH^d}(x\circ \gamma)$, and the distance $x^*d_{\HH^d}$ as the infimumum of lengths of paths joining two points.\\
Note that $L_\AdS(i\circ\gamma)\leq L_{\HH^d}(x\circ \gamma)$, so $i^*d_\AdS\leq x^*d_{\HH^d}$. Completeness of $i$ and the Hopf--Rinow Theorem for length spaces (see \cite[Theorem 1.9]{gromov_book} or \cite[Proposition I.3.7]{bridson_haefliger}) imply that closed balls for $i^*d_\AdS$ are compact, so  $x^*d_{\HH^d}$ is also complete. This shows that $x$ is a local isometry between length spaces, the source being complete, so it is a covering map (see \cite[Proposition I.3.28]{bridson_haefliger} for a proof in the context of length spaces). This implies that the hyperbolic metric on $x(\mathcal H)$ is complete, so $x$ is onto. Since it is a covering, it must be a homeomorphism. This proves $(1)$ and $(2)$.\\ 
Now consider a lift $\widetilde i:\mathcal H\to \widetilde{\AdS}^{d+1}$ to the universal cover. Lemma \ref{lem-spacelike implies graph locally} shows that $\widetilde i(\widetilde{\mathcal H})$ is weakly spacelike as defined in \cite[Section 3.2]{maximal}, so property $(3)$  follow from \cite[Proposition 3.5]{maximal}. This in turn implies that $\mathcal H$ is (globally) the graph of a $1$-Lipschitz function $\theta:\bS^d_+\to \R$ (for the spherical distance), so $(4)$ follows from the extendability of Lipschitz functions.

\end{proof}

As a consequence of this proof we get the following description of complete spacelike hypersurfaces in $\AdS^{d+1}$.

\begin{cor} \label{cor:complete spacelike hypersurfaces are graphs}
Let $\mathcal H\subset \AdS^{d+1}$ be a complete spacelike hypersurface. There is a  distance decreasing function $\theta:\bS^d_+\to \bS^1$ (i.e. $\vert \theta(x)-\theta(x')\vert < d_{\bS^d}(x,x')$ whenever $x\neq x'$) such that $\mathcal H=\{\Phi(\theta(x),x)|x\in S^d_+\}$.
\end{cor}

\begin{remark}
Here it is important to use the spherical distance on $\bS^d_+$ rather than the Euclidean distance.
\end{remark}

\subsection{Second fundamental form} Here we recall the classical notion of second fundamental form in a setting which applies both to spacelike hypersurfaces in $\AdS^{d+1}$ and to $\AdS^{d+1}$ itself inside the flat pseudo-Riemannian space $\R^{d,2}$.\\

Let $(M,g)$ be a smooth oriented pseudo-Riemannian manifold of signature $(p,q)$, and let $\mathcal H \subset M$ be an oriented hypersurface of class $\mathcal C^2$ such that the restriction of $g$ to $\mathcal H$ has signature $(p,q-1)$. Let us denote by $TM_{|\mathcal H}$ the pull-back of the tangent bundle $TM$ by the inclusion $\mathcal H\hookrightarrow M$. The tangent bundle $T\mathcal H$ is a sub-bundle of $TM_{|\mathcal H}$ and $TM_{|\mathcal H}$ splits orthogonally (with respect to $g$) as
\[TM_{|\mathcal H} = T\mathcal H \oplus \R N~,\]
where $N$ is the unit normal to $\mathcal H$ (i.e. $g(N,N) \equiv -1$ and the orientation of $N$ is compatible with those of $M$ and $\mathcal H$).

The Levi--Civita connection $\nabla^M$ of $M$ restricts to a connection on $TM_{|\mathcal H}$ (that we still denote $\nabla^M$). Since $g(N,N)$ is constant, $\nabla^M_X N$ is orthogonal to $N$ and thus tangent to $\mathcal H$ for every vector $X$ tangent to $\mathcal H$.

\begin{defi}
The \emph{second fundamental form} of $\mathcal H$ is the bilinear form on $T\mathcal H$ defined by
\[\secondFF_{\mathcal H}(X,Y) = g(\nabla^M_X N,Y)~.\]
\end{defi}

The second fundamental form relates the Levi--Civita connection $\nabla^{\mathcal H}$ of $(\mathcal H,g_{|\mathcal H})$ to the ambient connection $\nabla^M$:

\begin{prop}
For every vector fields $X$ and $Y$ on $\mathcal H$, we have
\[\nabla^M_X Y = \nabla^{\mathcal H}_X Y + \secondFF_{\mathcal H}(X,Y) N~.\]
\end{prop}

\begin{remark}
The reader familiar with Riemannian geometry will notice a sign difference in the definition of the second fundamental form. This is due to the fact that we assume here that $g$ is negative in the normal direction.\\
\end{remark}

\subsection{Convexity in $\AdS^{d+1}$}

Let $V$ be a hyperplane in $\R^{d,2}$ in restriction to which the quadratic form $\mathbf q$ has signature $(d,1)$. Then $V\cap \AdS^{d+1}$ is a two-sheeted hyperboloid of dimension $d$, each connected component of which is a totally geodesic spacelike hypersurface. We call such a connected component a \emph{spacelike hyperplane}. They are the totally geodesic copies of $\HH^d$ in $\AdS^{d+1}$. For any $\theta_0\in \bS^1$ the set $\Phi(\{\theta_0\}\times \bS^d_+)$, where $\Phi:\bS^1\times \bS^d_+\to \AdS^{d+1}$ is the diffeomorphism defined above, is a spacelike hyperplane.

If $W \subset V\cap \AdS^{d+1}$ is a spacelike hyperplane, we denote by $\overline{W}$ the other component of $V\cap \AdS^{d+1}$, i.e. the image of $W$ by $x\mapsto -x$. We say that a point $x\in \AdS^{d+1} \backslash \overline{W}$ is \emph{in the past} of $W$ if there exists a future-oriented timelike geodesic segment from $x$ to a point of $W$ which does not intersect $\overline{W}$.

\begin{defi}
Let $\mathcal H\subset\AdS^{d+1}$ be a Lipschitz spacelike hypersurface. We say that $\mathcal H$ is \emph{convex} if for every $x\in \mathcal H$, there exists a spacelike hyperplane $W$ containing $x$ such that $\mathcal H$ is contained in the past of $W$.

The hypersurface $\mathcal H$ is \emph{locally convex} if every point $x\in \mathcal H$ has an open neighbourhood $U$ such that $\mathcal H \cap U$ is convex.
\end{defi}

For complete hypersurfaces, the local convexity property globalizes.

\begin{prop} \label{p: Local convexity -> convexity}
Let $\mathcal H$ be a complete Lipschitz spacelike hypersurface. Then $\mathcal H$ is convex if and only if it is locally convex.
\end{prop}

\begin{proof}
Let $p\in \mathcal H$ and consider a spacelike hyperplane $W$ containing $p$ such that a neighbourhood of $p$ in  $\mathcal H$ is in the past of $W$. 

Up to the action of $\SO_\circ(d,2)$, we may assume that $W=\Phi(\{0\}\times S^d_+)$, and write $p=\Phi(0,x)$. We have seen in the proof of Proposition \ref{prop: properties of complete spacelike hypersurfaces} that $\mathcal H$ is the graph of a function $\theta:\bS^d_+\to \bS^1$  in the conformal model $\AdS^{d+1}\approx \bS^1\times \bS^d_+$. Since $\theta$ is distance decreasing, it is not onto and we may consider $\theta$ as a real valued map.

Assume by contradiction that $\mathcal H$ is not in the past of $W$, i.e. $\theta>0$ at some point $y\in \bS^d_+$.  Along the spherical geodesic $\gamma:[0,1]\to \bS^d_+$ joining $x$ and $y$, the function $\theta$ possesses a local minimum. If $z=\gamma(t_0)$ is such a point of $\bS^d_+$, consider a spacelike hyperplane $W'$ containing $\Phi(\theta(z),z)$ such that a neighbourhood of $\Phi(\theta(z),z)$ in $\mathcal H$ is in the past of $W'$. Note that $W'$ is the graph of a function $\alpha:\bS^d_+\to S^1$.  For $t$ near $t_0$, we have $\alpha\circ\gamma(t)\geq \theta\circ \gamma(t)\geq \theta(z)$, with equality at $t=t_0$. Hence $(\alpha\circ\gamma)'(t_0)=0$, and the geodesic $t\mapsto \Phi(\alpha\circ\gamma(t),\gamma(t))$ is tangent to $\Phi(\{\theta(z)\},\bS^d_+)$ at $\Phi(\theta(z),z)$, therefore $\alpha\circ\gamma$ is constant. It implies that $\theta\circ\gamma$ is constant on a neighbourhood of every point where a local minimum is obtained, hence $\theta\circ\gamma\geq 0$ which is a contradiction. 
\end{proof}

For hypersurfaces of class $\mathcal C^2$, convexity is characterized by the sign of the second fundamental form:

\begin{prop}
Let $\mathcal H$ be a spacelike hypersurface of class $\mathcal C^2$. Then $\mathcal H$ is locally convex if and only if its second fundamental form is non-negative.
\end{prop}

This motivates the following strengthenings:

\begin{defi} \label{df:unif-convex}
A complete spacelike hypersurface $\mathcal H$ of class $\mathcal C^2$ is called \emph{strongly convex} if its second fundamental form $\secondFF_{\mathcal H}$ is positive definite, and \emph{uniformly strongly convex} if there exists a constant $c>0$ such that
\[ \secondFF_{\mathcal H} \geq c\, g_{|\mathcal H}~.\]
\end{defi}
 
\subsection{Globally hyperbolic spacetimes} \label{subsec: globally hyperbolic spacetimes} 

Let $N$ be a Lorentzian manifold. A $\mathcal C^1$ curve on $N$ is called \emph{causal} if its tangent direction is nowhere spacelike. Such a curve is called \emph{inextensible} if it is maximal among causal curves (for the inclusion).

\begin{defi}
A \emph{Cauchy hypersurface} in $N$ is a topological hypersurface intersecting every inextensible causal curve at exactly one point. A Lorentzian manifold admitting a Cauchy hypersurface is called \emph{globally hyperbolic}.
\end{defi}

A globally hyperbolic Lorentzian manifold $N$ always admits a \emph{temporal function}, i.e. a function  to $\R$ with no critical points whose level sets are spacelike hypersurfaces, see e.g. \cite{bernal-sanchez}. Moreover, all smooth Cauchy hypersurfaces are diffeomorphic, and if $M$ is a smooth Cauchy hypersurface, then $N$ is diffeomorphic to $M\times \R$.

\begin{defi}
A globally hyperbolic Lorentzian manifold is \emph{Cauchy compact} if its Cauchy hypersurfaces are compact.
\end{defi}

\begin{defi}
  A globally hyperbolic Cauchy compact 
  $\AdS$ manifold is \emph{maximal} if it is not isometric to a proper subset of another globally hyperbolic $\AdS$ manifold.
\end{defi}

From now on, we abreviate ``globally hyperbolic Cauchy compact'' into ``GHC'', and ``globally hyperbolic maximal Cauchy compact'' into ``GHMC''.

We recall here a description of GHMC AdS spacetimes, due to Mess \cite{mess,mess-notes}. (It is only stated in dimension $2+1$ in \cite{mess}, but the argument works in higher dimension as pointed in \cite{barbot_causal}. For other proofs see \cite[Corollary 11.2]{barbot_causal} and \cite[Proposition 4.8]{barbot-merigot}.)

\begin{defi} \label{def:domainofdependence}
Let $\Lambda$ be a closed subset of $\Ein^d$. The \emph{domain of dependence} of $\lambda$ is the open set
\[\Omega(\Lambda)= \{x\in \AdS^{d+1} \mid \langle x, y\rangle <0 \textrm{ for all } [y] \in  \Lambda\}~.\]
\end{defi}

The following theorem describes GHMC AdS manifolds as quotients of the domain of dependence of the boundary at infinity of a Cauchy hypersurface. Il was proved by Mess in dimension 2+1 and extended by Barbot in higher dimensions.

\begin{theorem}[Mess] \label{theo-mess}
$\Gamma$ be a subgroup of $\SO_\circ (d,2)$ acting freely, properly discontinuously and cocompactly on a spacelike hypersurface $\mathcal H$ and let $\partial_\infty \mathcal H$ be the boundary of $\mathcal H$ in $\Ein^d$. Then $\Gamma$ acts properly discontinuously on $\Omega(\partial_\infty \mathcal H)$, the quotient $N= \Gamma \backslash \Omega(\partial_\infty \mathcal H)$ is a GHMC AdS manifold and $\Gamma \backslash \mathcal H \subset N$ is a Cauchy hypersurface.

Conversely, let $N$ be a GHMC AdS manifold of dimension $d+1$. Then there exists a discrete subgroup $\Gamma$ of $\SO_\circ(d,2)$ and a $\Gamma$-invariant complete spacelike hypersurface $\mathcal H \subset \AdS^{d+1}$ such that $\Gamma$ acts properly discontinuously and cocompactly on $\mathcal H$ and $N$ is isometric to $\Gamma \backslash \Omega(\partial_\infty \mathcal H)$.
\end{theorem}


The following statement combines results from Barbot--Mérigot \cite{barbot-merigot}, Barbot \cite{barbot_deformations} and Danciger--Guéritaud--Kassel \cite{DGK:convex}.

\begin{theorem}[Barbot--Mérigot, Barbot,  Danciger--Gu\'eritaud--Kassel]  \label{thm-characterization AdS convex}
Let $N= \Gamma \backslash \Omega$ be a GHMC $\AdS$ manifold. The following properties are equivalent:
\begin{itemize}
\item[$(i)$] The group $\Gamma$ is Gromov hyperbolic,
\item[$(ii)$] The limit set $\Lambda_\Gamma$  is acausal,
\item[$(iii)$] The manifold $N$ contains a convex Cauchy hypersurface,
\item[$(iv)$] The manifold $N$ contains a strongly convex Cauchy hypersurface,
\item[$(v)$] The group $\Gamma$ acts convex-cocompactly on $\Omega$.
\end{itemize}
\end{theorem}

We call such a GHMC $\AdS$ manifold \emph{quasifuchsian}. It is called \emph{Fuchsian} if it possesses a totally geodesic Cauchy hypersurface.

\begin{proof}[Sketch of the proof]
$(v)\Rightarrow(iv)$ is Lemma 6.4 in \cite{DGK:convex}.\\
$(iv)\Rightarrow(iii)$ is straightforward.\\
$(iii)\Rightarrow(i)$ follows from Proposition 8.3 in \cite{barbot-merigot}.\\
$(i)\Rightarrow(ii)$ is Theorem 1.4 in \cite{barbot_deformations}.\\
Finally, the main result of \cite{barbot-merigot} is that $(ii)$ is equivalent to $\Gamma$ being $P_1$-Anosov, and the latter is equivalent to $(v)$ according to Theorem 1.7 in \cite{DGK:convex}.
\end{proof}

\subsection{Spacelike $\AdS$ structures}
\label{ssc:spacelike structures}

In this section we introduce a notion \emph{spacelike $\AdS$ structure} on  a manifold $M$, in a way that emulates the notion of $(G,X)$-structure. The ``developing map'' of such a structure is an equivariant spacelike immersion of the universal cover of $M$, from which one obtains a GHMC $\AdS$ manifold homeomorphic to $M\times \R$. This will allow us to reduce the construction of GHMC manifolds with prescribed Cauchy hypersurfaces to the construction of a spacelike AdS structure on a manifold one dimension lower.

Let $M$ be a Lipschitz manifold of dimension $d$.

\begin{defi}
A \emph{spacelike $\AdS$ atlas} on $M$ is the data of an atlas $(U_i, \phi_i)_{i\in I}$ where $(U_i)_{i\in I}$ is an open cover of $M$ and $\phi_i: U_i \to \AdS^{d+1}$ is a Lipschitz spacelike immersion, such that for all $i,j \in I$ and all $x\in U_i\cap U_j$, there exists an orientation preserving isometry $g$ of $\AdS^{d+1}$ such that
\[\phi_j = g\circ \phi_i\]
in a neighbourhood of $x$.

A \emph{spacelike $\AdS$ structure} on $M$ is a maximal spacelike $\AdS$ atlas.

Two spacelike $\AdS$ atlases $(U_i, \phi_i)_{i\in I}$ and $(V_j, \psi_j)_{j\in J}$ on $M$ are \emph{equivalent} if for every $(i,j)\in I\times J$ and every $x\in U_i\cap V_j$, there exists an orientation preserving isometry $g$ of $\AdS^{d+1}$ such that
\[\psi_j = g\circ \phi_i\]
in a neighbourhood of $x$.
\end{defi}

Note that this definition is almost identical to that of a $(G,X)$-structure, except that the charts are required to be spacelike immersions instead of local homeomorphisms.\\

Emulating the theory of $(G,X)$-structures, we want to ``patch together'' the local charts into an equivariant spacelike immersion of the universal cover. In order to do so, remark first that there is a unique way to patch together two local charts:

\begin{lemma} \label{lm:stabiliserspacelikeimmersion}
Let $U$ be a non empty open subset of $M$, $\phi:U\to \AdS^{d+1}$ a spacelike immersion, and $g$ an orientation preserving isometry of $\AdS^{d+1}$ such that $g\circ \phi = \phi$. Then $g = \Id$.
\end{lemma}

\begin{proof}
Consider a point $x\in \phi(U)$ at which $\phi(U)$ is differentiable (it exists because of Rademacher's Theorem). The tangent space $T_{x}\phi(U)$ is spacelike because of Lemma \ref{lm:contracting graph is spacelike}. Now $g$ fixes every point of the spacelike hyperplane tangent to $\phi(U)$ at $x$, hence $g=\Id$.
\end{proof}

As a consequence, we can still define the holonomy representation and the developing map of a spacelike $\AdS$ structure.

\begin{cor} \label{cor-developing map spacelike AdS structure}
Let $(U_i, \phi_i)_{i\in I}$ be a spacelike $\AdS$ atlas on $M$. Then there exists a representation $\rho:\pi_1(M) \to \SO_\circ(d,2)$ and a $\rho$-equivariant spacelike immersion $\phi: \tilde M \to \AdS^{d+1}$ such that, for all $x\in \tilde M$ and all $U_i$ containing $\pi(x)$, there exists  $g\in \SO_\circ(d,2)$ such that
\[\phi_i \circ \pi = g \circ \phi\]
in a neighbourhood of $x$. (Here, $\pi:\tilde M \to M$ denotes the covering map.)

Moreover, if another pair $(\rho', \phi')$ satisfies the same properties, then there is a unique $g\in \SO_\circ(d,2)$ such that $ \rho' = g\rho g^{-1}$ and $\phi' = g\circ \phi$.
\end{cor}

\begin{proof}
The proof is the same as for $(G,X)$-structures. Let us write $G=\SO_\circ(d,2)$ and $H=\SO_\circ(d,1)$, so that $\AdS^{d+1}=G/H$.\\
Given $i,j\in I$ and $x\in U_i\cap U_j$, write $g_{ij}(x)\in G$ the element such that $\phi_i=g_{ij}(x)\cap \phi_j$ on a neighbourhood of $x$ (it is unique thanks to Lemma \ref{lm:stabiliserspacelikeimmersion}). Because of its uniqueness, it satisfies the cocycle rule $g_{ik}(x)=g_{ij}(x)g_{jk}(x)$ whenever $x\in U_i\cap U_j\cap U_k$, and we can consider the $G$-principal bundle $P$ over $M$ with transitions $g_{ij}$.\\
Another consequence of Lemma \ref{lm:stabiliserspacelikeimmersion} is that the map $x\mapsto g_{ij}(x)$ is locally constant, so $P$ inherits a flat connection, and we can consider its holonomy representation $\rho:\pi_1(M)\to G$.\\
Now let $E$ be the associated $G/H$-bundle over $M$. The relation $\phi_i=g_{ij}\circ \phi_j$ shows that there is a section $\sigma$ of $E$ that locally reads as $\phi_i$. This section lifts to a $\rho$-equivariant map $\phi:\widetilde M\to G/H=\AdS^{d+1}$ with the required properties.\\
If another pair $(\rho', \phi')$ satisfies the same properties, then for any $x\in\widetilde M$ there is an element $g(x)\in G$ such that $\phi'=g(x)\circ \phi$ on a neighbourhood of $x$. Using Lemma \ref{lm:stabiliserspacelikeimmersion} once again, we see that the map $x\mapsto g(x)$ is locally constant, hence constant. The equivariance implies that $\rho'=g\rho g^{-1}$.
\end{proof}

Conversely, a representation $\rho:\pi_1(M)\to \SO_\circ(d,2)$ and an equivariant spacelike immersion $\phi:\tilde M\to \AdS^{d+1}$  define a spacelike AdS structure on $M$ by choosing an open cover $(U_i)_{i\in I}$ of $M$ so that $\pi:\tilde M\to M$ is invertible on each $U_i$,  with  $\phi_i:U_i\to \AdS^{d+1}$ defined as  $\phi\circ \pi^{-1}$ on $U_i$. 

Let now $\phi:\tilde{M}\to \AdS^{d+1}$ be a $\rho$-equivariant spacelike immersion. Note that the pulled-back distance $\phi^*d_\AdS$ (as defined in Section \ref{ssc:spacelike hypersurfaces}) is $\pi_1(M)$-invariant. If $M$ is moreover compact, then it is complete, and $\phi$ is thus an embedding (see Proposition \ref{prop: properties of complete spacelike hypersurfaces}). This implies that $\rho$ is discrete and faithful and $\rho(\pi_1(M))$ acts properly discontinuously on $\phi(\tilde M)$. By Theorem \ref{theo-mess}, $\rho$ is the thus the holonomy of a GHMC AdS spacetime $N$ and $\phi$ embeds $M$ as a Cauchy hypersurface in $N$. We thus obtain the following:

\begin{theorem} \label{pr:spacelike}
Let $M$ be a closed manifold of dimension $d$, let $\rho$ be a representation of $\pi_1(M)$ into $\SO_\circ(d,2)$ and let $\phi: \tilde{M} \to \AdS^{d+1}$ be a $\rho$-equivariant spacelike embedding. Then there exists a unique $\rho$-invariant open domain $\Omega_\rho \subset \AdS^{d+1}$ such that:
\begin{enumerate}
\item $\pi_1(M)$ acts properly discontinuously on $\Omega_\rho$ via $\rho$,
\item the quotient $N_\rho= \rho(\pi_1(M))\backslash \Omega_\rho$ is a GHMC $\AdS$ manifold,
\item the map $\phi$ factors to an embedding of $M$ into $N_\rho$ whose image is a Cauchy hypersurface. 
\end{enumerate}
\end{theorem} 

\subsection{Convex ruled spacelike AdS structures}
\label{ssc:convex}

\begin{defi}
  Let $N$ be a quasifuchsian AdS spacetime of dimension $d+1$. A Lipschitz hypersurface $M\subset N$ is called \emph{spacelike} if its lift $\tilde M$ to the universal cover of $N$ (which is an open subset of  $\AdS^{d+1}$ by Mess's Theorem) is a Lipschitz spacelike hypersurface.\\ We say that $M$ is \emph{past-convex} if $\tilde M$ is past-convex.\\ We say that $M$ is {\em ruled} if  each $x\in M$ lies in the relative interior of a geodesic segment  of $N$ which is contained in $M$.  
\end{defi}

\begin{lemma} \label{lm:unique}
  Let $N$ be a quasifuchsian AdS spacetime. Then $N$ contains a unique past-convex ruled Cauchy hypersurface. 
\end{lemma}

\begin{proof}
For the existence, write $N=\Gamma\backslash \Omega$, and let $\Lambda=\partial\Omega\cap\partial_\infty \AdS^{d+1}$ be its limit set. We denote by $\Conv(\Lambda)\subset \AdS^{d+1}\cup\partial_\infty \AdS^{d+1}$ the convex hull of $\Lambda$, and $C(N)=\Gamma\backslash (\Conv(\Lambda)\setminus\Lambda)\subset N$ the convex core of $N$. Its boundary $\partial N$ has two connected components, and we consider the future  component $\partial_+C(N)$. It is a Cauchy hypersurface in $N$ \cite[Lemma 4.9]{barbot-merigot}, and moreover a spacelike hypersurface thanks to \cite[Lemma 3.16]{barbot-merigot} and Lemma \ref{lm:contracting graph is spacelike}. It is past-convex by definition. Let $x\in \partial_+C(N)$, and consider a lift $\tilde x\in \partial\Conv(\Lambda)$. This lift does not belong to $\Lambda$, so it cannot be an extreme point of the convex set $\Conv(\Lambda)$, it therefore lies  in the relative interior of a geodesic segment of $\AdS^{d+1}$ contained in $\Conv(\Lambda)$. This geodesic segment is spacelike because $\partial_+C(N)$ is a spacelike hypersurface. It must lie in $\partial \Conv(\Lambda)$ because the interior of $\Conv(\Lambda)$ is convex.\\
Now for the uniqueness, let $S\subset N$ be a ruled past convex Cauchy hypersurface, and $\tilde S\subset \Omega$ its lift. We then have $\partial \tilde S \subset \partial_\infty \AdS^{d+1}=\Lambda$ thanks to \cite[Corollary 3.8]{maximal}. Let $\mathcal C\subset \AdS^{d+1}\cup\partial_\infty \AdS^{d+1}$ be the closed convex hull of $\tilde S$. It is  compact and convex, so by the Krein--Milman Theorem it is the convex hull of its set of extreme points $E\subset \tilde S\cup \Lambda\subset \AdS^{d+1}\cup\partial\AdS^{d+1}$. Since $S$ is ruled, we have $\tilde S\cap E=\emptyset$, so $\tilde S\subset \mathcal C=\Conv(\Lambda)$. Since $S$ is past-convex, we find that $S=\partial_+C(N)$.
\end{proof}

\begin{defi}
  A spacelike $\AdS$ structure $(U_i, \phi_i)_{i\in I}$ on a manifold $M$ is {\em convex ruled} if for all $i\in I$, $\phi_i(U_i)$ is contained in a past-convex ruled hypersurface. 
\end{defi}

\begin{lemma} \label{lm:AdSquasifucsianConvexRuledSpacelikeStructure}
 Let $M$ be a closed manifold of dimension $d$. There is a one-to-one correspondence between convex ruled spacelike $\AdS$ structures on $M$ and quasifuchsian $\AdS$ spacetimes homeomorphic to $M\times \R$.
\end{lemma}

\begin{proof}
It follows from Lemma \ref{lm:unique}, Theorem \ref{pr:spacelike}, and characterization $(iii)$ in  Theorem \ref{thm-characterization AdS convex}.
\end{proof}

\subsection{Convex ruled hyperbolic embedding structures and hyperbolic ends} \label{subsec:hyperbolic ends}

In this section we outline the analog in hyperbolic manifolds of the notion of spacelike $\AdS$ structures developed in Sections \ref{ssc:spacelike structures} and \ref{ssc:convex}.

\begin{defi}
Let $\mathcal H$ be an oriented Lipschitz manifold. A map $\phi:\mathcal H\to \HH^{d+1}$ is called a \emph{Lipschitz embedding} if $(x,y)\mapsto d_{\HH^{d+1}}(\phi(x),\phi(y))$  is a Lipschitz distance on $\mathcal H$. It is called a \emph{Lipschitz immersion} if every point in $\mathcal H$ has a neighbourhood $U$ such that the restriction of $\phi$ to $U$ is a Lipschitz embedding.

A Lipschitz immersion $\phi:\mathcal H\to \HH^{d+1}$ is called \emph{ruled} if for every point $x\in \mathcal H$ there is an injective continuous curve $\gamma:(-\varepsilon,\varepsilon)\to \mathcal H$ such that $\gamma(0)=x$ and $\phi\circ \gamma$ is a geodesic segment in $\HH^{d+1}$.

A Lipschitz embedding $\phi:\mathcal H\to \HH^{d+1}$ is called \emph{convex} if there is an open convex set $V\subset  \HH^{d+1}$   such that $\phi(\mathcal H)$ is an open subset of $\partial V$ and the orientations induced on $\partial V$ from $V$ and $\mathcal H$ coincide. A Lipschitz immersion $\phi:\mathcal H\to \HH^{d+1}$ is called \emph{locally convex} if every point in $\mathcal H$ has a neighbourhood $U$ such that the restriction of $\phi$ to $U$ is a convex Lipschitz embedding.
\end{defi}

We can use convex ruled Lipschitz embeddings to define convex ruled hyperbolic embedding structures, the Riemannian counterpart of convex ruled spacelike $\AdS$ structures.

\begin{defi}
  Let $M$ be a $d$-dimensional oriented manifold. A {\em hyperbolic embedding atlas} on $M$ is the data of an atlas $(U_i, \phi_i)_{i\in I}$ where $(U_i)_{i\in I}$ is an open cover of $M$ and $\phi_i: U_i \to \HH^{d+1}$ is a Lipschitz embedding, such that for all $i,j \in I$ and all $x\in U_i\cap U_j$, there exists an orientation preserving isometry $g$ of $\HH^{d+1}$ such that
  \[\phi_j = g\circ \phi_i\]
  in a neighbourhood of $x$.

A \emph{hyperbolic embedding structure} is a maximal hyperbolic embedding atlas. It is called \emph{convex ruled} if the local charts $\phi_i$ are locally convex and ruled.

Two  hyperbolic embedding structures $(U_i, \phi_i)_{i\in I}$ and $(V_j, \psi_j)_{j\in J}$ on $M$ are \emph{isomorphic} if for every $(i,j)\in I\times J$ and every $x\in U_i\cap V_j$, there exists an orientation preserving isometry $g$ of $\HH^{d+1}$ such that
\[ \psi_j = g\circ \phi_i \]
in a neighbourhood of $x$.
\end{defi}

The  method  used for Corollary \ref{cor-developing map spacelike AdS structure} also provides a developing map and a holonomy representation for convex ruled hyperbolic embedding structures. 

\begin{lemma} \label{lem - developing map hyperbolic embedding structure}
Let $(U_i,\phi_i)$ be a hyperbolic embedding atlas on $M$. Then there exists a representation $\rho:\pi_1(M)\to \SO_\circ(d+1,1)$ and a $\rho$-equivariant Lipschitz immersion $\rho:\tilde M\to \HH^{d+1}$ such that, for all $x\in \tilde M$ and all $U_i$ containing $\pi(x)$, there exists $g\in \SO_\circ(d+1,1)$ such that \[ \phi_i\circ \pi=g\circ\phi\] in a neighbourhood of $x$. (Here $\pi:\tilde M\to M$ denotes the universal covering map.)

Moreover, if another pair $(\rho',\phi')$ satisfies the same properties, then there is a unique $g\in\SO_\circ(d+1,1)$ such that $\rho'=g\rho g^{-1}$ and $\phi'=g\circ\phi$.
\end{lemma}

Just as in the $\AdS$ case, the converse also holds: a pair $(\rho,\phi)$ where $\rho:\pi_1(\tilde M)\to \SO_\circ(d+1,1)$ is a representation and $\phi:\tilde M\to \HH^{d+1}$ is a $\rho$-equivariant Lipschitz immersion determines a hyperbolic embedding structure on $M$. Finally, the map $\phi$ is convex ruled if and only if the corresponding hyperbolic embedding structure is convex ruled.

 The hyperbolic manifolds associated to convex ruled hyperbolic embedding structures, hyperbolic ends, require a precise definition. In simple terms, a hyperbolic end is a hyperbolic manifold with compact concave boundary, which is maximal in the sense of inclusion for this condition. In order to give a precise definition, we need to define what we mean by a hyperbolic manifold with concave boundary.


\begin{defi}
Let $\phi:\mathcal H\to \HH^{d+1}$ be a convex Lipschitz embedding, and  $V\subset \HH^{d+1}$ a convex open set  such that $\phi(\mathcal H)$ is an open subset of $\partial V$. We define the set $\mathcal N(\phi)$ of  \emph{normal vectors} to $\phi$ to be the set of pairs $(x,v)$ where $x\in \mathcal H$ and $v\in T_{\phi(x)}\HH^d$ is a unit vector pointing outside of $V$ such that $v^\perp$ is a support hyperplane of $V$.

The \emph{concave development} of a convex Lipschitz embedding $\phi:\mathcal H\to \HH^{d+1}$ is the  set
\[\mathcal W(\phi)=\set{\exp_{\phi(x)}(tv)}{(x,v)\in \mathcal N(\phi), t\geq 0}. \]
\end{defi}

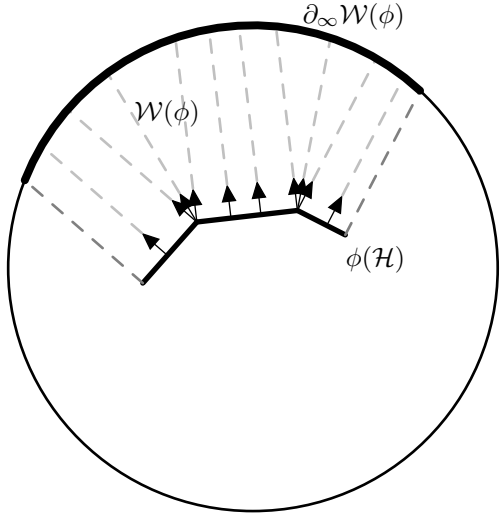
\begin{figure}[h]
\definecolor{cqcqcq}{rgb}{0.7529411764705882,0.7529411764705882,0.7529411764705882}
\definecolor{yqyqyq}{rgb}{0.5019607843137255,0.5019607843137255,0.5019607843137255}
\begin{tikzpicture}[line cap=round,line join=round,>=triangle 45,x=1.0cm,y=1.0cm,scale=0.45]
\clip(-12,-10.84) rectangle (28.24,6.3);
\draw [line width=2.pt] (3.22,-3.96)-- (4.8,-2.16);
\draw [line width=2.pt] (4.8,-2.16)-- (7.74,-1.82);
\draw [line width=2.pt] (7.74,-1.82)-- (9.14,-2.54);
\draw [line width=1.pt] (6.44,-3.52) circle (7.156870824599253cm);
\draw [->,line width=0.5pt] (3.938662575831532,-3.1412704832299) -- (3.187121379911219,-2.481584322366514);
\draw [->,line width=0.5pt] (4.8,-2.16) -- (4.048458804079692,-1.5003138391366186);
\draw [->,line width=0.5pt] (4.8,-2.16) -- (4.288516140232317,-1.3007071155902934);
\draw [->,line width=0.5pt] (4.8,-2.16) -- (4.685119399862813,-1.1666206929313867);
\draw [->,line width=0.5pt] (5.79821444880811,-2.044560233811307) -- (5.683333848670924,-1.051180926742693);
\draw [->,line width=0.5pt] (6.67344780345237,-1.9433427710293179) -- (6.5585672033151825,-0.9499634639607037);
\draw [->,line width=0.5pt] (7.74,-1.82) -- (7.625119399862814,-0.8266206929313854);
\draw [->,line width=0.5pt] (7.74,-1.82) -- (8.197348012620763,-0.9307121976818494);
\draw [->,line width=0.5pt] (7.74,-1.82) -- (7.927112107889996,-0.8376614335775253);
\draw [->,line width=0.5pt] (8.599522272433829,-2.262040025823112) -- (9.05687028505459,-1.3727522235049612);
\draw [line width=1.pt,dash pattern=on 5pt off 5pt,color=yqyqyq] (3.22,-3.96)-- (-0.23181306210852948,-0.9300752010380687);
\draw [line width=1.pt,dash pattern=on 5pt off 5pt,color=cqcqcq] (3.187121379911219,-2.481584322366514)-- (0.2593399506261076,0.08835715445041648);
\draw [line width=1.pt,dash pattern=on 5pt off 5pt,color=cqcqcq] (4.048458804079692,-1.5003138391366186)-- (1.0220622780636754,1.1561897781441062);
\draw [line width=1.pt,dash pattern=on 5pt off 5pt,color=cqcqcq] (4.685119399862813,-1.1666206929313867)-- (4.172261723991295,3.268089798428205);
\draw [line width=1.pt,dash pattern=on 5pt off 5pt,color=cqcqcq] (5.683333848670924,-1.051180926742693)-- (5.1546360602336625,3.5204999497442193);
\draw [line width=1.pt,dash pattern=on 5pt off 5pt,color=cqcqcq] (6.5585672033151825,-0.9499634639607037)-- (6.029479695508713,3.62508733883641);
\draw [line width=1.pt,dash pattern=on 5pt off 5pt,color=cqcqcq] (7.625119399862814,-0.8266206929313854)-- (7.112596413198242,3.6051957211681462);
\draw [line width=1.pt,dash pattern=on 5pt off 5pt,color=cqcqcq] (8.197348012620763,-0.9307121976818494)-- (10.045266793814678,2.6624632101952073);
\draw [line width=1.pt,dash pattern=on 5pt off 5pt,color=cqcqcq] (9.05687028505459,-1.3727522235049612)-- (10.851041043154133,2.115913139466372);
\draw [line width=1.pt,dash pattern=on 5pt off 5pt,color=yqyqyq] (9.14,-2.54)-- (11.325637763451418,1.7098512067110894);
\draw [line width=1.pt,dash pattern=on 5pt off 5pt,color=cqcqcq] (4.288516140232317,-1.3007071155902934)-- (2.184408657375351,2.234193455609408);
\draw [line width=1.pt,dash pattern=on 5pt off 5pt,color=cqcqcq] (7.927112107889996,-0.8376614335775253)-- (8.70907488085043,3.267643124464751);
\draw [shift={(6.44,-3.52)},line width=3pt]  plot[domain=0.8194133927423888:2.771309408192698,variable=\t]({1.*7.156870824599255*cos(\t r)+0.*7.156870824599255*sin(\t r)},{0.*7.156870824599255*cos(\t r)+1.*7.156870824599255*sin(\t r)});
\draw (8.88,-2.6) node[anchor=north west] {$\phi(\mathcal H)$};
\draw (7.62,4.58) node[anchor=north west] {$\partial_\infty\mathcal W(\phi)$};
\draw (2.7,1.7) node[anchor=north west] {$\mathcal W(\phi)$};
\end{tikzpicture}
\caption{The concave development}
\label{fig:concave_development}
\end{figure}

\begin{defi} \label{def:concave_hyperbolic_manifold}
A \emph{concave hyperbolic atlas} on a  topological space $X$ is the data of an atlas $(U_i,\phi)_{i\in I}$ where $(U_i)_{i\in I}$ is an open cover of $X$ and the maps $\phi_i:U_i\to \HH^{d+1}$ have the following properties:
\begin{enumerate}
\item Each $\phi_i$ is a homeomorphism from $U_i$ to $\phi_i(U_i)$,
\item For each $i\in I$ there is a convex Lipschitz embedding $\psi_i:\mathcal H_i\to \HH^{d+1}$ such that  $\phi_i(U_i)=\mathcal W(\psi_i)$,
\item For every pair $i,j\in I$ there is an element $g_{ij}\in \SO_\circ(d+1,1)$ such that $\phi_j=g_{ij}\circ \phi_i$ on $U_i\cap U_j$.
\end{enumerate}
A \emph{concave hyperbolic manifold} is the data of a connected Hausdorff second countable  topological space $X$ and a maximal concave hyperbolic atlas on $X$.
\end{defi}

Note that a concave hyperbolic manifold is always a topological manifold with boundary, the interior being a smooth hyperbolic manifold.

\begin{defi}
A \emph{hyperbolic end} of dimension $d+1$ is a concave hyperbolic manifold with non empty compact boundary, and which is maximal (in the sense of inclusion) under this condition.
\end{defi}

Examples of hyperbolic ends arise naturally as the components of the complement of the convex core in a convex cocompact hyperbolic manifold. Let us focus on the quasifuchsian case.

\begin{defi}
A \emph{quasifuchsian hyperbolic manifold} of dimension $d+1$ is a complete hyperbolic manifold $N$ homeomorphic to $M\times \R$, with $M$  a closed manifold of dimension $d$, and which contains a non-empty compact subset $K\subset N$ with the following convexity property: for any $x,y\in K$, any geodesic of $N$ joining $x$ and $y$ is contained in $K$.\\
The smallest such compact set $K$ is called the \emph{convex core} of $N$, denoted by $C(N)$. A quasifuchsian hyperbolic manifold is called \emph{Fuchsian} if its convex core is a totally geodesic hypersurface.
\end{defi}

If $N$ is a quasifuchsian hyperbolic manifold, then $N\setminus C(N)$ has two connected components, and the closure of each component is a hyperbolic end. When $d=1$, a hyperbolic end is just a funnel (the quotient by a hyperbolic isometry of a half hyperbolic plane bounded by the axis of this isometry). When $d\geq 2$, a hyperbolic end cannot always be obtained as an end of a quasifuchsian hyperbolic manifold. Indeed, a hyperbolic end with boundary homeomorphic to $M$ induces a holonomy representation $\rho:\pi_1(M)\to \SO_\circ(d+1,1)$, which is faithful and discrete for a quasifuchsian manifold, but may fail to be either for a hyperbolic end.\\

A hyperbolic end $E$ actually comes with two boundaries. The first one,  called the \emph{pleated boundary} of $E$ and denoted by $\partial_0E$, is the one already mentioned after Definition \ref{def:concave_hyperbolic_manifold},  obtained when $E$ is seen as a manifold with boundary. We will discuss its geometry further in the next section.

The other boundary is the ideal boundary $\partial_\infty E$. It is homeomorphic to $M$ and locally modelled on  $\partial_\infty\HH^{d+1}=\mathbb S^d$ with the action of $\SO_\circ(d+1,1)$ by M\"obius transformations. In other words, $\partial_\infty E$ is equipped with a conformally flat structure. Moreover, this correspondence between hyperbolic ends homeomorphic to $M\times [0,+\infty)$ and conformally flat structures on $M$ is a bijection, according to the following:

\begin{theorem}[Kulkarni--Pinkall \cite{kulkarni-pinkall}] \label{tm:kp}
Let $M$ be an oriented closed manifold of dimension $d$ whose fundamental group is not virtually abelian.\footnote{This topological condition ensures that conformally flat structures on $M$ are \emph{hyperbolic} as defined in \cite{kulkarni-pinkall}, and is satisfied by closed hyperbolic manifolds as well as Gromov--Thurston manifolds.} Then every conformally flat  structure on $M$ is the ideal boundary of a unique hyperbolic end with interior homeomorphic to $M\times (0,+\infty)$.
\end{theorem}


When $d=2$, the group $\SO_\circ(3,1)$ is isomorphic to ${\rm PSL}(2,\C)$, and a theorem of Gallo--Kapovich--Marden \cite{gallo-kapovich-marden} states that any non elementary representation of the fundamental group of a closed surface into $\mathrm{SL}(2,\C)$ descends to the holonomy representation of a conformally flat structure, thus providing an abundance of non quasifuchsian hyperbolic ends.

The bridge between convex ruled hyperbolic embedding structures and hyperbolic ends is the pleated boundary. The following lemma can be inferred from \cite[Theorems 5.9 and 8.6]{kulkarni-pinkall} (see also \cite{smith_moduli}), but we can also give a direct proof without going through conformally flat structures.

\begin{lemma}\label{lem : structure pleated boundary} If $E$ is a hyperbolic end, then the pleated boundary $\partial_0 E$ carries a convex ruled hyperbolic embedding structure.
\end{lemma}

\begin{proof}
Following  Definition \ref{def:concave_hyperbolic_manifold}, we see that the restriction of charts of a concave hyperbolic atlas to $\partial_0E$ are  convex Lipschitz embeddings. The fact that $\partial_0E$ is ruled comes from the maximality of an end: if $\partial_0E$ were strictly concave, we could ``push'' it to obtain a larger concave hyperbolic manifold (see Figure \ref{fig:convex_ruled_pleated_boundary}). More precisely, consider a point $x\in \partial_0E$ at which the property fails, and a concave hyperbolic atlas for which $x$ is in a unique chart $(U_i,\phi_i)$. Then consider a support hyperplane $H\subset \HH^{d+1}$ to $V_i$ at $\phi_i(x)$ (following the notations of Definition \ref{def:concave_hyperbolic_manifold}). The failure of the ruling at $x$ means that one can push $H$ slightly to a hyperplane $H'$ that cuts the image of $\phi$ in an arbitrarily small neighborhood of $\phi_i(x)$. In particular, we can assume that this neighborhood is not contained in any other chart. Replacing $V_i$ with its intersection with the half-space bounded by $H'$ that does not contain $\phi_i(x)$ and keeping the other charts, we get a concave hyperbolic atlas on a larger manifold, thus contradicting the maximality of $E$.
\end{proof}

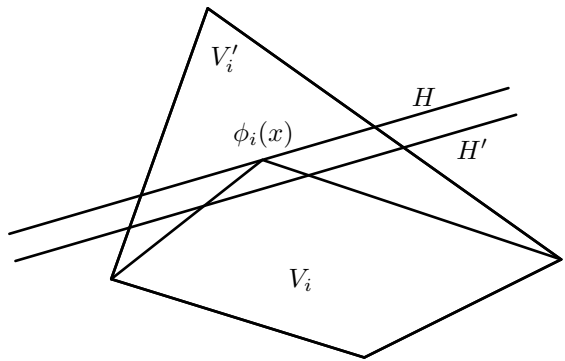
\begin{figure}[h]
\begin{tikzpicture}[line cap=round,line join=round,>=triangle 45,x=1.0cm,y=1.0cm,scale=0.5]
\clip(-4.3,-6.1) rectangle (28.24,4);
\draw[line width=1.pt] (2.04,-3.94) -- (8.7,-6.02) -- (13.88,-3.42) -- (4.58,3.24) -- cycle;
\draw [line width=1.pt] (2.04,-3.94)-- (8.7,-6.02);
\draw [line width=1.pt] (8.7,-6.02)-- (13.88,-3.42);
\draw [line width=1.pt] (13.88,-3.42)-- (4.58,3.24);
\draw [line width=1.pt] (4.58,3.24)-- (2.04,-3.94);
\draw [line width=1.pt] (13.88,-3.42)-- (6.02,-0.78);
\draw [line width=1.pt] (6.02,-0.78)-- (2.04,-3.94);
\draw (5,0.5) node[anchor=north west] {$\phi_i(x)$};
\draw (6.46,-3.42) node[anchor=north west] {$V_i$};
\draw (4.4,2.52) node[anchor=north west] {$V'_i$};
\draw (9.7,1.38) node[anchor=north west] {$H$};
\draw (10.88,0.02) node[anchor=north west] {$H'$};
\draw [line width=1.pt] (-0.64,-2.74)-- (12.492680072701319,1.124872814188377);
\draw [line width=1.pt] (-0.48,-3.46)-- (12.694908583901139,0.41730042409102674);
\end{tikzpicture}
\caption{Proof of Lemma \ref{lem : structure pleated boundary}.}
\label{fig:convex_ruled_pleated_boundary}
\end{figure}

We now wish to see how any convex ruled hyperbolic embedding structure can be obtained as the pleated boundary of a hyperbolic end. 

\begin{lemma} \label{lem:convex ruled hyperbolic embedding structure to hyperbolic end}
Any convex ruled hyperbolic embedding structure can be obtained as the pleated boundary of a hyperbolic end, unique up to isometry.
\end{lemma}

This could be obtained by   \cite[Theorem 10.6]{kulkarni-pinkall}, but we provide here a direct construction without going through conformally flat structures, inspired by  \cite[Section 4]{smith_moduli}.


\begin{proof}[Proof of Lemma \ref{lem:convex ruled hyperbolic embedding structure to hyperbolic end}]
Let  $(U_i,\phi_i)_{i\in I}$ be  a  convex ruled hyperbolic embedding atlas on a compact manifold $M$.   We obtain a concave hyperbolic manifold with pleated boundary isometric to $M$ by considering:
  \[ E=\bigsqcup_{i\in I}\mathcal W(\phi_i)/\sim \]
  where we identify $\mathcal W(\phi_i\vert_{U_i\cap U_j})$ with $\mathcal W(\phi_j\vert_{U_i\cap U_j})$ through the element $\gamma_{ij}\in\SO_\circ(d+1,1)$ such that $\phi_i\vert_{U_i\cap U_j} =\gamma_{ij}\circ \phi_j\vert_{U_i\cap U_j} $.
  
  Consider a concave hyperbolic manifold with compact pleated boundary $E'\supset E$, and let $(x,x')\in \partial_0E\times\partial_0E'$ be maximising the distance. In a chart, we find that $\partial_0E'$ must be (locally around $x'$) included in the $r$-neighbourhood of a geodesic where $r=d(x,x')$ (because $\partial_0E$ is ruled). But for $r>0$, the $r$-neighbourhood of a geodesic in $\HH^{d+1}$ is strictly convex, thus $r=0$, i.e. $E'=E$ and $E$ is maximal.
  
  Now consider another hyperbolic end $E'$ inducing  the same convex ruled hyperbolic embedding structure on the pleated boundary $\partial_0E'\approx M$. If $(V_j,\psi_j)_{j\in J}$ is a concave hyperbolic atlas on $E'$, then whenever $V_j\cap \partial_0E'\neq\emptyset$, $\psi_j(V_j)$ contains a neighbourhood of $\psi_j(V_j\cap \partial_0E')$ in $\mathcal W(\psi_j\vert_{V_j\cap \partial_0E'})$. So $E'$ contains a neighbourhood of $\partial_0E$ in $E$. By maximality of $E'$, we find that $E\subset E'$, and by maximality of $E$ we must have $E=E'$, hence the uniqueness.
  
  \end{proof}
  
  Note that we can also obtain the flat conformal structure from this construction. If $\phi:\mathcal H\to \HH^{d+1}$ is a convex Lipschitz embedding, we can consider the ideal boundary $\partial_\infty \mathcal W(\phi)$ of its concave development:
  
  \[ \partial_\infty \mathcal W(\phi)=\overline{\mathcal W(\phi)}\cap \partial_\infty \HH^{d+1}=\set{ \lim_{t\to+\infty}\exp_{\phi(x)}(tv)}{(x,v)\in \mathcal N(\phi)}\subset \partial_\infty \HH^{d+1}=S^d.\]

One gets a conformally flat manifold by setting:
\[ \partial_\infty E= \bigsqcup \partial_\infty\mathcal W(\phi_i)/\sim\]
  where we identify $\partial_\infty\mathcal W(\phi_i\vert_{U_i\cap U_j})$ with $\partial_\infty\mathcal W(\phi_j\vert_{U_i\cap U_j})$ through the element $\gamma_{ij}\in\SO_\circ(d+1,1)$ such that $\phi_i\vert_{U_i\cap U_j} =\gamma_{ij}\circ \phi_j\vert_{U_i\cap U_j} $.\\

As a consequence of Lemma \ref{lem : structure pleated boundary} and Lemma  \ref{lem:convex ruled hyperbolic embedding structure to hyperbolic end}, we get the following hyperbolic version of Lemma \ref{lm:AdSquasifucsianConvexRuledSpacelikeStructure}:

\begin{lemma} \label{lm:HyperbolicEndConvexRuledEmbeddingStructure}
Let $M$ be a closed manifold. There is a one-to-one correspondence between convex ruled hyperbolic embedding structures on $M$ up to equivalence and hyperbolic ends $E$ with pleated boundary $\partial_0E\approx M$ up to isometry.   
\end{lemma}



\section{Spherical and de Sitter polygons} 
\label{sc:polygons}

This section is devoted to the study of polygons in the sphere $\bS^2$ and spacelike polygons in the de Sitter space $\dS^2$. We construct a moduli space of such polygons, which is a smooth manifold, and describe various subsets of this moduli space, namely equilateral polygons, and equilateral polygons with a central symmetry.

As we will see in the next sections, spherical and de Sitter polygons parametrize bendings of Gromov--Thurston manifolds in the hyperbolic and anti-de Sitter space respectively. Once this is established, the results of the present section will readily prove the main theorems of the paper.

\subsection{Spherical polygons and their deformations}
\label{ssc:polygons}

Let us start with the more familiar setting of spherical polygons. We will use the following definition:\\

\begin{defi}
A spherical $k$-gon, $k\geq 3$, is a tuple $(v_i)_{i\in \Z/k\Z}$ of pairwise distinct points in $\mathbb S^2$ such that
\begin{enumerate}
\item $0 < d(v_i,v_{i+1}) < \pi$ for all $i\in \Z/k\Z$,
\item $(v_i,v_{i+1}) \cap [v_j,v_{j+1}] = \emptyset$ for $i\neq j$.
\end{enumerate}
\end{defi}

\begin{remark}
Condition 1 guarantees that two consecutive vertices $v_i, v_{i+1}$ are never antipodal, so that they are joined by a unique geodesic segment. Condition 2 and the fact that the vertices are pairwise distinct guarantee that our polygons do not have ``crossings'', so that the union of all edges $\bigcup_{i\in \Z/k\Z}[v_i,v_{i+1}]$ is an embedded topological circle.
\end{remark}

\begin{remark}
Our polygons are \emph{labelled}, meaning that a polygon is considered different from another one with the same vertices permuted. In particular, the polygon $(v_1,\ldots, v_k)$ is different from $(v_k, \ldots, v_1)$, so our polygons are also \emph{oriented}.
\end{remark}

The set $\mathcal U_k$ of spherical $k$-gons is clearly an open subset of $(\bS^2)^k$ and thus inherits a structure of $2k$-dimensional manifold. The group $\SO(3)$ acts smoothly on $\mathcal U_k(\bS)$ and this action is free (since the stabilizer of a given polygon fixes two non-antipodal points on the sphere). The quotient space
\[\mathcal P_k(\bS) \equaldef \SO(3) \backslash \mathcal U_k\]
is thus a manifold of dimension $2k-3$ which we call the \emph{moduli space of (labelled) $k$-gons}.

Given $p= (v_1,\ldots, v_k)$ a spherical polygon in $\bS^2$, let us introduce the following auxiliary vectors:
\begin{itemize}
\item $u_i^+$ is the unit vector in $T_{v_i} \bS^2 = v_i^\perp$ directing the edge $[v_i,v_{i+1}]$,
\item $u_i^-$ is the unit vector in $T_{v_i} \bS^2 = v_i^\perp$ directing the edge $[v_i,v_{i-1}]$,
\item $w_i = v_i \times u_i^+$ is the unit vector completing $(v_i, u_i^+)$ into an oriented orthonormal basis.
\end{itemize}
(These vectors depend on $p$ in the same way the vertices $v_i$ do, but we omit to write this dependence in order to lighten notations.)

Finally, we define $ l_i(p)$ to be the length of the edge $[v_i,v_{i+1}]$ and $\theta_i(p)$ to be the oriented angle at $v_i$ between $u_i^+$ and $-u_i^-$, i.e. $\theta_i(p) \in (-\pi,\pi)$ is such that
\[-u_i^- = \cos(\theta_i(p)) u_i^+ + \sin(\theta_i(p)) w_i~.\]
With this convention, $\theta_i = 0$ if and only if $v_{i-1}$, $v_i$ and $v_{i+1}$ are aligned.

The functions $l_i: \mathcal U_k \to (0,\pi)$ and $\theta_i: \mathcal U_k \to (-\pi, \pi)$ are clearly smooth and $\SO(3)$-invariant. They thus factor to smooth functions on $\mathcal P_k(\bS)$.

\begin{theorem} \label{tm:polygons}
The map 
\begin{align*}
\Phi: \mathcal P_k(\bS) & \to (0,\pi)^k\times (-\pi,\pi)^k\\
p & \mapsto (l_1(p),\ldots, l_k(p), \theta_1(p),\ldots, \theta_k(p))
\end{align*}
is an embedding, and the image of $\mathrm d \Phi$ at a point $p$ is the set of tuples $(\dot l_1,\ldots, \dot l_k, \dot \theta_1, \ldots , \dot \theta_k)$ such that
  \begin{equation} \label{eq:variation polygons S2}    \sum_{i=1}^k \dot \theta_i v_i - \dot l_i w_i =0~.
  \end{equation}
\end{theorem}

%

\begin{proof}[Proof of Theorem \ref{tm:polygons}]
 
It is well-known that a polygon is characterized up to isometry by its lengths and angles, so that $\Phi$ is a homeomorphism onto its image. While the fact that $\Phi$ is an immersion is also quite intuitive, let us prove it with a little more care. 

Fix $p\in \mathcal U_k(\bS)$ and $\dot p = (\dot v_1, \ldots \dot v_k)$ a tangent vector at $p$. Assume that $\mathrm d l_i(\dot p) = \mathrm d \theta_i(\dot p) = 0$. We need to prove the existence of $a \in \so(3)$ such that $\dot v_i = a v_i$ for all $i$.

For each $i$, note that $(v_i,v_{i+1}, w_i)$ is a basis of $\R^3$ for each $i$. Denoting by $\dot w_i$ the first order variation of $w_i$, we have
\begin{align*}
\langle \dot w_i, w_i\rangle &=0\\
\langle \dot w_i, v_i\rangle &= - \langle w_i, \dot v_i\rangle\\
\langle \dot w_i, v_{i+1} \rangle &= - \langle w_i, \dot v_{i+1}\rangle
\end{align*}
since $w_i$ is a unit vector orthogonal to $v_i$ and $v_{i+1}$. Define $a_i\in \End(\R^3)$ by 
\begin{align*}
a_i v_i&=\dot v_i\\
a_i v_{i+1} &= \dot v_{i+1}\\
a_i w_i &= \dot w_i~.
\end{align*}
The identity
\[\langle a_i v, w\rangle = - \langle v, a_i w\rangle\]
is satisfied on the basis $(v_i,v_{i+1}, w_i)$, hence $a_i\in \so(3)$.
 
Let us now prove that $a_i = a_{i-1}$. By construction, $a_i v_i = \dot v_i = a_{i-1} v_i$. Moreover, we have
\[w_{i-1} = \cos(\theta_i) w_i - \sin(\theta_i) u_i^+~.\]
Since $\dot \theta_i = 0$, we deduce that
\begin{align*}
a_{i-1} w_{i-1} & = \dot w_{i-1} \\
&= \cos (\theta_i) \dot w_i - \sin (\theta_i) \dot u_i^+\\
&= \cos (\theta_i) a_i w_i - \sin(\theta_i) a_i u_i^+\\
&= a_i w_{i-1}~.
\end{align*}
The endomorphism $a_i-a_{i-1}$ is in $\so(3)$ and has a kernel of dimension at least $2$, hence $a_i= a_{i-1}$.

By an immediate induction, we conclude that all the $a_i$ are equal to the same $a\in \so(3)$, which then satisfies $\dot v_i = av_i$ for all $i$. This proves that the map $\Phi$ is an embedding.\\
 
Let us now characterize the image of $\mathrm d \Phi$. Note that, since the $v_i$ and $w_i$ span $\R^3$, the space of tuples $(\dot l_1,\ldots, \dot l_k, \dot \theta_1, \ldots , \dot \theta_k)$ satisfying Equation \eqref{eq:variation polygons S2} has dimension $k-3$. By equality of dimension, it is thus enough to verify that the equation is satisfied on the image of $\mathrm d \Phi$, and by linearity it suffices to prove it for first order deformations where only one vertex, say $v_2$, is moving. 

Fix a polygon $p=(v_1,\ldots, v_k)$ and assume first that no three consecutive vertices are aligned. Then $(u_i^-, u_i^+)$ form a basis of $T_{v_i} \bS^2$ for each $i$, and it is enough to verify that the relation \eqref{eq:variation polygons S2} is satisfied for a first order variation of $p$ when only one of the $v_i$ moves in the direction $u_i^-$ or $u_i^+$. Let us thus prove that \eqref{eq:variation polygons S2} holds when $\dot v_2 = u_2^+$ and $\dot v_i = 0$ for $i\neq 2$ (this is enough by symmetry of the problem).

We can compute first order variations of $\theta_i$ and $l_i$ and get
  \begin{enumerate}
  \item $\dot \theta_1 = \frac{\sin(\theta_2)}{\sin(l_1)}$,
  \item $\dot \theta_2= - \sin(\theta_2) \cotan(l_1)$,
  \item $\dot l_1=\cos(\theta_2)$,
  \item $\dot l_2=-1$,
  \item $\dot l_i= \dot \theta_i = 0$ for $i\notin \{1,2\}$.
  \end{enumerate}

Writing coordinates in the orthonormal basis $(v_2, - u_2^-, w_1)$, we have:
  \begin{eqnarray*}
    v_1 & = & (\cos(l_1), -\sin(l_1), 0)~, \\
    v_2 & = & (1,0,0)~, \\
    w_1 & = & (0,0,1)~, \\
    w_2 & = & (0,\sin(\theta_2),\cos(\theta_2))~. 
  \end{eqnarray*}

And we conclude that
  $$ \sum_{i=1}^k \dot \theta_i v_i - \dot l_i w_i = \dot \theta_1 v_1 + \dot \theta_2 v_2 - \dot l_1 w_1 - \dot l_2 w_2 = $$
  $$ \begin{array}{cclcrc}
       \frac{\sin(\theta_2)}{\sin(l_1)} & \cdot & (\cos(l_1), & -\sin(l_1),& 0) & + \\
       (-\sin(\theta_2) \cotan(l_1)) & \cdot  & (1, & 0, & 0) & + \\
       -\cos(\theta_2) & \cdot & (0, & 0, & 1) & + \\
       1 & \cdot & (0, & \sin(\theta_2), & \cos(\theta_2)) & \\
  \end{array} $$
  which is equal to $0$.

We deduce that Equation \eqref{eq:variation polygons S2} holds on the image of $\mathrm d \Phi$ at every $p$ with no three consecutive vertices aligned, and conclude that it holds on all $\mathcal P_k(\bS^2)$ by density.
\end{proof}

\subsection{Spacelike polygons in $\dS^2$ and their deformations}

We know duplicate the above construction for spacelike polygons in $\dS^2$. The results and their proof are formally the same, and we will only stress out the additional technicalities.

The two-dimensional de Sitter space $\dS^2$ is the space of unit spacelike vectors in the $2+1$ dimensional Minkowski space $\R^{2,1}$. More precisely, we equip $\R^3$ with the bilinear symmetric form
\[ \langle x,y\rangle = x_0y_0 + x_1y_1-x_2y_2~. \]
 and consider the Lorentzian submanifold
\[ \dS^{2} = \{ x\in \R^{2,1}~|~\langle x,x\rangle =1 \}~. \]

A timelike vector in $\R^{2,1}$ is \emph{future pointing} if its third coordinate is positive. This defines an orientation of time in $\dS^2$. One can then define an orientation of space\marginnote{DM - isn't it more of an orientation of each spacelike geodesic?} in the following way: given $v\in \dS^2$ and $u\in T_v \dS^2$ spacelike, let $w$ be the unit future pointing vector orthogonal to $v$ and $u$. We say that $u$ is \emph{positive} if $(v,u,w)$ is a direct basis of $\R^{2,1}$. The group of isometries of $\dS^2$ preserving the orientation of time and space is $\SO_\circ (2,1)$, acting linearly on $\dS^2\subset \R^{2,1}$. Geodesics in $\dS^2$ are intersections of $\dS^2$ with linear planes in $\R^{2,1}$.


\begin{defi} \label{def:spacelikepolygon}
A spacelike de Sitter $k$-gon is a tuple $(v_i)_{i\in \Z/k\Z}$ of pairwise distinct points in $\dS^2$ such that
\begin{enumerate}
\item $v_i$ and $v_{i+1}$ are joined by a spacelike segment of length in $(0,\pi)$ for all $i\in \Z/k\Z$,
\item the vector directing the segment $[v_i,v_{i+1}]$ at $v_i$ is positive for all $i$,
\item $(v_i,v_{i+1}) \cap [v_j,v_{j+1}] = \emptyset$ for $i\neq j$.
\end{enumerate}
\end{defi}

\begin{remark}
Condition (1) guarantees that two consecutive vertices $v_i, v_{i+1}$ are never antipodal, so that they are joined by a unique geodesic segment. Condition (3) ensures that our polygons do not have ``crossings'', so that the union of all edges $\bigcup_{i\in \Z/k\Z}[v_i,v_{i+1}]$ is an embedded topological circle.
Finally, by Condition (2), this circle is ``positively oriented''. In particular, its homology class is the positive generator of $\mathrm H_1(\dS^2) = \Z$.
\end{remark}

\begin{figure}[h] 
\includegraphics[scale=1,angle=180]{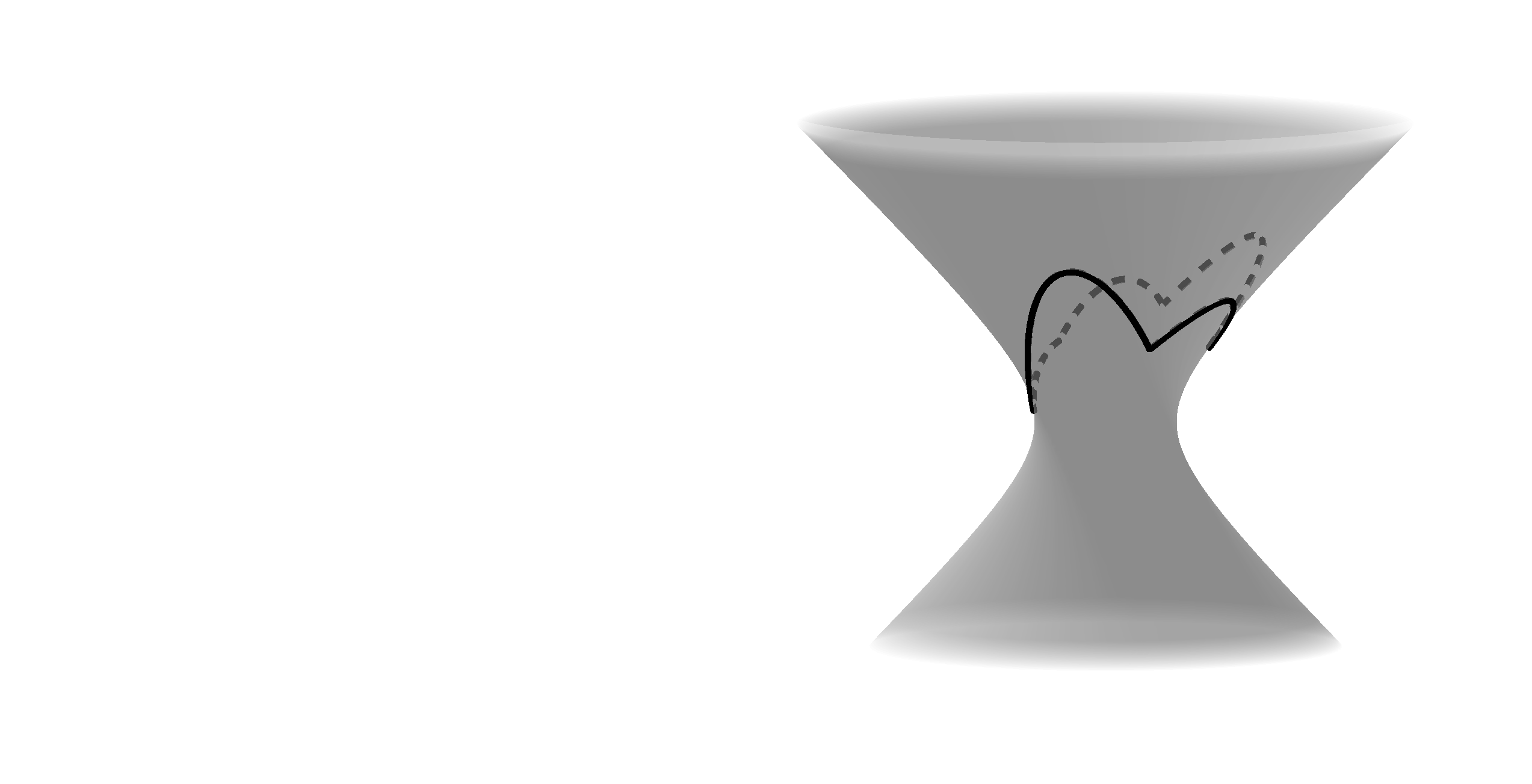}
\caption{A spacelike polygon in $\dS^2$.}
\label{fig:spacelike_polygon}
\end{figure}

%

As in the Euclidean case, the set $\mathcal U_k(\dS)$ of spacelike de Sitter $k$-gons is an open subset of $(\dS^2)^k$ on which the group $\SO_\circ (2,1)$ acts smoothly and freely. Since $\SO_\circ(2,1)$ is not compact anymore, one also needs to remark that this action is proper, which is easy because any element of $\SO_\circ (2,1)$ is entirely characterized by the image of $2$ independent vectors spanning a non-isotropic plane, such as two consecutive vertices of a spacelike polygon. We thus get that the quotient space
\[\mathcal P_k(\dS) \equaldef \SO_\circ(2,1) \backslash \mathcal U_k\]
is a manifold of dimension $2k-3$ which we call the \emph{moduli space of (labelled) spacelike de Sitter $k$-gons}.

Given $p= (v_1,\ldots, v_k)$ a spacelike polygon in $\dS^2$, we introduce again the auxiliary vectors:
\begin{itemize}
\item $u_i^+$ the unit vector in $T_{v_i} \dS^2 = v_i^\perp$ directing the edge $[v_i,v_{i+1}]$,
\item $u_i^-$ the unit vector in $T_{v_i} \dS^2 = v_i^\perp$ directing the edge $[v_i,v_{i-1}]$,
\item $w_i = v_i \times u_i^+$ the unit vector completing $(v_i, u_i^+)$ into an oriented orthonormal basis.
\end{itemize}
Note that we have $\langle w_i, w_i \rangle = -1$ since we are in the Minkowski space.

Finally, we define $ l_i(p)$ to be the length of the edge $[v_i,v_{i+1}]$ and $\theta_i(p)$ to be a ``Lorentzian angle'' at $v_i$ between $u_i^+$ and $-u_i^-$, i.e. $\theta_i(p) \in \R$ is such that
\[-u_i^- = \cosh(\theta_i(p)) u_i^+ + \sinh(\theta_i(p)) w_i~.\]
Again, $l_i: \mathcal U_k(\dS) \to (0,\pi)$ and $\theta_i: \mathcal U_k(\dS) \to \R$ factor to smooth functions on $\mathcal P_k(\dS)$.

\begin{theorem} \label{tm:polygons-ds}
The map 
\begin{align*}
\Phi: \mathcal P_k(\dS) & \to (0,\pi)^k\times \R^k\\
p & \mapsto (l_1(p),\ldots, l_k(p), \theta_1(p),\ldots, \theta_k(p))
\end{align*}
is an embedding, and the image of $\mathrm d \Phi$ at a point $p$ is the set of tuples $(\dot l_1,\ldots, \dot l_k, \dot \theta_1, \ldots , \dot \theta_k)$ such that
  \begin{equation} \label{eq:polygons-ds}
  \sum_{i=1}^k \dot \theta_i v_i + \dot l_i w_i =0~.
  \end{equation}
\end{theorem}

%

\begin{proof}[Proof of Theorem \ref{tm:polygons-ds}]
 
The proof is almost exactly the same as in the spherical case. 

First, a polygon is characterized up to isometry by its lengths and angles, so that $\Phi$ is a homeomorphism onto its image. 

To prove that $\Phi$ is an immersion, we construct $a_i \in \so(2,1)$ such that $a_i v_i = \dot v_i$ and $a_i v_{i+1} = \dot v_{i+1}$, and we prove that all the $a_i$ are equal by showing that $a_i w_{i-1} = a_{i-1} w_{i-1}$.
 
To characterize the image of $\mathrm d \Phi$, we reduce with the same arguments as in the spherical case to proving that Equation \eqref{eq:polygons-ds} is satisfied at a polygon $p$ where no three consecutive vertices are aligned and along a first order variation where $\dot v_2 = u_2^+$ and $\dot v_i = 0$ for $i\neq 2$.

We obtain similar formulae for the first order variation of the lengths and angles, namely
 \begin{enumerate}
  \item $\dot \theta_1 = \frac{\sinh(\theta_2)}{\sin(l_1)}$,
  \item $\dot \theta_2= - \sinh(\theta_2) \cotan(l_1)$,
  \item $\dot l_1=\cosh(\theta_2)$,
  \item $\dot l_2=-1$,
  \item $\dot l_i = \dot \theta_i = 0$ for $i\notin \{1,2\}$.
  \end{enumerate}

In the orthonormal frame $(v_2, u_2^-, w_1)$, we have:
  \begin{eqnarray*}
    v_1 & = & (\cos(l_1), -\sin(l_1), 0)~, \\
    v_2 & = & (1,0,0)~, \\
    w_1 & = & (0,0,1)~, \\
    w_2 & = & (0,-\sinh(\theta_2),\cosh(\theta_2))~. 
  \end{eqnarray*}

And we compute again that 
  $$ \sum_{i=1}^k \dot \theta_i v_i + \dot l_i w_i = \dot \theta_1 v_1 + \dot \theta_2 v_2 + \dot l_1 w_1 + \dot l_2 w_2 = 0~.$$
 Note the single sign change in the second coordinate of $w_2$, which induces the sign change in Equation \eqref{eq:polygons-ds} compared to Equation \eqref{eq:variation polygons S2}.
\end{proof}

%
%
%
%

\subsection{Convex polygons}

\begin{defi}
A spherical or de Sitter polygon $p$ will be called \emph{convex} if
\[\theta_i(p) \geq 0\]
for all $i$.
\end{defi}

By definition, the union of the edges of a spherical polygon forms an oriented Jordan curve, which separates the sphere into two topological discs. With our (perhaps non-standard) convention, a polygon is convex if and only if the disc \emph{to its right} is convex.

Similarly, a de Sitter polygon separates $\dS^2$ into two cylindrical domains, and the polygon is convex if and only if the domain \emph{in its past} is convex.

\subsection{Equilateral polygons}

\begin{defi}
A spherical or spacelike de Sitter $k$-gon $p$ is called \emph{equilateral} if
\[l_1(p)= \ldots = l_k(p)~.\]
\end{defi}

Let $\mathcal P_k^{\textit{eq}}(\bS)\subset \mathcal P_k(\bS)$ denote the set of (equivalence classes of) equilateral spherical  $k$-gons and $\mathcal P_k^{l}(\bS)\subset \mathcal P_k^{\textit{eq}}(\bS)$ the subset of equilateral $k$-gons $p$with $l_i(p) = l$ for all $i$.

\begin{prop} \label{prop:EquilateralPolygons}
The space $\mathcal P_k^{\textit{eq}}(\bS)$ is a submanifold of $\mathcal P_k(\bS)$ of dimension $k-2$. 

For all $l < \frac{2\pi}{k}$, the subpace $P_k^{l}(\bS)$ is a submanifold of $\mathcal P_k(\bS)$ of dimension $k-3$.
\end{prop}

\begin{proof}
Define
\[\begin{array}{crcl}
L: &\mathcal P_k(\bS) & \to & \R^{k}\\
& p & \to &(l_1(p), \ldots l_k(p))
\end{array}\]
and 
\[\begin{array}{crcl}
D: &\mathcal P_k(\bS) & \to & \R^{k-1}\\
& p & \to &(l_2(p)-l_1(p), \ldots, l_k(p)-l_{k-1}(p))~,
\end{array}\]
so that $\mathcal P_k^l(\bS)= L^{-1}(l,\ldots,l)$ and $\mathcal P_k^{\textit{eq}}(\bS) = D^{-1}(0,\ldots, 0)$. 

Let $p$ be an equilateral polygon of length $l$. By Theorem \ref{tm:polygons}, $(\dot l_1, \ldots, \dot l_k)$ belongs to the image of $\mathrm d L$ if and only if there exist $(\dot \theta_i)_{1\leq i \leq k}$ such that 
\[\sum_{i=1}^k \dot \theta_i v_i = \sum_{i=1}^k \dot l_i w_i~.\]
This is the case as long as the $v_i$ span $\R^3$. Otherwise, all the $v_i$ are aligned along an equator, hence $kl = 2\pi$. We conclude that $L$ is a submersion along $\mathcal P_k^l(\bS)$ for $l< \frac{2\pi}{k}$ and the second part of the theorem follows.

We also deduce that $D$ is a submersion at $p$ unless $p$ is contained in an equator. Assume now that $p$ is contained in an equator. Then all the $w_i$ are equal to a fixed unit vector $w$ orthogonal to this equator. Fix $(\delta_i)_{1\leq i \leq k-1} \in \R^{k-1}$ and set $\dot \theta_i =0$ and $\dot l_i = \sum_{j=1}^{i-1} \delta_j - s$, where
\[s= \frac{1}{k}\sum_{i=1}^k \sum_{j=1}^{i-1} \delta_j~.\]
Then we have
\[\dot l_{i+1} - \dot l_i = \delta_i\]
and
\[\sum_{i=1}^k \dot \theta_i v_i - \sum_{i=1}^k \dot l_i w_i = - \sum_{i=1}^k \dot l_i w = 0~. \]
Hence $(\delta_1, \ldots, \delta_{k-1})$ belongs to the image of $\mathrm d D$. We conclude that $D$ is a submersion along $\mathcal P_k^{\textit{eq}}(\bS)$ and the first part of the theorem follows.
\end{proof}

Similarly, denoting by $\mathcal P_k^{\textit{eq}}(\dS)\subset \mathcal P_k(\dS)$ the set of (equivalence classes of) equilateral spacelike de Sitter $k$-gons and $\mathcal P_k^{l}(\dS) \subset \mathcal P_k^{\textit{eq}}(\dS)$ the subset of equilateral $k$-gons of length $l$, we have

\begin{prop}  \label{prop:EquilateralPolygons-dS}
The space $\mathcal P_k^{\textit{eq}}(\dS)$ is a submanifold of $\mathcal P_k(\dS)$ of dimension $k-2$. 

For all $l > \frac{2\pi}{k}$, the subpace $P_k^{l}(\dS)$ is a submanifold of $\mathcal P_k(\dS)$ of dimension $k-3$.
\end{prop}

The proof is identical to that of Proposition \ref{prop:EquilateralPolygons}.\\

The following proposition guarantees in particular the existence of convex equilateral polygons with any lengths in the appropriated range.

\begin{prop} \label{prop: Existence polygons prescribed lengths}
There exists a smooth $1$-parameter family of spherical $k$-gons $(p_l)_{l \in (0,\frac{2\pi}{k}]}$ such that $p_l$ is convex equilateral of length $l$ and $\theta_i(p) = \theta(l)$ for some homeomorphism
\[\theta: \left(0,\frac{2\pi}{k}\right] \to \left(\left(1-\frac{2}{k}\right)\pi, \pi\right]~.\]

There exists a smooth $1$-parameter family of spacelike de Sitter $k$-gons $(p_l)_{l \in \R_{\geq 0}}$ such that $p_l$ is convex equilateral of length $l$ and $\theta_i(p) = \theta(l)$ for some homeomorphism
\[\theta: \R_{\geq 0} \to \R_{\geq 0}~.\]
\end{prop}

\begin{proof}
In the spherical case, fix an orthogonal basis and set $p_\alpha = (v_j(\alpha))_{1\leq j \leq k}$ where 
\[v_j(\alpha) = \left (\cos(\alpha) \cos\left(\frac{2\pi j}{k}\right), \cos (\alpha) \sin\left(\frac{2\pi j}{k}\right), \sin(\alpha)\right)~.\]

In the de Sitter case, fix an orthogonal basis and set 
$p_\alpha = (v_j(\alpha))_{1\leq j \leq k}$ where 
\[v_j(\alpha) = \left(\cosh(\alpha) \cos\left(\frac{2\pi j}{k}\right), \cosh (\alpha) \sin\left(\frac{2\pi j}{k}\right), \sinh(\alpha)\right)~.\]

The polygon $p_\alpha$ is symmetric under rotation of angle $\frac{2\pi}{k}$, hence it is equilateral with lengths $l(\alpha)$ and all angles equal to $\theta(\alpha)$. A straightforward computation shows that the maps
\[\alpha\mapsto l(p_\alpha)\quad \textrm{and} \quad \alpha \mapsto \theta(\alpha)\]
are both homeomorphisms between the appropriate intervals.
\end{proof}

\subsection{Equilateral polygons with a central symmetry}

Consider the involution $\sigma$ of $\mathcal U_{2k}(\bS)$ (respectively, of $\mathcal U_{2k}(\dS)$) given by
\[\sigma (v_1,\ldots ,v_{2k}) = (v_{k+1}, \ldots, v_{2k}, v_1,\ldots, v_k)~.\]
The involution commutes with the action of $\SO(3)$ (resp. $\SO_\circ(2,1)$) and thus factors to an involution of $\mathcal P_{2k}(\bS)$ (resp. $\mathcal P_{2k}(\dS)$) that we still denote by $\sigma$. Since two polygons with the same lengths and angles are congruent, the following properties are equivalent:
\begin{itemize}
\item the class of $p\in \mathcal P_{2k}(\bS)$ (resp. $\mathcal P_{2k}(\dS)$) is fixed by $\sigma$,
\item there exists a unit vector $w$ (resp. a unit negative vector $w$) such that
\[\sigma p = s_w p~,\]
where $s_w$ is the central symmetry with axis $w$,
\item $l_i(p)=l_{i+k}(p)$ and $\theta_i(p) = \theta_{i+k}(p)$ for all $i \in \Z/2k\Z$.
\end{itemize}

Note that $\sigma$ preserves the space of equilateral polygons. We denote by $\mathcal P_{2k}^{\textit{sym}}(\bS)$ (resp. $\mathcal P_{2k}^{\textit{sym}}(\dS)$) the moduli space of equilateral $2k$-gons with a central symmetry.

\begin{prop} \label{p:Dimension symmetric polygons}
The set $\mathcal P_{2k}^{\textit{sym}}(\bS)$ (resp. $\mathcal P_{2k}^{\textit{sym}}(\dS)$) is a submanifold of $\mathcal P_{2k}^{\textit{eq}}(\bS)$ (resp. $\mathcal P_{2k}^{\textit{eq}}(\dS)$) of dimension $k$.
\end{prop}

\begin{proof}
We do the proof in the spherical setting, but the proof in the de Sitter setting is identical. 

Since $\sigma$ is a smooth diffeomorphism of $\mathcal P_{2k}^{\textit{eq}}(\bS)$ of finite order, it is linearizable at every fixed point and its fixed locus is a submanifold of local dimension 
\[\dim \ker (\mathrm d \sigma -\Id)~.\]

Let $p = (v_1,\ldots, v_{2k})$ be an equilateral polygon with a central symmetry and $w$ such that 
\[v_{i+k} = s_w(v_i)~.\]
Then $\sigma$ fixes the isomorphism class of $p$ and the action of $\mathrm d_p \sigma$ on $T_p \mathcal \mathcal P_{2k}^{\textit{eq}}(\bS)$ sends a first order variation of the length and angles $(\dot \theta_1, \ldots, \dot \theta_{2k}, \dot l)$ to
\[(\dot \theta_{k+1}, \ldots \dot \theta_{2k}, \dot \theta_1, \ldots, \dot \theta_k, \dot l)~.\]
By Theorem \ref{tm:polygons}, the kernel of $\mathrm d_p \sigma + \Id$ is identified with the set of tuples
\[(\dot \theta_1, \ldots, \dot \theta_k, - \dot \theta_1, \ldots, -\dot \theta_k, 0)\]
satisfying the relation
\begin{equation}\label{eq: ker sigma + Id}
\sum_{i=1}^k \dot \theta_i (v_i - s_w v_i) = 0~.
\end{equation}
Note that the vectors $v_i -s_w v_i$ are all orthogonal to $w$. Moreover, they span $w^\perp$, for otherwise all the $v_i$ would be contained in a great circle passing through $w$, which is absurd because $s_w$ reverses the orientation of this circle while $\sigma$ preserves the orientation.

Hence the set of $(\dot \theta_i)_{1\leq i \leq k}$ satisfying \eqref{eq: ker sigma + Id} has dimension $k-2$. Since $T_p \mathcal \mathcal P_{2k}^{\textit{eq}}(\bS)$ has dimension $2k-2$, we deduce that
\[\dim \ker ( \mathrm d \sigma -\Id) = \dim \mathcal P_{2k}^{\textit{sym}}(\bS) = k~.\]
\end{proof}

Recall that there is (up to isometry) a unique equilateral spherical (resp.  spacelike de Sitter) $2k$-gon with vanishing angles. This polygon is contained in a (spacelike) geodesic, and its edges divide this geodesic in segments of length $\frac{\pi}{k}$. In particular, it has a central symmetry. We denote this polygon by $p_0$.

\begin{prop} \label{prop: angles symmetric polygons dS}
There exists a neighbourhood $\mathcal V$ of $p_0$ in $\mathcal P_{2k}^\textit{sym}(\dS)$ (resp. $\mathcal P_{2k}^\textit{sym}(\bS)$) such that the map 
\[\begin{array}{crcl}
\Theta: &\mathcal V & \to & \R^{k}\\
& p & \to &(\theta_1(p), \ldots \theta_k(p))
\end{array}\]
is a diffeomorphism onto an open neighbourhood of $(0,\ldots, 0)$.
\end{prop}

\begin{proof}
The proposition states that $\Theta$ is a local diffeomorphism at $p_0$. The proof is identical in the spherical and de Sitter case.
By Proposition \ref{p:Dimension symmetric polygons}, the space $\mathcal P_{2k}^{\textit{sym}}(\dS)$ (resp. $\mathcal P_{2k}^{\textit{sym}}(\bS)$) has dimension $k$, so it suffices to prove that $\mathrm d \Theta$ is injective at $p_0$.

Set $p_0= (v_1,\ldots v_{2k})$. Since all the $v_i$ belong to the same geodesic, they are orthogonal to the same unit vector $w$, hence all the auxiliary vectors $w_i$ are equal to $w$. 

By Theorem \ref{tm:polygons}, the kernel of $\mathrm d_p \Theta$ identifies with the set of infinitesimal variations of angles and length of the form
\[(0,\ldots, 0, \dot l)\]
satisfying the equation
\begin{equation} \label{eq: Kernel Theta}
\dot l \left (\sum_{i=1}^{2k} w_i\right) = 0~.
\end{equation}Since all the $w_i$ are equal to $w$, \eqref{eq: Kernel Theta} implies $\dot l= 0$. We conclude that $\mathrm d_p \Theta$ is injective, hence $\Theta$ is local diffeomorphism in a neighborhood of $p_0$.
\end{proof}

\begin{remark}
The proof shows more generally that the map $\Theta$ is immersive at every polygon $p$ for which $\sum_{i=1}^{k} w_i \neq 0$. One can show that it is the case when $p$ is a convex spherical polygon and when $p$ is any spacelike de Sitter polygon.
\end{remark}

\section{Geometrization of Gromov--Thurston manifolds} 
\label{sec:geometrization}

Thanks to the results of Section \ref{sc:globally}, in order to prove Theorem \ref{tm:existence_ads}, it is enough to construct a convex ruled spacelike $\AdS$ structure on a given Gromov--Thurston manifold $M^a$. The spacelike structure we construct will be totally geodesic away from the hypersurfaces $H_i$ where the spacelike embedding is ``folded''. Along the codimension $2$ locus $S$, several dihedra with a total angle greater than $2 \pi$ are patched together. A similar construction will be made to prove Theorem \ref{tm:existence_hyp}, with the only difference that the cone angle is less than $2\pi$ along the codimension 2 stratum.

\subsection{Hipped hypersurfaces in $\AdS^{d+1}$ and polygons in $\dS^2$}

In order to obtain a local isometry from a Gromov--Thurston cone-manifold $M^{a}$ with $a\geq 1$ into $\AdS^{d+1}$, we need to understand the polyhedral hypersurfaces in $\AdS^{d+1}$ that carry the same geometry. We will see that such hypersurfaces can be parametrized by spacelike polygons in $\dS^2$.

\begin{defi} \label{def:hip}
A \emph{hipped hypersurface} in $\AdS^{d+1}$ is a Lipschitz spacelike hypersurface $H\subset \AdS^{d+1}$ that is a finite union $H=\bigcup_{i=1}^k X_i$ of subsets with the following properties:
\begin{enumerate}
\item Each $X_i$ is a convex subset of a totally geodesic copy of $\HH^d$,
\item The relative boundary of $X_i$ is the union of two half-spaces $Y_i$ and $Y_{i+1}$ of totally geodesic copies of $\HH^{d-1}$,
\item $X_i\cap X_{i+1}=Y_{i+1}$ for all $i\in \{1,\dots,k\}$ (setting $X_{k+1}=X_1$ and $Y_{k+1}=Y_1$),
\item There is a totally geodesic copy $Z\subset \AdS^{d+1}$ of $\HH^{d-2}$, called the \emph{stem} of $H$, such that $Y_i\cap Y_{i+1}=Z$ for all $i\in \{1,\dots,k\}$.
\end{enumerate}
\end{defi}

 \begin{figure}[h] 
 \includegraphics[scale=0.8]{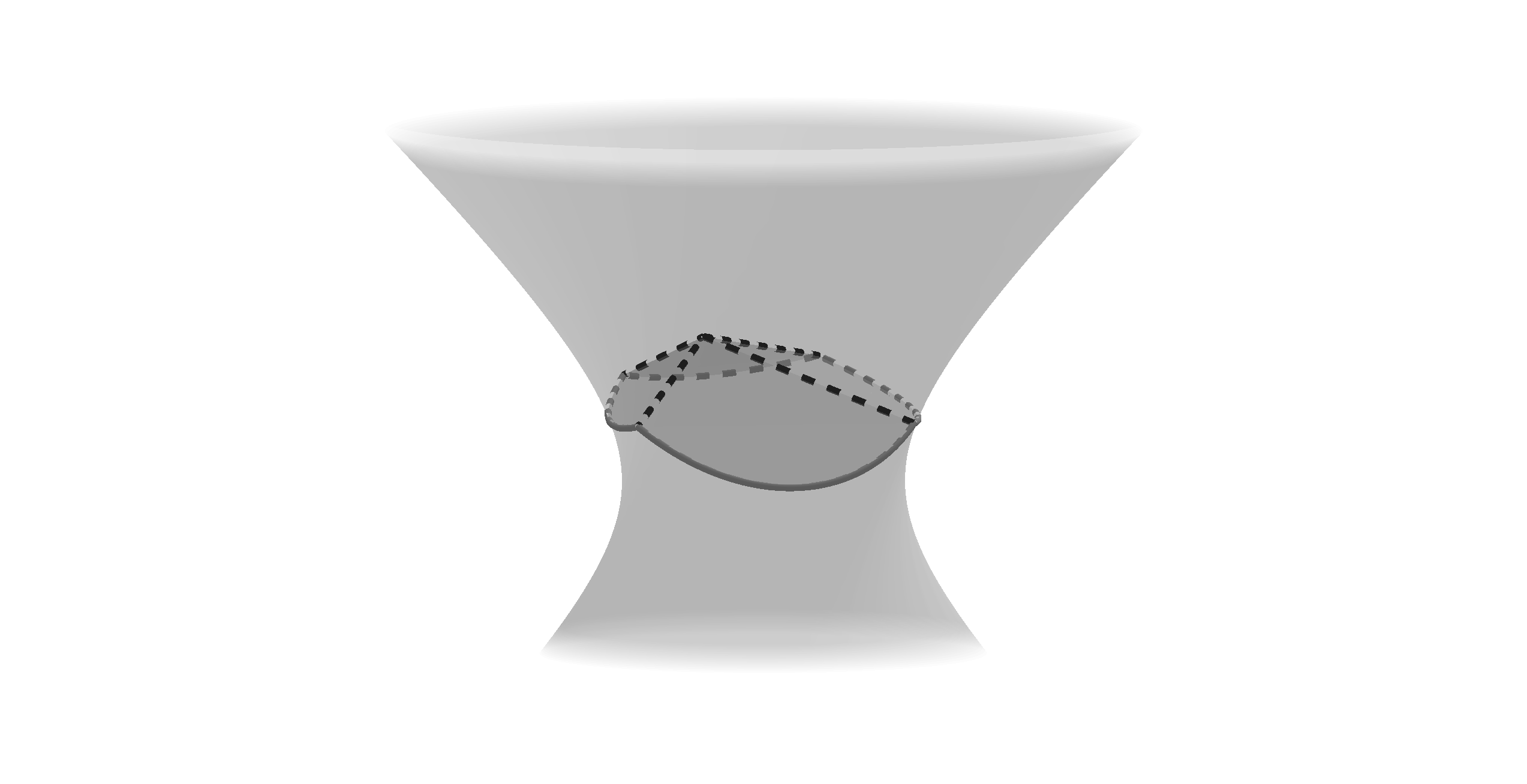}
 \caption{A hipped hypersurface in $\AdS^3$.}
 \label{fig:hipped hypersurface}
 \end{figure}

Let us give precise definitions of angles between  totally geodesic subspaces of $\AdS^{d+1}$. First, consider two totally geodesic copies $X_1,X_2\subset \AdS^{d+1}$  of $\HH^d$ intersecting along a  totally geodesic copy $Y$ of $\HH^{d-1}$. Consider an isometry $g\in \SO_\circ(d,2)$ such that $g(X_1)=X_2$ and $g$ is the identity on $Y$. Now $Y$ corresponds to a vector subspace $Y\subset \R^{d,2}$ of signature $(d-1,1)$, so $Y^\perp\subset \R^{d,2}$ is a plane of signature $(1,1)$. It follows that the restriction of $g$ to $Y^\perp$ is conjugate to an element of $\SO_\circ(1,1)\subset\SO_\circ(d,2)$, hence of the form $\begin{pmatrix} \cosh t & \sinh t\\ \sinh t & \cosh t \end{pmatrix}$. The angle between $X_1$ and $X_2$ is the real number $t$ (it could also be seen as the angle between the normal vectors to $X_1$ and $X_2$ at any point of $Y$, which is the angle between timelike vectors).

 Now consider  $Y_1,Y_2\subset \AdS^{d+1}$   half-spaces of totally geodesic copies of $\HH^{d-1}$ with common relative boundary $Z\approx \HH^{d-2}$, consider an isometry $g\in \SO_\circ(d,2)$ such that $g(Y_1)=Y_2$ and $g$ is the identity on $Z$. The angle between $Y_1$ and $Y_2$  is the unique  $\theta\in \R/2\pi\Z$ such that $g$ is conjugate in $\SO_\circ(d,2)$ to the matrix
\[ \begin{pmatrix}1_d & & \\ & \cos\theta  & -\sin\theta  \\ &\sin\theta & \cos\theta  \end{pmatrix}. \]


\begin{defi} \label{def:angleship}
Let $H=\bigcup_{i=1}^k X_i\subset \AdS^{d+1}$ be a hipped hypersurface. The \emph{dihedral angles} are the angles between $X_i$ and $X_{i+1}$, and the \emph{wedge angles} are the angles between $Y_i$ and $Y_{i+1}$.
\end{defi}

Using the exponential map of $\AdS^{d+1}$ at a point of the stem, the coordinates we obtain on  a hipped hypersurface show that it carries  a cone hyperbolic metric whose singular locus is the stem (and is totally geodesic) and whose angle is the sum of the wedge angles.

\begin{lemma} \label{lem:pastconvexhyp}
A hipped hypersurface in $\AdS^{d+1}$ is past-convex if and only if the dihedral angles are non negative.
\end{lemma}

\begin{proof}
If the angle between $X_i$ and $X_{i+1}$ is negative, consider a point $x$ in the relative interior of their common intersection $Y_{i+1}$. One can always find a spacelike geodesic in the future of $x$ that intersects both $X_i$ and $X_{i+1}$ transversally, so its intersection with the past of $H$ is disconnected.\\
Now assume that all the angles are non negative. In this case  $H$ is in the past of each spacelike hyperplane containing $X_1,\dots,X_k$, so it is past-convex.
\end{proof}

\begin{lemma} \label{lm:dSpolygonsToAdSpolyhedra}
Let $\alpha_1,\dots,\alpha_{k}\in (0,\pi)$ and $\theta_1,\dots,\theta_k\in \R$. The set of hipped hypersurfaces $H=\bigcup_{i=1}^rX_i\subset\AdS^{d+1}$ with wedge angles $\alpha_1,\dots,\alpha_k$ and dihedral angles $\theta_1,\dots,\theta_k$ considered up to isometry is in one-to-one correspondence with spacelike polygons $p\subset \dS^2$ with side lengths $\alpha_1,\dots,\alpha_k$ and angles $\theta_1,\dots,\theta_k$ up to isometry. Through this correspondence, convex polygons are associated to past-convex hypersurfaces.
\end{lemma}

\begin{proof} 
Start with a hipped hypersurface $H=\bigcup_{i=1}^r X_i\subset \AdS^{d+1}$, let $Z\approx \HH^{d-2}$ be its stem and consider the link $\mathcal L$  of $Z$. 
Fix $z\in Z$, so that $\mathcal L$ identifies with the set of unit spacelike vectors tangent to $\AdS^{d+1}$ at $z$ and orthogonal to $Z$. Notice that since $T_{z}\AdS^{d+1}=T_{z}Z \oplus (T_{z}Z)^\perp$, there is a natural identification between $\mathcal L$ and $\dS^2$ (seen as the set of unit spacelike vectors in $(T_{z}Z)^\perp$).

Intersecting $ H$ with $\mathcal L$ yields a spacelike polygon $p\subset \dS^2$. Its vertices $v_1,\dots,v_k$ are defined by  $v_i\in T_{z}Y_i\cap (T_{z}Z)^\perp$ and $\exp_z(v_i)\in Y_i$. The edges $[v_i,v_{i+1}]$ correspond to  $T_{z}X_i\cap (T_{z}Z)^\perp$.


Starting with such a polygon $p\subset \dS^2$ with vertices $v_1,\dots,v_r$, we consider any point $z\in \AdS^{d+1}$ and any totally geodesic copy $Z\subset \AdS^{d+1}$ of $\HH^{d-2}$ containing $z$.  By identifying $\dS^2$ with unit spacelike vectors in $(T_{z}Z)^\perp$, we can  define: 
\begin{align*}
Y_i&=\set{\exp_z(u+tv_i)}{u\in T_zZ \, ,~ t\geq 0}\\
X_i&=\set{\exp_z(u+tw)}{u\in T_zZ \, ,~w\in [v_i,v_{i+1}]\, ,~ t\geq 0}
\end{align*}
Then $ H=\bigcup_{i=1}^k X_i\subset\AdS^{d+1}$ is a hipped hypersurface, and these two constructions are inverse to each other.


In this correspondence, the length of $[v_i,v_{i+1}]$ is given by the angle between $Y_i$ and $Y_{i+1}$, i.e.~$\alpha_i$. The angle at $v_i$ equals the angle between $X_i$ and $X_{i+1}$, namely $\theta_i$. Following Lemma \ref{lem:pastconvexhyp}, we see that the hipped hypersurface $H$ is past-convex if and only if the spacelike polygon $p$ is convex.


\end{proof}

\subsection{Geometrization of Gromov--Thurston cone-manifolds for $a>1$}

We now consider a Gromov--Thurston manifold $M^a$ with $a=\frac{k}{n}>1$. Recall that $M^a$ is obtained by gluing $k$ ``wedges'' $V_1,\ldots ,V_{2k}$ of $\overline M$ along $S$, each making an angle $2\pi/n$ at $S$. These wedges are bounded by hypersurfaces $H_1,\dots, H_{2k}$ with boundary $S$.

We wish to construct a spacelike $\AdS$ structure on $M^a$. Since $M^a\setminus S$ is a hyperbolic manifold, the only real work consists in  constructing a  Lipschitz spacelike immersion of $M^a$ into $\AdS^{d+1}$ in a neighbourhood of $S$. The idea is to construct ``folded'' spacelike immersions using hipped hypersurfaces in $\AdS^{d+1}$.

\begin{defi} \label{def:foldedspacelikeimmersion}
A spacelike $\AdS$ structure  on $M^a$ is \emph{folded} if the image under the developing map of any connected component in the universal cover of $M^a$ of the complement of the  union of the lifts of $H_1,\dots, H_{2k}$ is included in a totally geodesic copy of $\HH^d$, and the induced metric on $M^a\setminus S$ is the original hyperbolic metric.
\end{defi}



The outcome of this discussion is therefore the following lemma.

\begin{lemma} \label{lm:equivalence}
Let $M^a$ be a Gromov--Thurston manifold of dimension $d$ with $a=\frac{k}{n}>1$.  There is a one-to-one correspondence between folded spacelike $\AdS$ structures on $M^a$ (up to equivalence) and hipped hypersurfaces in $\AdS^{d+1}$ with $2k$ wedges and wedge angles $\frac{\pi}{n}$ (up to isometry). This correspondence associates a convex spacelike $\AdS$ structure to a convex hipped hypersurface.
\end{lemma}


\begin{proof}
Let $\dev:\tilde M_a \to \AdS^{d+1}$ be the developing map of a folded spacelike AdS structure on $M^a$, let $U$ be a (sufficiently small) neighbourhood of a lift to $\widetilde{M^a}$ of a point of the stem $S$. Then, by definition of a folded structure, $\dev(U)$ is an open subset of a hipped hypersurface with $2k$ wedges, with wedge angles $\frac{\pi}{n}$.\\

Conversely, consider a  convex hipped hypersurface $\Sigma=\bigcup_{i=1}^{2k}X_i\subset \AdS^{d+1}$ with wedge angles~$\frac{\pi}{n}$. Denote by $V_i$ the connected component of $M^a\backslash \bigcup_{i=1}^{2k}H_i$ which is bounded by $H_i$ and $H_{i+1}$. Define an atlas of spacelike charts on $M^a$ in the following way:
\begin{itemize}
\item If $x\in M^a \backslash \bigcup_{i=1}^{2k} H_i$, choose a small ball $U_x$ around $x$ that does not intersect any of the $H_i$ and define a chart $\phi_x: U_x \to \AdS^{d+1}$ mapping $U_x$ isometrically onto a spacelike hyperplane in $\AdS^{d+1}$.
\item If $x\in H_i\backslash S$, choose a small ball $U_x$ around $x$ that does not intersect $S$, and define a continuous chart $\phi_x: U_x \to \AdS^{d+1}$ mapping isometrically $U_x\cap V_{i-1}$ into $X_{i-1}$, $U_x\cap V_i$ into $X_i$ and $U_x \cap H_i$ into $Y_i$.
\item If $x\in S$, choose a small ball $U_x$ around $x$, and define a continuous chart $\phi_x: U_x \to \AdS^{d+1}$ $x$ mapping isometrically $U_x\cap V_i$ into $X_i$ and $U_x\cap V_i$ into $Y_i$ for all $i$.
\end{itemize}
One easily verifies that the properties of the charts $\phi_x$ characterize them up to an isometry of $\AdS^{d+1}$. Hence they form the atlas of a folded spacelike AdS structure. Moreover, the neighbourhood of any point in the stem $x$ is mapped to a neighbourhood of the stem of the hipped hypersurface $\Sigma$, showing that this construction is a converse to the previous one.\\

Finally, through this correspondence, it is clear that a past convex hipped hypersurface is associated to a locally convex spacelike structure, hence a convex spacelike structure by Proposition\ref{p: Local convexity -> convexity}.
\end{proof}

We can now combine everything to prove the main theorems of the paper.

\begin{proof}[Proof of Theorems \ref{tm:existence_ads} and  \ref{tm:dim_ads}]
Let $M^a$ be a Gromov--Thurston manifold of dimension $d$ with $a=\frac{k}{n}>1$. 

By Proposition \ref{prop: Existence polygons prescribed lengths}, there exists a convex spacelike polygon $p$ in $\dS^2$ with $2k$ sides of length~$\frac{\pi}{n}$. By Lemma \ref{lm:dSpolygonsToAdSpolyhedra}, $p$ defines a convex hipped hypersurface $\Sigma_p$ in $\AdS^{d+1}$. By Lemma \ref{lm:equivalence}, $\Sigma_p$ defines a convex folded spacelike AdS structure $(\dev_p,\rho_p)$ on $M^a$. Finally, by Theorem \ref{pr:spacelike}, this folded hyperbolic structure defines an embedding of $M_a$ as a Cauchy hypersurface in a GHMC AdS manifold $N_{\rho_p}$. This already proves Theorem \ref{tm:existence_ads}.\\

We obtain a map $p \to N_{\rho_p}$ from $\mathcal P_{2k}^{\frac \pi n}(\dS)$ to the deformation space of GHMC AdS $d+1$-manifolds. Moreover, for each $p \in \mathcal P_{2k}^{\frac \pi n}(\dS)$, the manifold $N_{\rho_p}$ contains a past-convex folded spacelike hypersurface isometric to $M_a$. In dimension $d+1 \geq  4$, folded spacelike hypersurfaces are ruled, hence $(\dev_p,\rho_p)$ is a convex ruled spacelike AdS structure. By Lemma \ref{lm:unique}, the image of $\dev_p$ is the future boundary of the convex core of $N_{\rho_p}$ and by Lemma \ref{lm:AdSquasifucsianConvexRuledSpacelikeStructure}, the map $p\mapsto N_{\rho_p}$ is injective. Finally, by Proposition \ref{prop:EquilateralPolygons-dS}, the space $\mathcal P_{2k}^{\frac \pi n}(\dS)$ is a manifold of dimension $k-3$. Hence the family of GHMC AdS manifolds $\left(N_{\rho_p},p\in  \mathcal P_{2k}^{\frac \pi n}(\dS) \right)$ proves Theorem \ref{tm:dim_ads}.

%
%
\end{proof}

\begin{remark}
Though we did not mention anything about the regularity of the map $p \mapsto N_{\rho_p}$, it is quite clear that this map should be continuous for an appropriate topolopogy on the space of GHMC AdS manifolds homeomorphic to $M^a \times \R$. We will discuss further these regularity questions in Section \ref{ss:Regularity}.
\end{remark}

\subsection{Hipped hypersurfaces in $\HH^{d+1}$ and polygons in $\bS^2$}


We now move on to establish a Riemannian version of Lemma \ref{lm:dSpolygonsToAdSpolyhedra}.

\begin{defi} \label{def:hipRiemannian}
A \emph{hipped hypersurface} is an oriented  topological hypersurface $H\subset \HH^{d+1}$ which is a finite union $H=\bigcup_{i=1}^k X_i$ of subsets with the following properties:
\begin{enumerate}
\item Each $X_i$ is a convex subset of a totally geodesic copy of $\HH^d$,
\item The relative boundary of $X_i$ is the union of two half-spaces $Y_i$ and $Y_{i+1}$ of totally geodesic copies of $\HH^{d-1}$,
\item $X_i\cap X_{i+1}=Y_{i+1}$ for all $i\in \{1,\dots,k\}$ (setting $X_{k+1}=X_1$ and $Y_{k+1}=Y_1$),
\item There is a totally geodesic copy $Z\approx \HH^{d-2}$ of $\HH^{d+1}$, called the \emph{stem}, such that $Y_i\cap Y_{i+1}=Z$ for all $i\in \{1,\dots,k\}$.
\end{enumerate}
 The \emph{dihedral angles} of $H$ are the angles between $X_i$ and $X_{i+1}$, and the \emph{wedge angles} are the  angles between $Y_i$ and $Y_{i+1}$. A hipped hypersurface is \emph{convex} if all its dihedral angles are non-negative (or, equivalently, if the component of $\HH^{d+1}\setminus H$ inducing the orientation of $H$ with the outward pointing normal is convex).
\end{defi}

 \begin{figure}[h] 
 \includegraphics[scale=0.8]{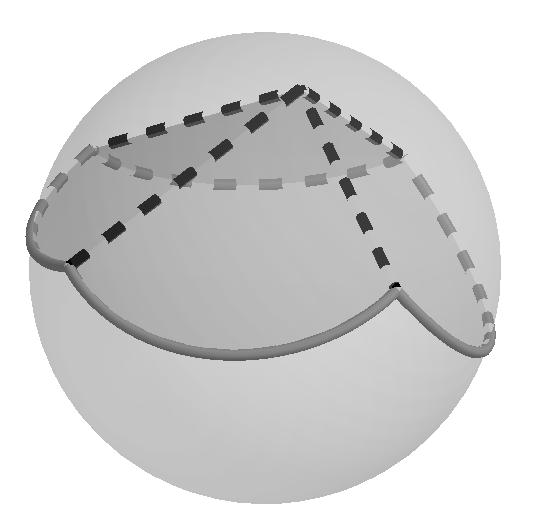}
 \caption{A hipped hypersurface in $\HH^3$.}
 \label{fig:hipped domain}
 \end{figure}

\begin{lemma} \label{lm:spherepolygonsToHyperbolicPolyhedra}
Let $\alpha_1,\dots,\alpha_{k}\in (0,\pi)$ and $\theta_1,\dots,\theta_k\in (-\pi,\pi)$. The set of hipped hypersurfaces $H\subset\HH^{d+1}$ with wedge angles $\alpha_1,\dots,\alpha_k$ and dihedral angles $\theta_1,\dots,\theta_k$ considered up to isometry is in one-to-one correspondence with spherical  polygons $p\subset \bS^2$ with side lengths $\alpha_1,\dots,\alpha_k$ and angles $\theta_1,\dots,\theta_k$ up to isometry. Through this correspondence, convex polygons are associated to convex hipped hypersurfaces.
\end{lemma}

\begin{proof}
  The first part of the proof is almost identical to that of Lemma \ref{lm:dSpolygonsToAdSpolyhedra}. Start with a hipped hypersurface   $H=\bigcup_{i=1}^k X_i\subset \HH^{d+1}$, let $Z\approx \HH^{d-2}$ be its stem and  consider the link $\mathcal L$  of $Z$. 
Fix $z\in Z$, so that $\mathcal L$ identifies with the set of unit vectors tangent to $\HH^{d+1}$ at $z$ and orthogonal to $Z$, and thus identify  $\mathcal L$ with $\bS^2$ as the unit vectors in $(T_zZ)^\perp$.

The intersection of $H$ with $\mathcal L$ is a polygon $p\subset \bS^2$. Its vertices $v_1,\dots,v_k$ are defined by  $v_i\in T_{z}Y_i\cap (T_{z}Z)^\perp$ and $\exp_z(v_i)\in Y_i$. The edges $[v_i,v_{i+1}]$ correspond to $T_{z}X_i\cap(T_{z}Z)^\perp$.


Starting with such a polygon $p\subset \bS^2$ with vertices $v_1,\dots,v_k$, we consider any point $z\in \HH^{d+1}$ and any totally geodesic copy $Z\subset \HH^{d+1}$ of $\HH^{d-2}$ containing $z$.  Identify $\bS^2$ with unit spacelike vectors in $(T_{z}Z)^\perp$, and consider the hipped hypersurface $H=\bigcup_{i=1}^rX_i\subset\HH^{d+1}$ where:
\begin{align*}
X_i&=\set{\exp_z(u+tw)}{u\in T_zZ \, ,~w\in [v_i,v_{i+1}]\, ,~ t\geq 0}~.
\end{align*}
Once again, these two constructions are inverse to each other.

In this correspondence, the length of the edge $[v_i,v_{i+1}]$ is given by the angle between $Y_i$ and~$Y_{i+1}$, i.e. $\alpha_i$. The angle at $v_i$ is equal the angle between $X_i$ and $X_{i+1}$, namely $\theta_i$. 

Both convexities are equivalent to the non-negativity of $\theta_1,\dots,\theta_k$, and are therefore concomitant.
\end{proof}

\subsection{Geometrization of Gromov--Thurston cone-manifolds for $a<1$}

We now consider a Gromov--Thurston manifold $M^a$ with $a=\frac{k}{n}<1$. 
We wish to construct a convex ruled hyperbolic embedding structure on $M^a$. Since $M^a\setminus S$ is a hyperbolic manifold, we only need to construct such a structure on a neighbourhood of $S$, and we will do so by using a hipped hypersurface.

\begin{defi} \label{def:foldedimmersion}
A  hyperbolic embedding structure on $M^a$ is \emph{folded} if the image under the developing map of any connected component in the universal cover of $M^a$ of the complement of the union of the lifts of $H_1,\dots, H_{2k}$ is included in a totally geodesic copy of $\HH^d$, and the induced metric on $M^a\setminus S$ is the original hyperbolic metric.
\end{defi}



The outcome of this discussion is therefore the following lemma.

\begin{lemma} \label{lm:equivalenceRiemannian}
Let $M^a$ be a Gromov--Thurston manifold of dimension $d$ with $a=\frac{k}{n}<1$.  There is a one-to-one correspondence between  folded hyperbolic embedding structures on $M^a$ (up to equivalence) and  hipped hypersurfaces in $\HH^{d+1}$ with $2k$ wedges and wedge angles $\frac{\pi}{n}$ (up to isometry). This correspondence associates a convex hyperbolic embedding structure to a convex hipped hypersurface.
\end{lemma}

\begin{proof}
The proof is almost identical to that of Lemma \ref{lm:equivalence}: 
the developing map of a folded hyperbolic embedding structure on $M^a$ sends a lift to $\widetilde M^a$ of the stem $S$ isometrically into  a  hipped hypersurface with $2k$ wedges, with wedge angles $\frac{\pi}{n}$.

Conversely, given a  hipped hypersurface $H=\bigcup_{i=1}^{2k}X_i$ consisting of $2k$ wedges with wedge angles $\frac{\pi}{n}$, one constructs a folded hyperbolic embedding structure on $M^a$ mapping locally isometrically $V_i$ into $X_i$, $H_i$ into $Y_i$ and $S$ into $\bigcap_{i=1}^{2k} Y_i$.


The two constructions are inverse to each other.
\end{proof}


\begin{proof}[Proof of Theorems \ref{tm:existence_hyp} and  \ref{tm:dim_hyp}]

The proof again follows closely that of Theorems \ref{tm:existence_ads} and \ref{tm:dim_ads}: 

Since every convex folded spacelike embedding in dimension $d+1\geq 4$ is ruled, Lemma \ref{lm:equivalenceRiemannian} defines a map 
\[p \mapsto (\dev_p, \rho_p)\]
from $\mathcal P_{2k}^{\frac \pi n}(\bS)$ to the set of convex ruled spacelike embeddings which are isometric to $M^a$. By Theorem \ref{lm:HyperbolicEndConvexRuledEmbeddingStructure}, every such spacelike embedding is the boundary of a unique hyperbolic end $N_{\rho_p}$. By Lemma \ref{prop: Existence polygons prescribed lengths}, $\mathcal P_{2k}^{\frac \pi n}(\bS)$ is non-empty, proving Theorem \ref{tm:existence_hyp}. By Lemma \ref{prop:EquilateralPolygons}, $\mathcal P_{2k}^{\frac \pi n}(\bS)$ is a manifold of dimension $k-3$, and the family of hyperbolic ends $\left( N_{\rho_p}, p\in \mathcal P_{2k}^{\frac \pi n}(\bS) \right)$ answers Theorem \ref{tm:dim_hyp}.

\end{proof}


\section{Fuchsian deformations in dimension 3+1} 
\label{sc:6}

Gromov--Thurston $3$-manifolds are irreducible, atoroidal and Haken. It thus follows from Thurston or Perelman's hyperbolization theorems that they also carry a smooth hyperbolic structure. Here we will give a simpler proof of this fact, showing moreover that the quasifuchsian AdS manifolds constructed in the previous sections are, in dimension $3+1$, deformations of Fuchsian manifolds. Barbot's conjecture thus remains open in dimension $3+1$. 

We also prove the same result for deformations to Fuchsian manifolds of the hyperbolic ends defined above (corresponding to Gromov--Thurston cone-manifolds of cone angle smaller than $2\pi$, still in dimension $3+1$). Both proofs are based on Hodgson--Kerckhoff's deformation theorem for conical hyperbolic $3$-manifolds.



\subsection{The Hodgson--Kerckhoff deformation theorem}

We recall first the Hodgson--Kerckhoff theorem, see \cite{HK}, for 3-dimensional hyperbolic cone-manifolds, which will be a useful tool for us in understanding deformations of quasifuchsian AdS spacetimes (or hyperbolic ends) in dimension $3+1$.

\begin{theorem}[Hodgson--Kerckhoff, \cite{HK}] \label{tm:hk}
  Let $M$ be a 3-dimensional hyperbolic manifold with cone singularities along a link $\gamma=\gamma_1\sqcup\cdots\sqcup \gamma_n$, with angle $\theta_i\in (0,2\pi)$ along $\gamma_i$, $1\leq i\leq n$. Then small deformations of $M$ among hyperbolic cone-manifolds with constant singular locus are parameterized by the variations of the cone angles $\theta_1, \cdots, \theta_n$. 
\end{theorem}

Note that this deformation result was extended to 3-dimensional hyperbolic cone-manifolds with singularities along a graph, still with angles less than $2\pi$, by Mazzeo and Moncouquiol \cite{mazzeo-montcouquiol}, see also \cite{weiss-local}.

The Hodgson--Kerckhoff deformation theorem leads quite naturally to a deformation result for the ``building blocks'' of Gromov--Thurston $3$-manifolds. This will in turn be used below to construct deformations of globally hyperbolic AdS spacetimes, or of hyperbolic ends, in dimension 3+1. We will actually need a more precise rigidity result which states that one can prescribe the cone angles as long as we remain in the range $(0,\pi)$.

\begin{theorem} \label{tm:orbifolds}
  Let $M$ be a 3-dimensional hyperbolic manifold with cone singularities along a link $\gamma=\gamma_1\sqcup\cdots\sqcup \gamma_n$, with angle $\theta_i\in (0,\pi)$ along $\gamma_i$, $1\leq i\leq n$. For every $\theta'_1, \cdots, \theta'_n\in (0,\pi)$, there is a unique one-parameter family $(g_t)_{t\in [0,1]}$ of hyperbolic structures on $M$ with cone singularities along the $\gamma_i$ of angle $(1-t)\theta_i+t\theta'_i$.
\end{theorem}

We refer the reader to \cite{boileau-porti} or \cite[Section 6.2]{boileau-leeb-porti} for a proof. Briefly, one can consider the subset of values of $t$ which can be achieved. Theorem \ref{tm:hk} shows that it is open in $[0,1]$, while a separate, compactness argument (using the condition that the angles are in $(0,\pi)$) shows that it is closed.

\subsection{Deformations towards the Fuchsian locus}

We now apply the previous results to the cyclic quotients of hyperbolic manifolds containing totally geodesic planes. We follow the notations of Section \ref{sc:2}, and consider a hyperbolic $3$-manifold $M$ with a diedral group of symmetries $D_n$ of order $2n$. We then denote by $\overline{M}$ the quotient of $M$ by the cyclic subgroup $R_n$ of $D_n$. Then $\overline{M}$ is a hyperbolic cone-manifold, with cone angle $2\pi/n$ along the projection of $S$ (which is still denoted by $S$). We also fix an integer $k\geq 2$, and let $a=k/n$. Theorem \ref{tm:orbifolds} readily gives the following:

\begin{cor} \label{cr:hk}
  Under those hypothesis, for $\epsilon>0$ sufficiently small, there exists a one-parameter family of hyperbolic metrics $( \bar g_\alpha)$ on $\overline{M}$ with a cone singularity along $S$ of cone angle $2\alpha$, for $\alpha$ ranging in the interval $(0,\frac \pi 2 + \epsilon)$. Moreover, this family is unique up to isotopy.
\end{cor}

We can then lift this deformation to a deformation of the manifold $M^a$, which is a cover of $\overline{M}$ of degree $k$ ramified over $S$.

\begin{cor} \label{cr:lift}
Still with the notations of above, for $\epsilon>0$ sufficiently small, there exists a one-parameter family of hyperbolic metrics $(g_\alpha)$ on $M^{\frac k n}$ with cone singularity along $S$ such that the hypersurfaces $H_i$ are geodesic, and such that $H_i$ and $H_{i+1}$ meet with an angle $\alpha$, for $\alpha$ ranging in the interval $(\frac \pi k -\epsilon,\frac \pi n]$ when $k > n $ and $[\frac\pi n , \frac \pi k + \epsilon)$ when $k<n$.
\end{cor}

\begin{proof}
Note that, for $k,n\geq 2$ the range of $\alpha$ in Corollary \ref{cr:lift} is contained in the range of $\alpha$ of Corollary \ref{cr:hk}. Thus, for $\alpha$ in the appropriate range, we can consider the deformation $(\bar g_\alpha)$ of Corollary \ref{cr:hk}. Let $\sigma$ denote the reflection of $\overline{M}$ induced by any reflection in $D_n$. Then $(\sigma^* g_\alpha)$ is another such deformation which coincides with $(\bar g_\alpha)$ for $\alpha = \frac \pi n$. By the uniqueness of Theorem~\ref{tm:hk}, up to isotoping $\bar g_\alpha$, we can thus assume that $\sigma^* \bar g_\alpha= \bar g_\alpha$ for all $\alpha$.\footnote{This point is not actually immediate since the uniqueness of the metric is only up to isotopy, but one can keep track of the involution $\sigma$ in Hodgson--Kerckhoff's proof to verify that, starting with a $\sigma$-invariant metric, it does produce a $\sigma$-invariant deformation of the metric.} Since $\sigma$ is the reflection along $\overline H_1 \cup \overline H_2$, we deduce that $\overline H_1$ and $\overline H_2$ are totally geodesic for $\bar g_\alpha$ and meet along $S$ with an angle $\alpha$. Then the pull-back $g_\alpha$ of $\bar g_{\alpha}$ to $M^{\frac k n}$ satisfies the required conditions. 
\end{proof}

Note that the metric $g_{\frac \pi n}$ is the original cone metric of the Gromov--Thurston manifold $M^{\frac{k}{n}}$ while the metric $g_{\frac \pi k}$ is a smooth hyperbolic metric. This proves in particular that Gromov--Thurston $3$-manifolds are hyperbolic.

\subsection{Deformation to Fuchsian manifolds}

We can now conclude the proof of Theorems \ref{tm:31_ads} and \ref{tm:31_hyp} by showing how, in dimension $3+1$, the folded convex AdS (resp. hyperbolic) structures that we considered in Section \ref{sec:geometrization} on a manifold $M^{\frac k n}$ with $k>n$ (resp. $k<n$) can be deformed to the structure associated to a Fuchsian AdS $4$-manifold (resp. a Fuchsian hyperbolic end).

\begin{proof}[Proof of Theorem \ref{tm:31_ads}]
Assume $k> n$. Let $I$ denote the interval $(\frac \pi k-\epsilon,\frac \pi n]$. For every $\alpha \in I$ there is a metric $g_\alpha$ on $M^{\frac k n}$ for which the hypersurfaces $H_i$ are totally geodesic and form angles $\alpha$ along~$S$. Repeating the proof of Theorem \ref{tm:dim_ads}, we can associate to every spacelike polygon $p\in \mathcal P_{2k}^\alpha(\dS)$ a folded spacelike AdS structure on $(M^{\frac k n}, g_\alpha)$, which defines a GHMC AdS manifold $N_p$ by Theorem~\ref{pr:spacelike}. We thus get a map from
\[\mathcal P_{2k}^I(\dS) \equaldef \bigsqcup_{\alpha \in I} \mathcal P_{2k}^\alpha (\dS)\]
to the deformation space of quasifuchsian AdS $4$-manifolds homeomorphic to $M^{\frac{k}{n}}\times\R$.

When the polygon $p$ is convex, the image of the folded spacelike embedding of $(M^{\frac k n}, g_\alpha)$ in $N_p$ is the future boundary of its convex core, and we deduce that the map $p\mapsto N_p$ is injective in restriction to convex polygons, as in the proof of Theorem \ref{tm:dim_ads}. 

Now, by Proposition \ref{prop:EquilateralPolygons-dS}, the set $\mathcal P_{2k}^I(\dS)$ is a manifold of dimension $2k-2$, which contains the codimension $1$ submanifold $\mathcal P_{2k}^{\frac{\pi}{n}}(\dS)$. By Proposition \ref{prop: Existence polygons prescribed lengths}, there is a continuous path $(p_\alpha)_{\alpha \in I}$ in $\mathcal P_{2k}^I(\dS)$ such that $p_\alpha$ is convex with side length $\alpha$. For $\alpha=\frac{\pi}{k}$, the polygon $p_\alpha$ is a spacelike geodesic in $\dS^2$ divided in $2k$ segments of length $\frac \pi k$,  the corresponding folded spacelike embedding is totally geodesic, and the GHMC AdS $4$-manifold $N_{p_{\frac \pi k}}$ is thus Fuchsian.

Hence the family $\left(N_p, p\in \mathcal P_{2k}^I(\dS) \right )$ satisfies the required properties, proving Theorem \ref{tm:31_ads}.

\end{proof}

\begin{proof}[Proof of Theorem \ref{tm:31_hyp}]
The proof proceeds in the same way as the proof of Theorem \ref{tm:31_ads}, except that we now associate hyperbolic ends $N_p$ to equilateral spherical polygons 
\[p \in \mathcal P_{2k}^I (\bS) \equaldef \bigsqcup_{\alpha \in I} \mathcal P_{2k}^\alpha (\bS)~,\]
for $\alpha$ ranging in the interval
\[I= \left[ \frac{\pi}{n}, \frac \pi k + \epsilon \right)~.\]
\end{proof}

\subsection{Integration of bending deformations} \label{ss:IntegrationBending}

Restricting the map $p\mapsto N_p$ to polygons with a central symmetry, one proves Theorems \ref{tm:bending3d-hyp} and \ref{tm:bending3d-ads}:

\begin{proof}[Proof of Theorem \ref{tm:bending3d-ads}]

The geodesic polygon $p_0$ with side length $\frac{\pi}{k}$ and angles $0$ belongs to the interior of $\mathcal P_{2k}^I(\dS)$, which thus contains an open neighbourhood $O$ of $p_0$ in $\mathcal P_{2k}^\sym(\dS)$. By Proposition~\ref{prop: angles symmetric polygons dS}, the map 
\[ \begin{array}{cccc}
\Theta : &O & \to & \R^k\\
& p & \mapsto & (\theta_i(p))_{1\leq i \leq k}
\end{array}\]
is, up to restricting $O$, a diffeomorphism onto an open neighbourhood $U$ of $0$ in $\R^k$.

For $\theta = (\theta_1,\ldots, \theta_k)$, define
\[N_\theta = N_{\Theta^{-1}(\theta_1,\ldots, \theta_k)}~.\]
Then, by construction, $N_\theta$ is a GHMC AdS $4$-manifold containing a Cauchy hypersurface homeomorphic to $M^a$ which is piecewise totally geodesic, folded along the $H_i$, and such that, for all $1\leq i \leq k$, the folding angle at $H_i$ and $H_{i+k}$ is $\theta_i$. The family $N_\theta$ thus satisfies the required properties.

\end{proof}

\begin{proof}[Proof of Theorem \ref{tm:bending3d-hyp}]
  The proof is identical to the proof of Theorem \ref{tm:bending3d-ads}.
\end{proof}

Let us now recall how the families of hyperbolic ends constructed in Theorem \ref{tm:bending3d-hyp} relate to the ``bending deformations'' constructed by Johnson and Millson.

Let $M$ be a hyperbolic manifold of dimension $d$ containing a smooth totally geodesic connected separating hypersurface $H$. In \cite{Johnson1987}, Johnson and Millson construct a $1$-parameter deformation of the Fuchsian representation $\rho_0: \pi_1(M) \to \SO_\circ (d,1)$ into $\SO_\circ (d+1, 1)$:

Let $M_1$ and $M_2$ denote the two components of $M\backslash H$. By Van Kampen's theorem, one can write the fundamental group of $M$ as an amalgamated product
\[\pi_1(M)= \pi_1(M_1) *_{\pi_1(H)} \pi_1(M_2)~.\]
Now, up to conjugation, $\rho_0(\pi_1(H))$ is contained in $\SO(d-1,1)$ and is thus centralized by a rotation subgroup isomorphic to $\SO(2)$. There is thus a unique representation
\[\rho_{H,t}: \pi_1(M) \to \SO_\circ (d+1,1)\]
such that
\[{\rho_{H,t}}_{\vert \pi_1(M_1)} = {\rho_0}_{\vert \pi_1(M_1)}\]
and 
\[{\rho_{H,t}}_{\vert \pi_1(M_2)} = r_t{\rho_0}_{\vert \pi_1(M_2)}r_{-t}~,\]
with $r_t$ the rotation of angle $t$ commuting with $\SO_\circ(d-1,1)$.  This deformation has a geometric interpretation: the representation $\rho_{H,t}$ is the holonomy of the folded hyperbolic embedding structure on $M$ which is isometric and totally geodesic on $M_1$ and $M_2$ and is folded along $H$ with an angle $t$.

The same construction can be used to deform $\rho_0$ into $\SO_\circ (d,2)$. This time, the subgroup centralizing $\rho_0(\pi_1(H))$ is $\SO(1,1)$ and the deformations are holonomies of folded spacelike embedding structures on $M$.\\

Let us now return to the setting of Theorem \ref{tm:bending3d-ads}. The Gromov--Thurston $3$-manifold $M^a$ equipped with the smooth metric $g_{\frac \pi k}$ admits $k$ separating totally geodesic hypersurfaces $\hat H_i \equaldef H_i \cup H_{i+k}$, $1\leq i \leq k$. For each of these one gets a Johnson--Millson deformation $\rho_{\hat H_i, t}$ which is nothing but the holonomy of the quasifuchsian AdS manifold $N_{t \theta^i}$ where $\theta^i_j = \delta_{i,j}$. Theorem~\ref{tm:bending3d-ads} thus gives an example where Johnson--Millson's bending deformations along $k$ intersecting hypersurfaces fit into a $k$-parameter deformation family.

\subsection{Regularity of the map $p\mapsto N_p$} \label{ss:Regularity}

Until now, we have been very vague about the question of the regularity of our ``families of GHMC AdS manifolds'', i.e. about the regularity of the map $p\mapsto N_p$. Here we give a little bit more precisions on that.

Fix a Gromov--Thurston manifold of dimension $d$. Let us say that a  quasifuchsian AdS manifold of dimension $d+1$ is \emph{marked by $M^a$} if it is equipped with an isomorphism $\pi: \pi_1(M^a) \to \pi_1(N)$ induced by a spacelike embedding of $M^a$ in $N$. Via the holonomy representation, the set of quasifuchsian AdS $d+1$ manifolds marked by $M^a$ identifies with the open domain of convex-cocompact representations in $\Hom(\pi_1(M^a),\SO_\circ (d,2))/\SO_\circ(d,2)$.\footnote{While the quotient $\Hom(\pi_1(M^a),\SO_\circ (d,2))/\SO_\circ(d,2)$ might not be Hausdorff, one can prove that the subset of convex cocompact representations is Hausdorff.}
This gives a natural topology to the space of quasifuchsian AdS $d+1$ manifolds marked by $M^a$, and even an analytic structure.

Now, in the proof of Theorem \ref{tm:dim_ads}, Theorem \ref{tm:31_ads} and Theorem \ref{tm:bending3d-ads}, one associates to every spacelike de Sitter polygon $p$ in some family $\mathcal P$ (which is proven to be an analytic manifold in Section \ref{sc:polygons}) a certain spacelike AdS structure on $M^a$. Looking closely at the construction (see the proof of Lemma \ref{lm:equivalence}), one can verify that it can be defined by local charts that depend smoothly on the polygon $p$. Looking back at the  proof of Corollary \ref{cor-developing map spacelike AdS structure}, one deduces that this family of spacelike AdS structures is given by a family of pairs $(\dev_p, \rho_p)$ where $\dev_p$ is a developing map and $\rho_p$ a holonomy representation depending smoothly on $p$. The representation $\rho_p$ is the holonomy of the corresponding marked quasifuchsian AdS manifold $N_p$. In this sense we can say that the map $p\mapsto N_p$ is smooth.

Let $p$ be a point in $\mathcal P$. Denote by $\Ad_{\rho_{p}}$ the composition of $\rho_{p}: \pi_1(M^a) \to \SO_\circ (d,2)$ with the adjoint representation of $\SO_\circ(d,2)$. The cohomology group
\[\mathrm H^1 (\pi_1(M^a), \Ad_{\rho_{p}})\]
is the tangent space to the character stack $\Hom(\pi_1(M^a), \SO_\circ(d,2))/ \SO_\circ(d,2)$ at $\rho_{p}$. In particular, the derivative of the map $p\mapsto \rho_p$ defines a linear map from $T_p \mathcal P$ to $\mathrm H^1 (\pi_1(M^a), \Ad_{\rho_{p}})$.\\

Let us now specify the discussion of the previous paragraph to the situation of Theorem\ref{tm:bending3d-ads}. We take as $\mathcal P$ a neighbourhood of $p_0$ in the set of equilateral polynomials with a central symmetry, which is diffeomorphic to a small neighbourhood $U$ of $0$ in $\R^k$. We thus obtain a smooth map
\[ \begin{array}{cccc}
\mathrm R: & U & \to & \Hom(\pi_1(M^a),\SO_\circ (3,2))/ \SO_\circ (3,2)\\
& \theta & \mapsto & \rho_\theta \equaldef \rho_{p_\theta}~.
\end{array}\]
Its derivative at $0$ is a linear map
\[\mathrm d_0 \mathrm R: \R^k \to \mathrm H^1(\pi_1(M^a), \Ad_{\rho_0})~,\]
where $\rho_0$ is the Fuchsian representation of $\pi_1(M^a)$.

Noting as above $\theta^i \in \R^k$ the vector such that $\theta^i_j = \delta_{i,j}$, we have in particular that $\mathrm d_0 \mathrm R (\theta^i)$ is the derivative of Johnson--Millson's bending deformation along $\hat H_i$. The image of $\mathrm d_0 R$ is the subspace of $\mathrm H^1(\pi_1(M^a), \Ad_{\rho_0})$ spanned by these infinitesimal bending deformations.

In general, the character stack $\Hom(\pi_1(M^a),\SO_\circ (d,2))/ \SO_\circ (d,2)$ needs not be smooth at $\rho_0$, and not every vector in $H^1(\pi_1(M^a), \Ad_{\rho_0})$ need to be the derivative of an actual deformation of $\rho_0$ in $ \SO_\circ (d,2)$. At one extreme, one could imagine that $\Hom(\pi_1(M^a),\SO_\circ (d,2))/ \SO_\circ (d,2)$ is a union of $k$ curves intersecting at $\rho_0$, corresponding to the $k$ bending deformations.

However, Theorem \ref{tm:bending3d-ads} shows that it is not the case in dimension $3+1$. The existence of the map $\Phi$ shows that any linear combination of the infinitesimal bending deformations along the hypersurfaces $\hat H_i$ can be integrated into an actual deformation.

\section{Initial singularities}
\label{sc:initial}

\subsection{Dualities}

Before proving Theorems \ref{tm:initial-ads} and \ref{tm:initial-ds}, we recall a basic notion of duality (or polarity) in $\AdS^{d+1}$, or between $\HH^{d+1}$ and $\dS^{d+1}$. A good description of the duality for hyperbolic polyhedra can be found in \cite{HR}, and an extension to other constant curvature pseudo-Riemannian spaces can be found e.g. in \cite{shu}. 

\subsubsection*{The duality between $\HH^{d+1}$ and $\dS^{d+1}$}

Recall that the hyperbolic $d+1$-dimensional space can be defined as a quadric in the Minkowski space of dimension $d+2$, denoted here as $\R^{d+1,1}$. This Minkowski space is $\R^{d+2}$ equipped with the bilinear symmetric form:
$$ \langle x,y\rangle_{d+1,1} =  \sum_{i=1}^{d+1} x_iy_i-x_{d+2}y_{d+2}~. $$
The hyperbolic space can then be defined as
$$ \HH^{d+1} = \{ x\in \R^{d+1,1}~|~\langle x,x\rangle_{d+1,1}=-1~,~x_{d+2}>0 \}~, $$
equipped with the induced metric. In the same space, we can consider the de Sitter space, defined as
$$ \dS^{d+1} = \{ x\in \R^{d+1,1}~|~\langle x,x\rangle_{d+1,1}=1 \}~, $$
again with the induced metric. It is a geodesically complete Lorentzian space of constant curvature~$1$, simply connected if $d\geq 2$.

Let $x\in \HH^{d+1}$, and let $x^\perp$ be the hyperplane in $\R^{d+1,1}$ orthogonal to $x$. Since $x$ is timelike, its orthogonal $x^\perp$ is spacelike, and its intersection with $\dS^{d+1}$ is a totally geodesic, spacelike hyperplane, which we denote by $x^*$. Conversely, any spacelike hyperplane $H$ in $dS^{d+1}$ is the intersection of $\dS^{d+1}$ with a hyperplane $\bar H$ of $\R^{d+1,1}$ containing $0$. This hyperplane $\bar H$ is orthogonal to a unique unit, future-oriented timelike vector, which we denote $H^*$. This construction provides a one-to-one correspondence between points in $\HH^{d+1}$ and (un-oriented) spacelike totally geodesic hyperplanes in $\dS^{d+1}$.

Similarly, given $y\in \dS^{d+1}$, the intersection $y^\perp\cap \HH^{d+1}$ is an {\em oriented} totally geodesic hyperplane in $\HH^{d+1}$, which we denote by $y^*$. And conversely, if $H\subset \HH^{d+1}$ is any oriented totally geodesic hyperplane, then it is the intersection with $\HH^{d+1}$ of an oriented hyperplane in $\R^{d+1,1}$ containing $0$. The oriented unit normal to this hyperplane is a point in $\dS^{d+1}$, which we denote by $H^*$. 

This duality relation has several important consequences.
\begin{itemize}
\item Two oriented hyperplanes $H,H'\subset \HH^{d+1}$ intersect if and only if the dual points $H^*, H'^*$ are connected by a spacelike geodesic segment. The angle between $H$ and $H'$ is then the length of the segment connecting $H^*$ to $H'^*$.
\item The intersection angle between two hyperplanes $H, H'\subset \dS^{d+1}$ is equal to the distance between the dual points $H^*, H'^*\subset \HH^{d+1}$.
  \item For all $x\in \HH^{d+1}$, $(x^*)^*=x$, and similarly for $y\in \dS^{d+1}$.
\end{itemize}
This duality relation extends to convex polyhedra (see \cite{HR}) and to smooth, strictly convex surfaces (see e.g. \cite{shu}).

\subsubsection*{The duality between points and hyperplanes in $\AdS^{d+1}$}

The $d+1$-dimensional anti-de Sitter space $\AdS^{d+1}$ can be defined as a ``pseudo-sphere'' in the flat space $\R^{d,2}$ of signature $(d,2)$. Specifically, $\R^{d,2}$ can be defined as $\R^{d+2}$ equipped with the bilinear symmetric form
$$ \langle x,y\rangle_{d,2} = \sum_{i=1}^{d} x_iy_i - x_{d+1}y_{d+1}-x_{d+2}y_{d+2}~, $$
and
$$ \AdS^{d+1} = \{ x\in \R^{d,2}~|~ \langle x,x\rangle_{d,2} = -1\}~. $$
It is a geodesically complete Lorentzian space of constant curvature $-1$.

Let $x\in \AdS^{d+1}$. Its orthogonal $x^\perp$ is an oriented hyperplane in $\R^{d,2}$ of signature $(d,1)$, which therefore intersects $\AdS^{d+1}$ along a spacelike totally geodesic oriented hyperplane, denoted by $x^*$. As above, the same construction works, conversely, to associate to any totally geodesic spacelike oriented hyperplane a dual point.

There is an ``intrinsic'' definition of this duality: the hyperplane $x^*$ dual to a point $x$ is the totally geodesic plane composed of points at time distance $\pi/2$ from $x$ in the future.

This duality has the same properties as the duality between $\HH^{d+1}$ and $\dS^{d+1}$.

\subsection{Initial singularities of de Sitter spacetimes}

GHMC de Sitter spacetimes can be defined in the same way as GHMC anti-de Sitter spacetimes (see section \ref{subsec: globally hyperbolic spacetimes}).
We  briefly describe the duality between GHMC de Sitter spacetimes and hyperbolic ends (additional details and proofs can be found in \cite{kulkarni-pinkall,scannell}), and then outline how this duality leads to the proof of Theorem \ref{tm:initial-ds}. 

Let us give more  details about the correspondence established in \cite{kulkarni-pinkall} between flat conformal structures on a manifold $M$ with non virtually Abelian fundamental group and hyperbolic ends with pleated boundary homeomorphic to $M$ that was  mentioned in Section \ref{subsec:hyperbolic ends}. Consider first a hyperbolic end $E$ with pleated boundary homeomorphic to $M$. Then the ideal boundary  $\partial_\infty E$ of $E$ is diffeomorphic to $M$ and equipped with a flat conformal structure $c$. It is proved in \cite{kulkarni-pinkall} that (in dimension $d\geq 3$) $E$  is uniquely determined by $c$.  More specifically, the  pleated boundary $\partial_0E$ of $E$ is equipped with a stratification in ideal polyhedra of varying dimensions between $1$ and $d$, while $(M,c)$ also has a natural stratification, with each point contained in the interior of the ``convex hull'' of the boundary of a unique maximal round ball. There is a natural map from $\partial_\infty E$ to $\partial_0E$, preserving the stratification, in the sense that each strata of the stratification of $\partial_\infty E$ is sent homeomorphically to a strata of $\partial_0E$ (but many strata of $\partial_\infty E$ can be sent to the same strata of $\partial_0E$). In fact, round balls in $\partial_\infty \HH^{d+1}$ are in one-to-one correspondence with oriented hyperplanes in $\HH^{d+1}$, and the strata of $(M,c)$ are in one-to-one correspondence with the support hyperplanes of $\partial_0E$. 

A similar construction is provided by Scannell \cite[Section 4]{scannell} for  GHMC de Sitter spacetimes. Namely, if $N$ is such a spacetime, diffeomorphic to $M\times \R$, where $M$ is again a closed $d$-dimensional manifold, then its future asymptotic boundary is equipped with a flat conformal structure $c$, and this flat conformal structure again uniquely determines the  GHMC structure on $N$. If the fundamental group of $M$ is not virtually Abelian,  GHMC de Sitter structures on $N$ are therefore in one-to-one correspondence with hyperbolic ends diffeomorphic to $M\times \R$. 

The stratification of $(M,c)$ is directly related to the initial singularity of $N$, which we denote by $\partial_0N$. Round disks in $\partial_\infty \HH^{d+1}$ are in one-to-one correspondence with points in $\dS^{d+1}$, and each stratum of $(M,c)$ determines a unique point in the initial singularity of $N$. The points that are obtained in this way are exactly those where the boundary of $M$ admits a spacelike supporting hyperplane, see  \cite[Section 4]{scannell}.

Summing up the relations, each hyperbolic end structure on $M\times \R$ is determined uniquely by a flat conformal structure on $M$, which in turns determines a unique  GHMC de Sitter spacetime. Hyperbolic ends are therefore in one-to-one correspondence with  GHMC de Sitter spactimes. Moreover, each support hyperplane of $\partial_0E$ corresponds to a point of $\partial_0 N$ where it admits a spacelike support plane, and conversely. 

This correspondence is somewhat easier to visualize when the developing $\dev$ of the conformal structure is injective. Then
$$ E = (\HH^{d+1}\setminus CH(\partial_{\infty}\HH^{d+1} \setminus \dev(\tilde M)))/\rho(\pi_1M)~, $$
where $\rho:\pi_1M\to SO(d+1,1)$ is the holonomy representation of $(M,c)$, while $\tilde N$ is a {\em domain of dependence}, that is, the intersection of the half-spaces containing $\bS^d$ and bounded by a hyperplane tangent to $\bS^d$ at a point of $\Lambda_\rho$, the limit set of $\rho$. The initial singularity of $\tilde N$ is then the set of points dual to the support hyperplanes of $CH(\partial_{\infty}\HH^{d+1} \setminus \dev(\tilde M))$.

Suppose now that $E$ is a hyperbolic end such that $\partial_0E$ is folded, that is, it is the union of $2k$ $d$-dimensional totally geodesic polyhedra meeting pairwise along a $(d-1)$-dimensional face and which all share a $(d-2)$-dimensional face $S$. Let $N$ be the dual convex GHMC de Sitter spacetime. The initial singularity of $N$ is then particularly simple:
\begin{itemize}
\item each $d$-dimensional polyhedron in $\partial_0E$ corresponds to a vertex of $\partial_0N$,
\item each $(d-1)$-dimensional intersection hypersurface in $\partial_0E$ corresponds to an edge of $\partial_0N$,
  \item the ``spine'' $S$ corresponds to a $2$-dimensional face of $\partial_0N$.
\end{itemize}
Theorem \ref{tm:initial-ds} follows from Theorem \ref{tm:existence_hyp}, and of the construction used in its proof, through this correspondence.

\subsection{Initial singularities of AdS spacetimes}

A similar description applies in the anti-de Sitter setting. It is somewhat simpler because, in the AdS case, any quasifuchsian AdS spacetime is the quotient of a domain of dependence by the image of a representation into $\SO_\circ(d,2)$, as proved by Mess \cite{mess}.

A quasifuchsian AdS spacetime $N$ contains a smallest non-empty closed convex subset, its convex core $C(N)$. The past of the future boundary $\partial_+C(N)$ of the convex core is the union of the timelike geodesic segments of length $\pi/2$ orthogonal to support planes of $C(N)$ along $\partial_+C(N)$ towards the past. Similarly the future of the past boundary $\partial_-C(N)$ is the union of timelike segments of length $\pi/2$ orthogonal to support planes of $C(N)$ along $\partial_-C(N)$ towards the future.

Moreover, the universal cover $\tilde N$ of $N$ is isometric to a convex domain in $\AdS^{d+1}$ which is a domain of dependence, that is, the set of points $x$ in $\AdS^{d+1}$ such that all timelike geodesics through $x$ intersect the lift of any Cauchy hypersurface in $N$. In this picture, $\tilde C(N)$ is the convex hull of the asymptotic boundary of the lift to $\AdS^{d+1}$ of any Cauchy surface.

As a consequence, $\partial_+C(N)$ is dual to the initial singularity of $N$, while $\partial_- C(N)$ is dual to the final singularity of $N$. As in the de Sitter case, the description is simpler when $\partial_+C(N)$ is folded, since in this case the intial singularity of $N$ is a 2-dimensional complex with vertices corresponding to the maximal dimension faces of $\partial_+C(N)$, edges corresponding to the hypersurfaces along which the maximal dimension faces meet, and one totally geodesic 2-dimensional face dual to the codimension 2 ``spine''. 

Theorem \ref{tm:initial-ads} follows from Theorem \ref{tm:existence_ads}, and of the construction used in its proof, through this correspondence.


\section{Compact Clifford--Klein forms}
\label{sc:clifford}

In this section we explain why quasifuchsian AdS manifolds of dimension $2d+1$ provide compact quotients of the pseudo-Riemannian symmetric space $\O(2d,2)/\U(d,1)$. We prove that these compact quotients admit a smooth fibration over a manifold of dimension $2d$, with fibers isomorphic to the compact subspace $\O(2d)/\U(d)$.

\subsection{From GHC manifolds to compact quotients}

Benoist \cite{Benoist96} and Kobayashi \cite{Kobayashi92} independently gave a necessary and sufficient criterion for a discrete subgroup $\Gamma$ of a semisimple Lie group $G$ to act properly discontinuously on a reductive homogeneous space $G/H$, in terms of the Cartan projections of $\Gamma$ and $H$. This criterion bares a strong resemblance with the \emph{Anosov property} of the group $\Gamma$ as reformulated by Gu\'eritaud--Guichard--Kassel--Wienhard \cite{GGKW17} and Kapovich--Leeb--Porti \cite{KLP}. As an application, the first group of authors remarked the following:

\begin{theorem} \label{t:GHMC->CompactQuotient}
Let $\Gamma \backslash \Omega$ be an $\AdS$ quasifuchsian spacetime of dimension $2d+1$. Then the group $\Gamma$ acts properly discontinuously and cocompactly on the pseudo-Riemannian symmetric space $\O(2d,2)/\U(d,1)$.
\end{theorem}

\begin{proof}[Outline of the proof]
The group $\Gamma$ is a projective Anosov subgroup of $\O(2d,2)$ (see \cite{barbot-merigot}). In this situation, it implies that $\Gamma$ satisfies the Benoist--Kobayashi criterion and thus acts properly discontinuously on $\O(2d,2)/\U(d,1)$. The cocompactness comes from a cohomological dimension argument: the space $\O(2d,2)/\U(d,1)$ has dimension $d(d+1)$ and deformation retracts on the compact symmetric space $\O(2d)/\U(d)$, of dimension $d(d-1)$. By a classical application of the Leray--Serre spectral sequence, it follows that a group acting properly discontinuously on $\O(2d,2)/\U(d,1)$ has virtual cohomological dimension at most $2d$, with equality if and only if its action is cocompact. On the other hand, $\Gamma$ acts properly discontinuously and cocompactly on a complete spacelike hypersurface in $\AdS^{2d+1}$, which is diffeomorphic to a disc of dimension $2d$. Hence $\Gamma$ has cohomological dimension $2d$ and thus acts cocompactly on $\O(2d,2)/\U(d,1)$.
\end{proof}

Theorem \ref{t:GHMC->CompactQuotient} provides a motivation to better understand the relationship between the $\AdS$ quasifuchsian manifold $\Gamma \backslash \Omega$ and the corresponding compact quotient $\Gamma \backslash \O(2d,2)/\U(d,1)$. In the following subsections, we explain that $\Gamma \backslash \O(2d,2)/\U(d,1)$ can be seen as a fiber bundle with ``geometric fibers'' over a strongly convex Cauchy hypersurface of $\Gamma \backslash \Omega$.

\subsection{Geodesic Killing fields}

We first start by understanding the relationship between $\AdS^{2d+1}$ and the space $\mathrm O(2d,2)/\U(d,1)$.\footnote{Note that, since $\SO_\circ(2d,2)$ has finite index in $\mathrm O(2d,2)$, acting properly discontinuously and cocompactly on $\mathrm O(2d,2)/\U(d,1)$ and on $\SO_\circ(2d,2)/\U(d,1)$ are equivalent.} For this, let us recall the definition and basic properties of a \emph{geodesic Killing field}.

Let $(M,g)$ be a smooth connected pseudo-Riemannian manifold and denote by $\nabla$ its Levi--Civita connection. Recall that a vector field $X$ on $M$ is a \emph{Killing field} if its flow preserves the metric $g$. Killing fields are characterized by the following property:

\begin{prop}{\normalfont(see for instance \cite[Proposition 9.25]{oneill1983semiriemannian})}
A vector field $X$ on $(M,g)$ of class $\mathcal C^1$ is a Killing field if and only if the tensor $\nabla X \in \End(TM)$ is antisymmetric with respect to $g$, i.e.
\[g(Y, \nabla_Z X) = - g(X, \nabla_Z Y)\]
for all vector fields $Y$ and $Z$.
\end{prop}

A vector field $X$ on $(M,g)$ is \emph{geodesic} if the orbits of its flow are geodesics. Equivalently, $X$ is geodesic if it satisfies
\[\nabla_X X = 0~.\]
For Killing fields, we have the following characterization:

\begin{prop} \label{p:GeodesicKillingFieldsConstantNorm}
Let $X$ be a Killing field on $(M,g)$. Then $X$ is geodesic if an only if $g(X,X)$ is constant on $M$.
\end{prop}

\begin{proof}
For any vector field $Y$, we have
\begin{eqnarray*}
d_Yg(X,X) & = & 2 g(X, \nabla_YX) \quad \textrm{since $\nabla$ preserves $g$}\\
& = & -2 g(Y,\nabla_X X) \quad \textrm{since $X$ is Killing.}
\end{eqnarray*}
Thus $g(X,X)$ is constant if and only if $\nabla_X  X= 0$.
\end{proof}

Let us now turn to the case where $(M,g)$ is the anti-de Sitter space of dimension $2d+1$. In that case, there is a natural isomorphism
\[ u \mapsto X^u\]
between the Lie algebra $\so(2d,2)$ of $\mathrm O(2d,2)$ and the Lie algebra of Killing fields on $\AdS^{2d+1}$. In concrete terms, we see $\so(2d,2)$ as a Lie subalgebra of the space of square matrices of size $2d+2$. Every $u\in \so(2d,2)$ defines a linear vector field $\hat X^u$ on $\R^{2d,2}$, defined by:
\[\hat X^u(x) = u \cdot x~.\]
This vector field is tangent to the quadric
\[\{\mathbf q(x) \equaldef x_1^2+\ldots + x_{2d}^2 - x_{2d+1}^2 - x_{2d+2}^2 = -1\}\equaldef \AdS^{2d+1}\]
and thus restricts to a vector field $X^u$ on $\AdS^{2d+1}$

We say that a Killing field $X$ on $\AdS^{2d+1}$ is \emph{timelike unitary} if it satisfies $g(X,X) = -1$ (it is then necessarily geodesic by Proposition \ref{p:GeodesicKillingFieldsConstantNorm}). The main purpose of this subsection is the following description of timelike unitary Killing fields. 
 
\begin{lemma}
Let $u$ be a vector in the Lie algebra of $\so(2d,2)$. Then the corresponding Killing field $X^u$ is timelike unitary if and only if $u^2 = -\Id$. The space of timelike unitary Killing fields is therefore equivariantly isomorphic to the homogeneous space $\mathrm O(2d,2)/\U(d,1)$.
\end{lemma}

\begin{proof}
Let $\hat \nabla$ denote the standard flat connection on $\R^{2d,2}$, that is, the Levi--Civita connection of the flat pseudo-Riemannian metric $\mathbf q$. For any $u,v\in \so(2d,2)$, we have
\begin{eqnarray*}
	\hat \nabla_{\hat X^u} \hat X^v(x) &=& \left. \frac{\mathrm d}{\mathrm d t}\right |_{t=0}\hat X^v(x+ t u\cdot x)\\
	&= & \left. \frac{\mathrm d}{\mathrm d t}\right |_{t=0} v \cdot x + t vu \cdot x\\
	&=& vu \cdot x\\
	&=& \hat X^{vu}(x)~.
\end{eqnarray*}

Now, since $\AdS^{2d+1}$ is a submanifold equipped with the restricted metric, its Levi--Civita connexion $\nabla$ is the orthogonal projection of $\hat \nabla$ to $T \AdS^{2d+1}$. Since $T_x \AdS^{2d+1} = x^\perp$, we get that$ \nabla_{X^u} X^u(x) = 0$ if and only if $\hat \nabla_{\hat X^u} \hat X^u(x) = \hat X^{u^2}(x)= u^2 \cdot x$ is colinear to $x$.

Since $u$ is linear, $u^2 \cdot x$ is colinear to $x$ for every $x$ if and only if $u^2 \in \R \Id$, and we conclude that the Killing field $X^u$ is a geodesic if and only if 
\[u^2 = \lambda \Id\]
for some $\lambda\in \R$.

If $u^2 = \lambda \Id$, then at every point $x\in \AdS^{2d+1}$ we have
\begin{eqnarray*}
 g_\AdS (X^u(x), X^u(x)) & = & \langle u\cdot x,u\cdot x\rangle_{2d,2} \\
&=& - \langle u^2\cdot x, x\rangle_{2d,2} \\
&=& - \lambda \langle x,x\rangle_{2d,2} \\
&=&\lambda~.
\end{eqnarray*}

Hence $X^u$ is timelike and unitary if and only if $\lambda = -1$, i.e. $u^2 = -\Id$.\\

Each such $u$ defines a complex structure on $\R^{2d,2}$ compatible with the metric $\mathbf q$. There is thus a unique pseudo-Hermitian form $\mathbf h_u$ on $(\R^{2d,2}, u)$ such that $\Re(\mathbf h_u) = \mathbf q$. This Hermitian form has complex signature $(d,1)$, and the subgroup of $\mathrm O(2d,2)$ commuting with $u$ is the the group $\U(d,1)$.

Finally, given $u,v\in \so(2d,2)$ with $u^2 = v^2 = -\Id$, the pseudo-Hermitian spaces $(\R^{2d+2}, u, \mathbf h_u)$ and $(\R^{2d+2}, v, \mathbf h_v)$ are isomorphic (since they have the same dimension and signature). Hence there exists $g\in \GL(2d+2,\R)$ such that $gug^{-1} = v$ and $g^* \mathbf h_v = \mathbf h_u$. In particular, $g^* \mathbf q= \mathbf q$, hence $g\in \O(2d,2)$.

In conclusion, $\mathrm O(2d,2)$ acts transitively on the space of timelike unitary Killing fields, and the centralizer of such a Killing field is isomorphic to $\U(d,1)$. The space of timelike unitary Killing fields is thus isomorphic to the homogeneous space $\mathrm O(2d,2)/\U(d,1)$.
\end{proof}

\begin{remark}
Elaborating on the above proof, one could give a complete classification of geodesic Killing fields on $\AdS^{d+1}$:
\begin{itemize}
\item Lightlike Killing fields are given by elements $u\in \so(d,2)$ satisfying $u^2 = 0$ and exist for all $d\geq 1$,
\item Timelike geodesic Killing fields are given by elements $u\in \so(d,2)$ satisfying $u^2 = \lambda \Id$, $\lambda < 0$ and only exist for even $d$,
\item Spacelike geodesic Killing fields are given by elements $u\in \so(d,2)$ satisfying $u^2 = \lambda \Id$, $\lambda > 0$ and only exist for $d=2$.
\end{itemize}
\end{remark}

\subsection{Killing fields orthogonal to a strongly convex hypersurface}

Let $\mathcal H$ be a complete spacelike hypersurface in $\AdS^{2d+1}$ and $N$ its future-pointing unit normal. Let $g$ denote both the metric of $\AdS^{2d+1}$ and its restriction to $\mathcal H$. Recall that the \emph{second fundamental form} of $\mathcal H$ is given by
\[\secondFF(\cdot, \cdot) = g(\nabla_\cdot N, \cdot)~.\]

Recall from Definition \ref{df:unif-convex} that $\mathcal H$ is uniformly strongly (past) convex if there exists a constant $c>0$ such that $\secondFF + c g$ is negative definite.

\begin{lemma} \label{l:Orthogonal Killing field hypersurface}
Let $H$ be a uniformly strongly convex complete spacelike hypersurface in $\AdS^{2d+1}$ and $X$ a unitary timelike geodesic Killing field. Then $X$ is orthogonal to $\mathcal H$ at exactly one point.
\end{lemma}

\begin{proof}
Up to multiplying $X$ by $-1$, we can assume that $X$ is future pointing, so that $g(X,N) \leq -1$ with equality exactly where $X$ is orthogonal to $\mathcal H$.

Let us decompose $X$ along $\mathcal H$ as
\[X = \bar X + f N\]
where $\bar X$ is tangent to $\mathcal H$ and $f = - g(X,N) \geq 1$. Since $X$ is unitary, we have
\[g(\bar X, \bar X) = f^2 -1~.\]

Our goal is to prove that $f:\mathcal H\to [1,+\infty)$ achieves the value $1$ at a unique point. It will follow from the following three facts:
\begin{itemize}
\item[(a)] the function $f$ is proper,
\item[(b)] if $x$ is a critical point of $f$, then $f(x)= 1$,
\item[(c)] the points where $f=1$ are isolated.
\end{itemize}

We will then conclude by looking at the gradient flow of $f$.

\begin{itemize}
\item Proof of (a):\quad Fix a point $x_0 \in \mathcal H$, let $x$ be a point at distance $T$ from $x_0$ (for the restricted metric $g$) and let $\gamma: [0,T] \to \mathcal H$ be a unit speed geodesic from $x_0$ to $x$. The equation of geodesics on $\mathcal H$ can be written as:
\[\nabla_{\dot \gamma} \dot \gamma = \secondFF(\dot \gamma, \dot \gamma) N~,\]
where $\nabla$ is the ambient Levi--Civita connection of $\AdS^{2d+1}$. 

Consider the function 
\[\begin{array}{rccc}
h: & [0,T] & \to & \R \\
& t & \mapsto & g_{\gamma(t)}(\dot \gamma, X) = g_{\gamma(t)}(\dot \gamma, \bar X)~.
\end{array}\]
By the Cauchy--Schwarz inequality on $T\mathcal H$, we have
\begin{equation} \label{eq:h bounded by f}
h(t)^2 \leq g_{\gamma(t)}(\bar X,\bar X) = f^2(\gamma(t)) - 1~,
\end{equation}
which we can also write
\begin{equation} \label{eq:f bounded by h}
f(\gamma(t))\geq \sqrt{h(t)^2 +1}~,
\end{equation}

Deriving $h$ gives
\begin{eqnarray*}
h'(t) & = & g(\nabla_{\dot \gamma} \dot \gamma ,X) + g(\dot \gamma, \nabla_{\dot \gamma} X) \\
&=& g(\nabla_{\dot \gamma} \dot \gamma ,X) \quad \textrm{since $X$ is a Killing field}\\
&=& \secondFF(\dot \gamma, \dot \gamma) g(N,X) \\
&=& - f(\gamma(t)) \secondFF(\dot \gamma, \dot \gamma)\\
&\geq & c f(\gamma(t)) \quad \textrm{by uniform strict convexity of $\mathcal H$}\\
&\geq & c \sqrt{h(t)^2+1}~.
\end{eqnarray*}

We deduce that $h(t)$ is greater or equal to the solution of the ordinary differential equation
\[y'= c\sqrt{y^2+1}\]
whit initial condition $y(0)= h(0)$, i.e.
\[h(t) \geq \sinh\left(ct+ \sinh^{-1}(h(0)\right)~.\]
Using \eqref{eq:f bounded by h} and \eqref{eq:h bounded by f}, we conclude that
\[f(x) = f(\gamma(T)) \geq \cosh(cT-c')~,\]
where
\[c'= \sinh^{-1}(\sqrt{f(\gamma(0))^2 -1})~.\]
This shows that $f$ is proper.\\

\item Proof of (b): \quad let us compute the derivative of $f = - g(X,N)$ along $\bar X$. We have
\[df(\bar X) = - g(\nabla_{\bar X} X, N) - g(X, \nabla_{\bar X} N)~.\]

On one side, we have
\[g(\nabla_{\bar X} X, N) = g(\nabla_X X, N) - f g(\nabla_N X, N) = 0\]
since $X$ is Killing and geodesic.

On the other side, we have
\[g(X, \nabla_{\bar X} N) = g(\bar X, \nabla_{\bar X} N) = \secondFF(\bar X, \bar X)\]

since $d_{\bar X}g(N,N)=0$.
\[df(\bar X) = -\secondFF(\bar X, \bar X) \geq  c g(\bar X, \bar X)~.\]
At a critical point of $f$, we thus have $\bar X= 0$, i.e. $X$ is orthogonal to $\mathcal H$ and $f = 1$.\\

\item Proof of (c):\quad Let $\bar \nabla$ denote the Levi--Civita connection of the induced metric on $\mathcal H$. We have $\bar{\nabla} \bar X = \pi(\nabla \bar X)$ where $\pi$ denotes the orthogonal projection to $T\mathcal H$. For every vector $Y$ tangent to $\mathcal H$, we have
\begin{eqnarray*}
g_x(\bar{\nabla}_Y \bar X, Y) & = & g_x(\nabla_Y \bar X, Y) \\
&=&  g_x(\nabla_Y X, Y) - f g(\nabla_Y N, Y) - \mathrm d f(Y) g(N,Y)~.
\end{eqnarray*}
The first and third term vanish since $X$ is a Killing field and $Y$ is tangent to $\mathcal H$. We conclude that \[g(\bar \nabla_Y \bar X, Y) = -f \secondFF(Y,Y) \geq c g(Y,Y)~.\]

Hence $g(\bar{\nabla}_\cdot \bar X, \cdot)$ is symmetric and positive definite, which implies that $\bar \nabla \bar X \in \End(T\mathcal H)$ is invertible at every point. In particular, the zeros of $\bar X$, which are exactly the points where where $f = 1$, are isolated. 

\item Conclusion of the proof: \quad Consider the gradient flow of $f$, i.e. the flow of the vector field $-\Grad_g f$. Since $f$ is proper, every trajectory of the flow converges to a critical point of $f$. Since each critical point is a minimum and is isolated, its basin of attraction is non-empty and open. Since $\mathcal H$ is connected and decomposes as the disjoint union of these basins of attraction, there is exactly one such critical point, at which $f= 1$. This is the unique point where $X$ is orthogonal to $\mathcal H$.\\
\end{itemize}
\end{proof}

\subsection{Fiber bundle over convex Cauchy hypersurfaces}

Let now $\Gamma \backslash \Omega$ be an $\AdS$-quasifuchsian manifold of dimension $2d+1$. Let $\mathcal H$ be a smooth strongly convex Cauchy hypersurface in $\Omega$. Then $\mathcal H$ is the quotient of a $\Gamma$-invariant uniformy strongly convex hypersurface $\tilde{\mathcal H}$ in $\AdS^{2d+1}$. 

Recall that the homogeneous space $\mathrm O(2d,2)/\U(d,1)$ identifies with the space $\mathrm{Kill}_{-1}(\AdS^{2d+1})$ of timelike unitary Killing fields in $\AdS^{2d+1}$. To shorten notations, we will write
\begin{itemize}
\item $G= \mathrm O(2d,2)$,
\item $H= \U(d,1)$,
\item $K = \mathrm O(2d)\times \mathrm O(2)$,
\item $L = H\cap K = \U(d)\times \U(1)$.
\end{itemize}

Define
\[\tilde M = \{(x,X)\in \tilde \cH \times \mathrm{Kill}_{-1}(\AdS^{2d+1})\mid X \textrm { orthogonal to $\tilde \cH$ at $x$}\}~,\]
And let $p_1$ and $p_2$ denote respectively the projections from $\tilde M$ to $\tilde \cH$ and $\mathrm{Kill}_{-1}(\AdS^{2d+1})$. The projections $p_1$ and $p_2$ are clearly both $\Gamma$-equivariant.

\begin{prop} \label{p:fibers of p1}
Let $x$ be a point in $\tilde \cH$. Then $p_2(p_1^{-1}(x)) = g K/L$ for some $g\in G$.
\end{prop}

\begin{proof}
Since $G$ acts transitively on pairs $(x,H)$ consisting of a point $x\in \AdS^{2d+1}$ and a spacelike hyperplane $H\in T_x \AdS^{2d+1}$, it is enough to prove that the space of timelike unitary Killing fields orthogonal at $x_0 = (0,\ldots, 0,1)$ to the hyperplane $H_0 = \{(x_1,\ldots, x_d, 0,0)\}$ identifies with $K/L$ inside $G/H$.

Let $X^u$ be a timelike unitary Killing field orthogonal to $H_0$ at $x_0$, given by a matrix $u\in \so(2d,2)$ satisfying $u^2 = -\Id$ and $u(x_0) \in H_0^\perp$. Then $u$ preserves $H_0^\perp$ and $H_0$. It is thus conjugated by an element of $K$ to the image of the standard complex structure on $\R^{2d,2}$, centralized by $H$. Thus $X^u$ lies in the $K$-orbit of the basepoint of $G/H$. The converse is similar.
\end{proof}

We can now conclude the proof of the following theorem.

\begin{theorem}
Let $\Gamma \backslash \Omega$ be an $\AdS$ quasifuchsian manifold of dimension $2d+1$ and let $\mathcal H$ be a smooth strongly convex Cauchy hypersurface in $\Gamma \backslash \Omega$. Then there exists a $\Gamma$-equivariant fibration from $G/H$ to $\tilde{\mathcal H}$, the fibers of which are translates of $K/L$.

In particular, $\Gamma$ acts freely, properly discontinuously and cocompactly 
on $G/H$ and the quotient manifold $\Gamma \backslash G/H$ is a smooth fiber bundle over $\mathcal H$ with fibers diffeomorphic to $K/L$.
\end{theorem}

\begin{proof}
By Lemma \ref{l:Orthogonal Killing field hypersurface}, every timelike unitary geodesic Killing field is orthogonal to $\tilde{H}$ at a single point. It follows that $p_2: \tilde M \to G/H$ is a $\Gamma$-equivariant homeomorphism. Set
\[p \equaldef p_1\circ p_2^{-1}: G/H \to \tilde H~.\]
Then $p$ is $\Gamma$-equivariant.

Let $C$ be a closed ball in $\tilde{\mathcal H}$. Choose continuously, for every $x\in C$, an element $g_x \in G$ such that $g\cdot x_0 = x$ and $g \cdot H_0 = T_x \tilde{\mathcal H}$ (with the notations of Proposition \ref{p:fibers of p1}). By Proposition \ref{p:fibers of p1}, we have that \[p^{-1}(C) = \bigsqcup_{x\in C}  g_x K/L~.\]
This shows that $p$ is a topological fibration whose fibers are translates of the compact subspace $K/L$. In particular, $p$ is proper.

Now, since $\Gamma$ acts freely, properly discontinuously and cocompactly on $\tilde{\mathcal H}$, we deduce that $\Gamma$ acts freely, properly discontinuously, and cocompactly on $G/H$. The equivariant fibration $p$ then factors to a fibration $\Gamma \backslash G/H \to \mathcal H$ with fibers homeomorphic to $K/L$.
\end{proof}

\begin{remark}
Though we didn't discuss the regularity of $p_1$ and $p_2$, a little extra care would easily show that $p_1 \circ p_2^{-1}$ is a smooth submersion as soon as $\tilde H$ is smooth.
\end{remark}

\begin{remark}
In particular we gave a proof that $\Gamma$ acts properly discontinuously on $G/H$ which is independent of that of Gu\'eritaud--Guichard--Kassel--Wienhard, based on the intrinsic geometry of $\AdS$ quasifuchsian manifolds.
\end{remark}



\bibliographystyle{plain}
\bibliography{grothu}

\end{document}